\documentclass[pdflatex,sn-mathphys-num]{sn-jnl}
\usepackage{amsmath}
\usepackage{amssymb}
\usepackage{amsthm}
\usepackage{newtx}
\hypersetup{hidelinks}

\newcommand{\bbC}{\mathbb{C}}
\newcommand{\bbF}{\mathbb{F}}
\newcommand{\bbH}{\mathbb{H}}
\newcommand{\bbR}{\mathbb{R}}
\newcommand{\bbT}{\mathbb{T}}
\newcommand{\bbZ}{\mathbb{Z}}

\newcommand{\bfA}{\boldsymbol{A}}
\newcommand{\bff}{\boldsymbol{f}}
\newcommand{\bfF}{\boldsymbol{F}}
\newcommand{\bfG}{\boldsymbol{G}}
\newcommand{\bfL}{\boldsymbol{L}}
\newcommand{\bfP}{\boldsymbol{P}}
\newcommand{\bfQ}{\boldsymbol{Q}}
\newcommand{\bfu}{\boldsymbol{u}}
\newcommand{\bfU}{\boldsymbol{U}}
\newcommand{\bfv}{\boldsymbol{v}}
\newcommand{\bfV}{\boldsymbol{V}}
\newcommand{\bfx}{\boldsymbol{x}}
\newcommand{\bfX}{\boldsymbol{X}}
\newcommand{\bfy}{\boldsymbol{y}}

\newcommand{\bfzero}{\boldsymbol{0}}

\newcommand{\bfgamma}{\boldsymbol{\gamma}}
\newcommand{\bfGamma}{\boldsymbol{\varGamma}}
\newcommand{\bfdelta}{\boldsymbol{\delta}}
\newcommand{\bfeta}{\boldsymbol{\eta}}
\newcommand{\bftheta}{\boldsymbol{\theta}}
\newcommand{\bfLambda}{\boldsymbol{\varLambda}}
\newcommand{\bfPi}{\boldsymbol{\varPi}}
\newcommand{\bfphi}{\boldsymbol{\varphi}}
\newcommand{\bfPhi}{\boldsymbol{\varPhi}}
\newcommand{\bfpsi}{\boldsymbol{\psi}}
\newcommand{\bfPsi}{\boldsymbol{\varPsi}}
\newcommand{\bfchi}{\boldsymbol{\chi}}

\newcommand{\brI}{\mathbf{I}}
\newcommand{\brJ}{\mathbf{J}}
\newcommand{\brT}{\mathbf{T}}

\newcommand{\calA}{\mathcal{A}}
\newcommand{\calC}{\mathcal{C}}
\newcommand{\calD}{\mathcal{D}}
\newcommand{\calE}{\mathcal{E}}
\newcommand{\calF}{\mathcal{F}}
\newcommand{\calG}{\mathcal{G}}
\newcommand{\calH}{\mathcal{H}}
\newcommand{\calK}{\mathcal{K}}
\newcommand{\calM}{\mathcal{M}}
\newcommand{\calN}{\mathcal{N}}
\newcommand{\calR}{\mathcal{R}}
\newcommand{\calS}{\mathcal{S}}
\newcommand{\calT}{\mathcal{T}}
\newcommand{\calX}{\mathcal{X}}

\newcommand{\rmc}{\mathrm{c}}
\newcommand{\rmd}{\mathrm{d}}
\newcommand{\rmi}{\mathrm{i}}
\newcommand{\rmj}{\mathrm{j}}
\newcommand{\rms}{\mathrm{s}}

\newcommand{\im}{\operatorname{im}}
\newcommand{\tr}{\operatorname{tr}}
\newcommand{\Tr}{\operatorname{Tr}}
\newcommand{\Fro}{{\operatorname{Fro}}}
\newcommand{\rank}{\operatorname{rank}}
\newcommand{\real}{\operatorname{Re}}
\newcommand{\imag}{\operatorname{Im}}
\newcommand{\Span}{\operatorname{span}}
\newcommand{\diag}{\operatorname{diag}}
\newcommand{\Sym}{\operatorname{Sym}}
\newcommand{\dsample}{\operatorname{\downarrow}}
\newcommand{\usample}{\operatorname{\uparrow}}

\newcommand{\ETF}{\operatorname{ETF}}
\newcommand{\TFF}{\operatorname{TFF}}
\newcommand{\UNTF}{\operatorname{UNTF}}
\newcommand{\SUNTF}{\operatorname{SUNTF}}
\newcommand{\ECTFF}{\operatorname{ECTFF}}
\newcommand{\EITFF}{\operatorname{EITFF}}

\newcommand{\phm}{\phantom{-}}
\newcommand{\medoplus}{\mathop{\textstyle\bigoplus}\limits}

\newcommand{\abs}[1]{|{#1}|}

\newcommand{\biggabs}[1]{\biggl|{#1}\biggr|}

\newcommand{\bigparen}[1]{\bigl({#1}\bigr)}
\newcommand{\Bigparen}[1]{\Bigl({#1}\Bigr)}
\newcommand{\biggparen}[1]{\biggl({#1}\biggr)}

\newcommand{\set}[1]{\{{#1}\}}

\newcommand{\norm}[1]{\|{#1}\|}

\newcommand{\Bignorm}[1]{\Bigl\|{#1}\Bigr\|}
\newcommand{\biggnorm}[1]{\biggl\|{#1}\biggr\|}

\newcommand{\ip}[2]{\langle{#1},{#2}\rangle}
\newcommand{\bigip}[2]{\bigl\langle{#1},{#2}\bigr\rangle}

\newcommand{\alphi}{\renewcommand{\labelenumi}{(\alph{enumi})}}

\newcommand{\romani}{\renewcommand{\labelenumi}{(\roman{enumi})}}

\theoremstyle{definition}
\newtheorem{theorem}{Theorem}[section]
\newtheorem{lemma}[theorem]{Lemma}
\newtheorem{corollary}[theorem]{Corollary}

\newtheorem{definition}[theorem]{Definition}
\newtheorem{remark}[theorem]{Remark}
\newtheorem{example}[theorem]{Example}
\newtheorem{problem}[theorem]{Problem}

\begin{document}

\title[Incoherent filter banks]{Incoherent filter banks}
\author[1]{\fnm{William J.} \sur{Brinkley}}
\author[1]{\fnm{Matthew} \sur{Fickus}}
\author[1]{\fnm{John} \sur{Jasper}}
\author[2]{\fnm{Dustin G.} \sur{Mixon}}

\affil[1]{
\orgdiv{Department of Mathematics \& Statistics},
\orgname{Air Force Institute of Technology},
\orgaddress{\city{Wright-Patterson Air Force Base}, \postcode{45433}, \state{Ohio}, \country{USA}}}

\affil[2]{
\orgdiv{Department of Mathematics and Translational Data Analytics Institute},\newline
\orgname{The Ohio State University},
\orgaddress{\city{Columbus}, \postcode{43210}, \state{Ohio}, \country{USA}}}

\abstract{Equi-isoclinic (EI) and equichordal (EC) tight fusion frames (TFFs) are types of minimally coherent (maximally orthogonal) subspace packings that arise in several applications, such as compressed sensing. They seem challenging to construct, and the problem of characterizing their existence remains mostly open. In this paper, we characterize those special complex EITFFs and ECTFFs that arise from isometries whose cross-Gram matrices commute, as is the case with filter banks. In particular, we show that every filter bank EITFF is equivalent to a direct sum of equiangular tight frames (ETFs). More generally, we show that every filter bank ECTFF is analogously equivalent to a novel generalization of ETFs that we dub stratified unit norm tight frames (SUNTFs). We further establish a general theory of SUNTFs, providing various explicit constructions of them, and some necessary conditions on their existence. All of our techniques are explicit, yielding direct constructions of incoherent filter banks.}



\maketitle

\section{Introduction}

Let $\bbF$ be $\bbR$ or $\bbC$.
In various applications, such as compressed sensing~\cite{Tropp04,EldarKB10},
waveform design~\cite{StrohmerH03} and quantum tomography~\cite{Zauner99},
the following problem arises:
given integers $N\geq 2$ and $D\geq R\geq 1$ such that $NR>D$,
design $N$ subspaces \smash{$(\calF_n)_{n=1}^{N}$} of $\bbF^D$,
each of dimension $R$, to be minimally coherent (nonorthogonal).
More precisely, design $N$ matrices $(\bfPhi_n)_{n=1}^{N}$ in $\bbF^{D\times R}$,
with the columns of each $\bfPhi_n$ forming an orthonormal basis for $\calF_n$,
that minimize either the \textit{spectral (block) coherence}~\cite{EldarKB10} or \textit{chordal coherence}, defined as
\begin{equation}
\label{eq.spectral and chordal coherence}
\mu_{\rms}:=\max_{n_1\neq n_2}\norm{\bfPhi_{n_1}^*\bfPhi_{n_2}^{}}_2,
\quad
\mu_{\rmc}:=\max_{n_1\neq n_2} \tfrac1{\sqrt{R}}\norm{\bfPhi_{n_1}^*\bfPhi_{n_2}^{}}_{\Fro},
\end{equation}
respectively.
Equivalently, design such $(\calF_n)_{n=1}^{N}$ to be maximally pairwise separated with respect to either the \textit{spectral distance}~\cite{DhillonHST08} or \textit{chordal distance}~\cite{ConwayHS96}, namely
\begin{equation}
\label{eq.spectral and chordal distance}
\rmd_{\rms}(\calF_{n_1},\calF_{n_2}):=(1-\norm{\bfPhi_{n_1}^*\bfPhi_{n_2}^{}}_2^2)^{\frac12},
\quad
\rmd_{\rmc}(\calF_{n_1},\calF_{n_2}):=(R-\norm{\bfPhi_{n_1}^*\bfPhi_{n_2}^{}}_\Fro^2)^{\frac12},
\end{equation}
respectively.
That is, optimally pack $N$ members in the Grassmannian $\operatorname{Gr}_R(\bbF^{D})$ equipped with one of these two distances.
Here, \smash{$\norm{\bfPhi_{n_1}^*\bfPhi_{n_2}^{}}_2^2$} is the maximum (worst) value of \smash{$\abs{\ip{\bfy_{n_1}}{\bfy_{n_2}}}^2$} over all unit vectors $\bfy_{n_1}\in\calF_{n_1}$ and $\bfy_{n_2}\in\calF_{n_2}$,
whereas \smash{$\frac1{R^2}\norm{\bfPhi_{n_1}^*\bfPhi_{n_2}^{}}_{\Fro}^2$} is the mean (average) such value when such $\bfy_{n_1}$ and $\bfy_{n_2}$ are distributed independently and uniformly.
Small spectral coherence is particularly important in compressed sensing~\cite{EldarKB10}:
block orthogonal matching pursuit, for example, efficiently recovers $N$ vectors $(\bfx_n)_{n=1}^{N}$ in $\bbF^R$ from $\sum_{n=1}^{N}\bfPhi_n\bfx_n$ provided
\smash{$\#\set{n : \bfx_n\neq\bfzero}<\frac12(\mu_{\rms}^{-1}+1)$}.
Such capacity is limited by~\textit{Welch bounds}~\cite{Welch74,StrohmerH03,CalderbankTX15}:
\begin{equation}
\label{eq.Welch bounds}
\mu_{\rms}^2
\geq\mu_{\rmc}^2
\geq\tfrac{NR-D}{D(N-1)}.
\end{equation}
In general,
a finite sequence $(\calF_n)_{n=1}^{N}$ of $R$-dimensional subspaces of $\bbF^D$
is \textit{equi-isoclinic} (EI) if there exists $\varLambda\geq0$ such that $\bfPhi_{n_1}^*\bfPhi_{n_2}^{}\bfPhi_{n_2}^*\bfPhi_{n_1}^{}=\varLambda\brI$ for any distinct $n_1,n_2$,
is \textit{equichordal} (EC) if, more weakly, there exists $\varLambda\geq0$ such that \smash{$\tfrac1R\norm{\bfPhi_{n_1}^*\bfPhi_{n_2}^{}}_{\Fro}^2=\varLambda$} for any distinct $n_1,n_2$,
and is a \textit{tight fusion frame} (TFF) for $\bbF^D$ if $\sum_{n=1}^N\bfPhi_n^{}\bfPhi_n^*=A\brI$ for some $A>0$.
It is well-known~\cite{LemmensS73b,ConwayHS96,StrohmerH03,DhillonHST08,CalderbankTX15} that $(\calF_n)_{n=1}^{N}$ achieves equality in the right-hand inequality of~\eqref{eq.Welch bounds} if and only if it is an \textit{equichordal tight fusion frame} (ECTFF) for $\bbF^D$,
and achieves equality throughout~\eqref{eq.Welch bounds} if and only if it is an \textit{equi-isoclinic tight fusion frame} (EITFF) for $\bbF^D$.
Via~\eqref{eq.spectral and chordal coherence} and~\eqref{eq.spectral and chordal distance},
the inequalities of~\eqref{eq.Welch bounds} become the \textit{simplex bounds}~\cite{ConwayHS96}:
\begin{equation*}
\min_{n_1\neq n_2}\rmd_{\rms}^2(\calF_{n_1},\calF_{n_2})
\leq\min_{n_1\neq n_2}\tfrac1{R}\rmd_{\rmc}^2(\calF_{n_1},\calF_{n_2})
\leq\tfrac{N}{N-1}\tfrac{D-R}{D}.
\end{equation*}
In particular, each EITFF or ECTFF is an optimal Grassmannian packing with respect to the spectral or chordal distance, respectively,
with the former being well-suited for compressed sensing.
Every ECTFF with $R=1$ yields an \textit{equiangular tight frame}~(ETF), and vice versa.

For these and other reasons,
much attention has been paid to the \textit{existence problem}:
for what $\bbF$ and $(D,R,N)$ does a corresponding EITFF or ECTFF exist?
(We caution that in some of the literature, the parameters are presented in $(D,N,R)$ order.)
This problem remains largely open:
see~\cite{FickusMW21} for a survey of EITFFs and ECTFFs,
\cite{FickusM16} for a survey of ETFs in particular,
and~\cite{Brady25,FickusIJM22,FickusIJM23,FickusIJM24,FickusGI25,FallonI25,Gritsenko21,IversonJM24,IversonO25,FickusIJM26,FickusJM26} for some recent developments.
A classical~\cite{LemmensS73b} construction of such objects from \textit{direct sums} of smaller ones is particularly straightforward and well-known~\cite{CasazzaFMYZ11,CalderbankTX15}.

In this paper, we characterize the existence of complex TFFs whose \textit{cross-Gram matrices} $\bfPhi_{n_1}^*\bfPhi_{n_2}^{}$ commute.
In particular, in Section~3,
we show that such TFFs are equivalent to ones that arise from a finite sequence \smash{$\bigparen{(\bfphi_{r,n})_{n=1}^N}_{r=1}^{R}$} of $R$ \textit{unit norm tight frames} (UNTFs) for \smash{$\bbC^{\frac DR}$} via direct sums (Theorems~\ref{thm.SUNTF to ECTFF} and~\ref{thm.commuting}),
and are moreover equivalent to ones that arise as \textit{filter banks} (Theorem~\ref{thm.filter banks}).
In the latter case, for some subgroup $\calG$ of a finite abelian group $\calH$,
the orthonormal columns of each $\bfPhi_n$ arise as the translates of some vector $\bfphi_n\in\bbC^\calH$ that is related to $(\bfphi_{r,n})_{r=1}^R$ via classical \textit{polyphase} theory~\cite{Vaidyanathan93}.
These results generalize some of the main ideas of \cite{BolcskeiHF98,ChebiraFM11}.
More significantly, we show in these same results that such a TFF is EI if and only if each UNTF \textit{layer} $(\bfphi_{r,n})_{n=1}^N$ is an ETF for \smash{$\bbC^{\frac DR}$},
and more generally is EC if and only if these layers are \textit{stratified} in the sense of Definition~\ref{def.SUNTF}.
Notably, this leads to $(D,R,N)$ for which both an EITFF and a filter-bank ECTFF exist,
but no filter-bank EITFF exists (Example~\ref{ex.polyphase}).
This formally justifies some common sense:
though filter bank structure is often imposed in various applications such as remote sensing, wireless communication and machine learning,
doing so might come with a cost to incoherence.
To aid practitioners who would prefer a filter-bank ECTFF over a non-filter-bank EITFF,
we devote Section~4 to a general study of SUNTFs.
In particular, we show in Theorem~\ref{thm.SUNTF} that many results about ETFs generalize to the SUNTF setting,
including the Welch bound~\cite{Welch74,StrohmerH03} and Gerzon bound~\cite{LemmensS73,SustikTDH07}.
In Theorem~\ref{thm.DF to SUNTF},
we generalize the well-known equivalence between \textit{harmonic} ETFs and \textit{difference sets}~\cite{StrohmerH03,XiaZG05,DingF07},
showing that harmonic SUNTFs equate to \textit{difference families} for finite abelian groups.
When combined with Theorem~\ref{thm.SUNTF to ECTFF}, these recover an ECTFF construction technique of~\cite{FickusMW25}.
In Theorem~\ref{thm.Gauss sums},
we moreover generalize the Gauss-sum-based ETFs of~\cite{Renes07,Strohmer08} to the SUNTF setting.
See~\cite{Brinkley25} for preliminary versions of these results.
In the next section, we discuss some notation and previously known concepts that we will use later.

\section{Preliminaries}

To streamline notation, we do not require the indices of our vectors, matrices, etc.\ to be ordered.
Rather, for us here, a \textit{finite sequence} of members of a set $\calX$ is simply a function $\bfx$ from some nonempty finite set $\calN$ into $\calX$.
We usually denote such an $\bfx$ sequentially, that is, as $(x_n)_{n\in\calN}$ where $x_n:=\bfx(n)$ for any $n\in\calN$.
For any such sets $\calX$ and $\calN$, we denote the set of all such finite sequences $(x_n)_{n\in\calN}$ as $\calX^\calN$.

Though we focus here on real and complex Grassmannian packings,
our community occasionally considers generalizations of them to the quaternionic setting~\cite{CohnKM16,Waldron25},
and we accordingly regard vector spaces as right-modules.
(That said, bowing to tradition, we define left-multiplication by a real scalar to be right-multiplication by it.
Since $\bbR$ is the center of the quaternions $\bbH$, it is harmless to do so.)
In particular, for any field $\bbF$ and nonempty finite set $\calN$,
we equip the \textit{standard space} $\bbF^\calN$ with entrywise addition and vector-scalar multiplication,
that is, $(\bfx c)(n):=\bfx(n)c$ for any $\bfx\in\bbF^\calN$ and $c\in\bbF$.
We denote the standard basis for this space as $(\bfdelta_n)_{n\in\calN}$ and so $\bfdelta_n(n)=1$ while $\bfdelta_n(n')=0$ when $n'\neq n$.
It follows that for any nonempty finite sets $\calM$ and $\calN$,
a function $\bfL:\bbF^\calN\rightarrow\bbF^\calM$ is linear if and only if there exists (an \textit{$\calM\times\calN$ matrix}) $\bfA\in\bbF^{\calM\times\calN}$ such that $\bfL(\bfx)=\bfA\bfx$ for any $\bfx\in\bbF^\calN$ where $\bfA\bfx\in\bbF^\calM$, $(\bfA\bfx)(m):=\sum_{n\in\calN}\bfA(m,n)\bfx(n)$.
Here, $\bfA(m,n)=(\bfL\bfdelta_n)(m)$ for any $(m,n)\in\calM\times\calN$.
We usually identify such $\bfL$ and $\bfA$ with each other and denote them by a common symbol.
For a positive integer $N$,
we let $[N]:=\set{n\in\bbZ: 1\leq n\leq N}$ and denote $\bbF^{[N]}$ as $\bbF^N$.
Similarly, for positive integers $M$ and $N$,
an $M\times N$ matrix is a member of $\bbF^{M\times N}:=\bbF^{[M]\times[N]}$.

Now let $\calE$ be a vector space over $\bbF$ and let $\calN$ be a finite set of cardinality $N>0$.
The \textit{synthesis map} of a finite sequence $(\bfphi_n)_{n\in\calN}$ in $\calE$ is $\bfPhi:\bbF^\calN\rightarrow\calE$, $\bfPhi\bfx:=\sum_{n\in\calN}\bfphi_n\bfx(n)$.
Clearly, any such $\bfPhi$ is linear and satisfies $\bfPhi\bfdelta_n=\bfphi_n$ for any $n\in\calN$.
Conversely, since any $\bfx\in\bbF^\calN$ decomposes as $\bfx=\sum_{n\in\calN}\bfdelta_n\bfx(n)$,
any linear map $\bfPhi:\bbF^\calN\rightarrow\calE$ is the synthesis map of the finite sequence $(\bfphi_n)_{n\in\calN}$ in $\calE$ defined by $\bfphi_n:=\bfPhi\bfdelta_n$ for any $n\in\calN$.
The image of this map is $\im(\bfPhi)=\Span(\bfphi_n)_{n\in\calN}$
and $\bfPhi$ is injective if and only if $(\bfphi_n)_{n\in\calN}$ is linearly independent.
In particular, $\bfPhi$ is invertible if and only if $(\bfphi_n)_{n\in\calN}$ is a basis for $\calE$.
When $\calE=\bbF^\calM$, we identify $\bfPhi$ with the $\calM\times\calN$ matrix that has $\bfphi_n$ as its $n$th column.

Now let $\bbF$ be $\bbR$ or $\bbC$,
and denote the conjugate of $c\in\bbF$ as $\overline{c}$.
Let $\calE$ and $\calF$ be Euclidean (finite-dimensional Hilbert) spaces over $\bbF$,
let $\bfL:\calF\rightarrow\calE$ be linear,
and let \smash{$\bfPhi:\bbF^\calN\rightarrow\calE$} be the synthesis map of a finite sequence $(\bfphi_n)_{n\in\calN}$ in $\calE$.
We denote the induced norm (largest singular value) of $\bfL$ as $\norm{\bfL}_2$,
denote its adjoint as $\bfL^*$,
and denote the Frobenius inner product of two such maps as $\ip{\bfL_1}{\bfL_2}_{\Fro}:=\Tr(\bfL_1^*\bfL_2^{})$.
To facilitate generalization to $\bbH$ and collaboration with quantum theorists,
all inner products in this paper are conjugate-linear in their first argument and linear in their second.
In particular, our default inner product on $\bbF^\calN$ is \smash{$\ip{\bfx_1}{\bfx_2}:=\sum_{n\in\calN}\overline{\bfx_1(n)}\bfx_2(n)$},
and so the \textit{analysis map} of $(\bfphi_n)_{n\in\calN}$ is $\bfPhi^*:\calE\rightarrow\bbF^\calN$, \smash{$(\bfPhi^*\bfy)(n)=\ip{\bfphi_n}{\bfy}$}.
As a special case of this,
we identify each vector $\bfphi\in\calE$ with a linear map $\bfphi:\bbF\rightarrow\calE$, $\bfphi(x):=\bfphi x$ whose adjoint is the linear functional
$\bfphi^*:\calE\rightarrow\bbF$, $\bfphi^*\bfy=\ip{\bfphi}{\bfy}$.
The adjoint of $\bfA\in\bbF^{\calM\times\calN}$ is its conjugate-transpose $\bfA^*\in\bbF^{\calN\times\calM}$, \smash{$\bfA^*(n,m)=\overline{\bfA(m,n)}$}.
We make repeated use of the following result,
a short proof of which is given in the appendix.

\begin{lemma}[Folklore]
\label{lem.equal Gram}
If $\calE$, $\widetilde{\calE}$ and $\calF$ are Euclidean spaces and $\bfL:\calF\rightarrow\calE$ and \smash{$\widetilde{\bfL}:\calF\rightarrow\widetilde{\calE}$} are linear surjections then \smash{$\widetilde{\bfL}^*\widetilde{\bfL}=\bfL^*\bfL$} if and only if \smash{$\widetilde{\bfL}=\bfU\bfL$} for some unitary \smash{$\bfU:\calE\rightarrow\widetilde{\calE}$}.
\end{lemma}

We call $\bfL$ an \textit{isometry} if $\norm{\bfL\bfy}=\norm{\bfy}$ for any $\bfy\in\calF$,
namely if and only if $\bfL^*\bfL=\brI$.
When this occurs, $\bfL^*$ is a surjection and $\bfL\bfL^*$ is the (orthogonal) projection (operator) onto $\im(\bfL)$.
By Lemma~\ref{lem.equal Gram},
an isometry \smash{$\widetilde{\bfL}:\widetilde{\calF}\rightarrow\calE$} has this same image if and only if
\smash{$\widetilde{\bfL}=\bfL\bfU^*$} for some unitary \smash{$\bfU:\calF\rightarrow\widetilde{\calF}$}.
An invertible isometry is unitary, and vice versa.
The \textit{frame operator} of $(\bfphi_n)_{n\in\calN}$ is $\bfPhi\bfPhi^*=\sum_{n\in\calN}\bfphi_n^{}\bfphi_n^*$.
We call $(\bfphi_n)_{n\in\calN}$ a \textit{frame for $\calE$} if $\bfPhi\bfPhi^*:\calE\rightarrow\calE$ is invertible,
namely if and only if $(\bfphi_n)_{n\in\calN}$ spans $\calE$.
The \textit{Gram matrix} of $(\bfphi_n)_{n\in\calN}$ is $\bfPhi^*\bfPhi$.
Its $(n_1,n_2)$th entry is
$(\bfPhi^*\bfPhi)(n_1,n_2)
=\ip{\bfdelta_{n_1}}{\bfPhi^*\bfPhi\bfdelta_{n_2}}
=\ip{\bfphi_{n_1}}{\bfphi_{n_2}}$.
Thus, $(\bfphi_n)_{n\in\calN}$ is orthonormal if and only if $\bfPhi$ is an isometry,
and is an orthonormal basis for $\calE$ if and only if $\bfPhi:\bbF^\calN\rightarrow\calE$ is unitary.

Any Gram matrix is positive semidefinite (PSD).
Conversely, any PSD matrix $\bfG\in\bbF^{\calN\times\calN}$ factors as $\bfG=\bfU\bfLambda\bfU^*$ where $\bfU$ is the (unitary) synthesis map of an ONB \smash{$(\bfu_n)_{n=1}^{N}$} for $\bbF^\calN$ that consists of eigenvectors of $\bfG$ and \smash{$\bfLambda\in\bbF^{N\times N}$} is the diagonal matrix whose diagonal entries \smash{$(\lambda_n)_{n=1}^{N}$} are their corresponding eigenvalues, without loss of generality sorted in nonincreasing order.
It follows that any nonzero such $\bfG$ is the Gram matrix $\bfPhi^*\bfPhi$ of the frame $(\bfphi_n)_{n\in\calN}$ for $\bbF^{\,\rank(\bfG)}$ defined by
\smash{$\bfphi_n(d)=\sqrt{\lambda_d}\,\overline{\bfu_d(n)}$} for any $d\in[\rank(\bfG)]$.
(Here, $\bfPhi\bfPhi^*$ is diagonal and its $d$th diagonal entry is $\lambda_d$.)
By Lemma~\ref{lem.equal Gram},
it follows that any frame $(\widetilde{\bfphi}_n)_{n\in\calN}$ for $\widetilde{\calE}$ with Gram matrix $\bfG$ is a unitary mapping of $(\bfphi_n)_{n\in\calN}$ and, moreover,
that any putatively complex frame with a real-valued Gram matrix is a unitary mapping of a real frame.

We say $(\bfphi_n)_{n\in\calN}$ is \textit{unit norm} if $\norm{\bfphi_n}=1$ for any $n\in\calN$,
and is \textit{equiangular} if it is unit norm and the value of $\abs{\ip{\bfphi_{n_1}}{\bfphi_{n_2}}}$ is constant over all distinct $n_1,n_2\in\calN$.
We call $(\bfphi_n)_{n\in\calN}$ a \textit{tight frame for $\calE$} if $\bfPhi\bfPhi^*=A\brI$ for some \textit{frame constant} $A>0$.
When this occurs,
$A\dim(\calE)
=\Tr(A\brI)
=\Tr(\bfPhi\bfPhi^*)
=\Tr(\bfPhi^*\bfPhi)
=\sum_{n\in\calN}\norm{\bfphi_n}^2$
and $\frac1{A}\bfPhi^*\bfPhi$ is a projection onto a space of dimension $\rank(\bfPhi)=\dim(\calE)$.
Conversely,
by the previous paragraph,
any nonzero self-adjoint $\bfG\in\bbF^{\calN\times\calN}$ that satisfies $\bfG^2=A\bfG$ for some $A>0$ is necessarily the Gram matrix of a tight frame for a space of dimension $\rank(\bfG)$.
It follows that if $(\bfphi_n)_{n\in\calN}$ is a UNTF for $\calE$ then \smash{$A=\frac{N}{D}\geq1$} where $D:=\dim(\calE)$.
Here, if $N>D$, it moreover follows that $\frac{N}{N-D}(\brI-\frac{D}{N}\bfPhi^*\bfPhi)$ is the Gram matrix \smash{$\check{\bfPhi}^*\check{\bfPhi}$} of a UNTF
\smash{$(\check{\bfphi}_n)_{n\in\calN}$} for a Euclidean space \smash{$\check{\calE}$} over $\bbF$ of dimension $N-D$ that is unique up to unitary transformations, dubbed the \textit{Naimark complement} of $(\bfphi_n)_{n\in\calN}$.
In particular, if $(\bfphi_n)_{n\in\calN}$ is an ETF for $\calE$ with $N>D$ then
\smash{$(\check{\bfphi}_n)_{n\in\calN}$} is an ETF for \smash{$\check{\calE}$}.

\subsection{Fusion frames}

Let $D$, $R$ and $N$ be positive integers such that $D\geq R$,
let $\calE$ be a Euclidean space over $\bbF$ of dimension $D$,
let $\calR$ and $\calN$ be sets of cardinality $R$ and $N$, respectively,
and let $(\calF_n)_{n\in\calN}$ be a finite sequence of $R$-dimensional subspaces of $\calE$.
For each $n\in\calN$, let $\bfPhi_n:\bbF^\calR\rightarrow\calE$ be the synthesis map of an ONB $(\bfphi_{n,r})_{r\in\calR}$ for $\calF_n$.
Thus, $\bfPhi_n$ is an isometry onto $\calF_n$ and $\bfPi_n:=\bfPhi_n^{}\bfPhi_n^*=\sum_{r\in\calR}\bfphi_{n,r}^{}\bfphi_{n,r}^{*}$ is the projection onto $\calF_n$.
By Lemma~\ref{lem.equal Gram}, each $\bfPhi_n$ is unique up to right-multiplication by a unitary in \smash{$\bbF^{\calR\times\calR}$}.
The \textit{fusion frame operator} of $(\calF_n)_{n\in\calN}$ is the sum \smash{$\sum_{n\in\calN}\bfPi_n=\sum_{(n,r)\in\calN\times\calR}\bfphi_{n,r}^{}\bfphi_{n,r}^{*}$} of these projections.
It is thus the (ordinary) frame operator of the \textit{concatenation} $(\bfphi_{n,r})_{(n,r)\in\calN\times\calR}$ of these ONBs.
That is, $\bfPhi\bfPhi^*=\sum_{n\in\calN}\bfPi_n$ where $\bfPhi:\bbF^{\calN\times\calR}\rightarrow\calE$ is the synthesis map of this concatenation.
Since $\bfPhi\bfx=\sum_{n\in\calN}\bfPhi_n\bfx_n$ for any $\bfx\in\bbF^{\calN\times\calR}$ where $(\bfx_n)_{n\in\calN}\in(\bbF^\calR)^\calN$ is defined by $\bfx_n(r):=\bfx(n,r)$,
we call $\bfPhi$ the \textit{fusion synthesis map} of $(\bfPhi_n)_{n\in\calN}$.
Since $(\bfPhi^*\bfPhi)\bigparen{(n_1,r_1),(n_2,r_2)}
=\ip{\bfphi_{n_1,r_1}}{\bfphi_{n_2,r_2}}
=(\bfPhi_{n_1}^{*}\bfPhi_{n_2}^{})(r_1,r_2)$
for any $n_1,n_2\in\calN$, $r_1,r_2\in\calR$,
we regard the \textit{fusion Gram matrix} $\bfPhi^*\bfPhi$ as a block matrix, in particular as the $\calN\times\calN$ array of $\calR\times\calR$ blocks whose $(n_1,n_2)$th block is the \textit{cross-Gram} matrix $\bfPhi_{n_1}^{*}\bfPhi_{n_2}^{}\in\bbF^{\calR\times\calR}$ of the two corresponding isometries.
Since each $\bfPhi_n$ is an isometry,
each diagonal block is $\brI$.

Since each cross-Gram matrix $\bfPhi_{n_1}^{*}\bfPhi_{n_2}^{}$ has
$\norm{\bfPhi_{n_1}^*\bfPhi_{n_2}^{}}_2
\leq\norm{\bfPhi_{n_1}}_2\norm{\bfPhi_{n_2}}_2
=1$,
its singular values can be written as \smash{$(\cos\theta_{n_1,n_2,r})_{r=1}^{R}$} for some nondecreasing finite sequence $(\theta_{n_1,n_2,r})_{r=1}^{R}$ in $[0,\frac{\pi}{2}]$.
As singular values are preserved under multiplication by unitaries,
this finite sequence $(\theta_{n_1,n_2,r})_{r=1}^{R}$ of \textit{principal angles between $\calF_{n_1}$ and $\calF_{n_2}$} is invariant to one's choice of orthonormal bases for these subspaces.
For any $n_1,n_2\in\calN$,
note that
\smash{$\frac1{\sqrt{R}}\norm{\bfPhi_{n_1}^*\bfPhi_{n_2}^{}}_\Fro
\leq\norm{\bfPhi_{n_1}^*\bfPhi_{n_2}^{}}_2$} and that equality holds here if and only if $\calF_{n_1}$ and $\calF_{n_2}$ are \textit{isoclinic} in the sense that \smash{$\theta_{n_1,n_2,r}$} is constant over all $r\in[R]$.
Notably, \smash{$\theta_{n_1,n_2,r}$} is constant over all distinct $n_1,n_2\in\calN$ and all $r\in[R]$ if and only if $(\calF_n)_{n\in\calN}$ is EI,
namely if and only if there exists $\varLambda\geq0$ such that \smash{$\bfPhi_{n_1}^*\bfPhi_{n_2}^{}\bfPhi_{n_2}^*\bfPhi_{n_1}^{}=\varLambda\brI$}
(or equivalently, \smash{$\bfPi_{n_1}\bfPi_{n_2}\bfPi_{n_1}=\varLambda\bfPi_{n_1}$}) for any distinct $n_1,n_2\in\calN$.
More generally, $(\calF_n)_{n\in\calN}$ is EC if there exists a \textit{chordal constant} $\varLambda\geq0$ such that
\smash{$\frac1{R}\Tr(\bfPi_{n_1}\bfPi_{n_2})
=\frac1{R}\norm{\bfPhi_{n_1}^*\bfPhi_{n_2}^{}}_\Fro^2
=\varLambda$}
for any distinct $n_1,n_2\in\calN$.
The (EC version of the) \textit{Gerzon bound}~\cite{LemmensS73,SustikTDH07} gives that if $\calE$ contains $N$ nonidentical EC subspaces then
\begin{equation}
\label{eq.Gerzon}
N
\leq\dim_\bbR\set{\bfL:\calE\rightarrow\calE\ |\ \bfL^*=\bfL}
=D+\tfrac{D(D-1)}{2}[\bbF:\bbR]
=\left\{\begin{array}{cl}
\tfrac12D(D+1),&\ \bbF=\bbR,\\
D^2,&\ \bbF=\bbC.
\end{array}\right.
\end{equation}
For the sake of completeness, a proof of this fact is given in the appendix.
Zauner famously conjectured~\cite{Zauner99} that it can be met for any $D\geq 2$ provided $\bbF=\bbC$ and $R=1$.
For refinements to the Gerzon bound in the special case where $(\calF_n)_{n\in\calN}$ is assumed to be EI,
see~\cite{LemmensS73b,FickusGI25,IversonO25}.

We call \smash{$(\calF_n)_{n\in\calN}$} a TFF for $\calE$ if \smash{$\sum_{n\in\calN}\bfPi_n=A\brI$} for some $A>0$.
This occurs if and only if \smash{$(\bfphi_{n,r})_{(n,r)\in\calN\times\calR}$} is a UNTF for $\calE$ and so necessarily \smash{$A=\frac{NR}{D}\geq 1$}.
Assuming for the moment that this holds with $NR>D$,
this UNTF has a Naimark complement \smash{$(\check{\bfphi}_{n,r})_{(n,r)\in\calN\times\calR}$} that is a UNTF for a Euclidean space \smash{$\check{\calE}$} of dimension $NR-D$ that is uniquely defined (up to unitary transformations) by having \smash{$\check{\bfPhi}^*\check{\bfPhi}=\frac{NR}{NR-D}(\brI-\frac{D}{NR}\bfPhi^*\bfPhi)$} as its Gram matrix.
Comparing diagonal blocks gives that for each $n\in\calN$, the synthesis map $\check{\bfPhi}_n$ of \smash{$(\check{\bfphi}_{n,r})_{r\in\calR}$} satisfies
\smash{$\check{\bfPhi}_{n}^{*}\check{\bfPhi}_{n}^{}=\frac{NR}{NR-D}(\brI-\frac{D}{NR}\brI)=\brI$} and so is an isometry onto \smash{$\check{\calF}_n:=\Span(\check{\bfphi}_{n,r})_{r\in\calR}$}.
Comparing off-diagonal blocks, it follows that \smash{$(\check{\bfPhi}_n)_{n\in\calN}$} is a finite sequence of isometries for a TFF \smash{$(\check{\calF}_n)_{n\in\calN}$} for \smash{$\check{\calE}$} with
\smash{$\check{\bfPhi}_{n_1}^{*}\check{\bfPhi}_{n_2}^{}=-\frac{D}{NR-D}\bfPhi_{n_1}^{*}\bfPhi_{n_2}^{}$} for any distinct $n_1,n_2\in\calN$.
In particular, \smash{$(\check{\calF}_n)_{n\in\calN}$} is an ECTFF or EITFF for \smash{$\check{\calE}$} if
\smash{$(\calF_n)_{n\in\calN}$} is an ECTFF or EITFF for $\calE$, respectively.
When this occurs, we call \smash{$(\check{\bfPhi}_n)_{n\in\calN}$} a Naimark complement of \smash{$(\bfPhi_n)_{n\in\calN}$}.

Some of our new results here rely upon the Welch bounds~\eqref{eq.Welch bounds} and so, for the sake of completeness,
we provide proofs of the latter in the appendix.
In the special case where $R=1$,
we regard each isometry as $\bfPhi_n:\bbF\rightarrow\calE$, $\bfPhi_n x:=\bfphi_n x$ for some unit norm vector $\bfphi_n\in\calE$.
Here,
$\sum_{n\in\calN}\bfPi_n=\sum_{n\in\calN}\bfphi_{n}^{}\bfphi_{n}^{*}$ and each cross-Gram matrix $\bfPhi_{n_1}^{*}\bfPhi_{n_2}^{}$ has $\ip{\bfphi_{n_1}}{\bfphi_{n_2}}$ as its sole entry,
implying
\smash{$\norm{\bfPhi_{n_1}^{*}\bfPhi_{n_2}^{}}_\Fro
=\abs{\ip{\bfphi_{n_1}}{\bfphi_{n_2}}}
=\norm{\bfPhi_{n_1}^{*}\bfPhi_{n_2}^{}}_2$}.
Thus, in this setting, $(\calF_n)_{n\in\calN}$ is a TFF for $\calE$ if and only if $(\bfphi_{n})_{n\in\calN}$ is a UNTF for $\calE$,
and the following are equivalent: $(\calF_n)_{n\in\calN}$ is EC; $(\calF_n)_{n\in\calN}$ is EI; $(\bfphi_{n})_{n\in\calN}$ is equiangular.
Here, when $N>1$, both quantities in~\eqref{eq.spectral and chordal coherence} reduce to the \textit{coherence}
\smash{$\mu:=\max_{n_1\neq n_2}\abs{\ip{\bfphi_{n_1}}{\bfphi_{n_2}}}$} of \smash{$(\bfphi_n)_{n\in\calN}$},
and both bounds in~\eqref{eq.Welch bounds} reduce to the classical Welch bound~\cite{Welch74},
namely that \smash{$\mu^2\geq\frac{N-D}{D(N-1)}$},
which holds with equality if and only if \smash{$(\bfphi_{n})_{n\in\calN}$} is an ETF for $\calE$.
We shall also make use  of the following result.

\begin{lemma}[Folklore]
\label{lem.equivalence}
Let $\bbF$ be $\bbR$ or $\bbC$,
let $\calN$, $\calR$ and \smash{$\widetilde{\calR}$} be nonempty finite sets,
and let $\calE$ and \smash{$\widetilde{\calE}$} be Euclidean spaces over $\bbF$ of equal dimension.
For each $n\in\calN$, let $\calF_n$ and \smash{$\widetilde{\calF}_n$} be the images of isometries \smash{$\bfPhi_n:\bbF^{\calR}\rightarrow\calE$} and \smash{$\widetilde{\bfPhi}_n:\bbF^{\widetilde{\calR}}\rightarrow\widetilde{\calE}$}, respectively.
The following are equivalent:
\begin{enumerate}
\romani

\item
$\exists$ a unitary
\smash{$\bfU:\calE\rightarrow\widetilde{\calE}$}
such that
\smash{$\widetilde{\calF}_n=\bfU(\calF_n)$} for any $n\in\calN$.
\smallskip

\item
$\exists$ a unitary
\smash{$\bfU:\calE\rightarrow\widetilde{\calE}$}
such that
\smash{$\widetilde{\bfPhi}_n^{}\widetilde{\bfPhi}_n^{*}=\bfU\bfPhi_n^{}\bfPhi_n^{*}\bfU^*$} for any $n\in\calN$.
\smallskip

\item
$\exists$ unitaries
\smash{$\bfU:\calE\rightarrow\widetilde{\calE}$}
and
\smash{$(\bfV_n)_{n\in\calN}$} in \smash{$\bbF^{\widetilde{\calR}\times\calR}$}
such that
\smash{$\widetilde{\bfPhi}_n^{}=\bfU\bfPhi_n^{}\bfV_n^*$} for any $n\in\calN$.
\smallskip

\item
$\exists$ unitaries
\smash{$(\bfV_n)_{n\in\calN}$} in \smash{$\bbF^{\widetilde{\calR}\times\calR}$}
such that
\smash{$\widetilde{\bfPhi}_{n_1}^{*}\widetilde{\bfPhi}_{n_2}^{}
=\bfV_{n_1}^{}\bfPhi_{n_1}^{*}\bfPhi_{n_2}^{}\bfV_{n_2}^{*}$}
for any $n_1,n_2\in\calN$.

\end{enumerate}
In fact, if (i), (ii), (iii) or (iv) holds for a particular such $\bfU$ or \smash{$(\bfV_n)_{n\in\calN}$},
then (i)--(iv) hold with that particular $\bfU$ or \smash{$(\bfV_n)_{n\in\calN}$}, respectively.
\end{lemma}

For the sake of completeness, a proof of Lemma~\ref{lem.equivalence} is given in the appendix.
In this context, we say that \smash{$(\widetilde{\calF}_n)_{n\in\calN}$} is \textit{equivalent} to \smash{$(\calF_n)_{n\in\calN}$} if (i) holds,
and say that \smash{$(\widetilde{\bfPhi}_n)_{n\in\calN}$} is equivalent to \smash{$(\bfPhi_n)_{n\in\calN}$} if (iii) holds.
This notion of equivalence preserves equi-isoclinism, equichordality and tightness.
As such, the existence of an (EC/EI) TFF depends solely on the underlying field $\bbF$,
the dimension $D$ of the ambient space $\calE$,
and the dimension $R$ and number $N$ of subspaces $(\calF_n)_{n\in\calN}$.
The simplifying assumptions used in the introduction,
namely that $\calE=\bbF^D$, $\calR=[R]$ and $\calN=[N]$,
are thus made without loss of generality.
As such, we denote such objects as \smash{(EC/EI) $\TFF_\bbF(D,R,N)$}.
Similarly, we denote an $N$-vector UNTF or ETF for a Euclidean space over $\bbF$ of dimension $D$ as an \smash{$\UNTF_\bbF(D,N)$} or \smash{$\ETF_\bbF(D,N)$}, respectively.

\section{The coherence of filter banks}

For any integer $M>0$, tensoring the finite sequence of projections $(\bfPi_n)_{n\in\calN}$ of a (EC/EI) $\TFF_\bbF(D,R,N)$ with an identity matrix $\brI$ of size $M$ yields a (EC/EI) $\TFF_\bbF(MD,MR,N)$~\cite{LemmensS73b,CasazzaFMYZ11,CalderbankTX15}.
Moreover, entrywise application of the map \smash{$u+\rmi v\mapsto\left[\begin{smallmatrix}u&-v\\v&\phm u\end{smallmatrix}\right]$}
converts a (EC/EI) $\TFF_\bbC(D,R,N)$ into a (EC/EI) $\TFF_\bbR(2D,2R,N)$,
and the map
\smash{$u+\rmj v\mapsto\left[\begin{smallmatrix}u&-\overline{v}\\v&\phm\overline{u}\end{smallmatrix}\right]$} similarly converts a (EC/EI) $\TFF_\bbH(D,R,N)$ into a (EC/EI) $\TFF_\bbC(2D,2R,N)$~\cite{Hoggar77,Waldron25}.
Our work here relies on the following unification, generalization and elaboration of these ideas \`{a} la~\cite{FickusMW21}:

\begin{lemma}[Folklore]
\label{lem.FF direct sums}
Let $\bbF$ be $\bbR$ or $\bbC$ and let $\calM$ and $\calN$ be nonempty finite sets.
For each $m\in\calM$, let $\calR_m$ be a nonempty finite set and let $\calE_m$ be a Euclidean space over $\bbF$.
For each $m\in\calM$ and $n\in\calN$ let $\calF_{m,n}$ be the image of an isometry $\bfPhi_{m,n}:\bbF^{\calR_m}\rightarrow\calE_m$.
Let \smash{$\widehat{\calE}:=\oplus_{m\in\calM}\calE_m$} and
identify \smash{$\oplus_{m\in\calM}\bbF^{\calR_m}$} with \smash{$\bbF^{\widehat{\calR}}$} where \smash{$\widehat{\calR}:=\sqcup_{m\in\calM}(\set{m}\times\calR_m)$}.
For each $n\in\calN$, let \smash{$\widehat{\calF}_n$} be the image of \smash{$\widehat{\bfPhi}_{n}:\bbF^{\widehat{\calR}}\rightarrow\widehat{\calE}$},
\smash{$\widehat{\bfPhi}_{n}:=\oplus_{m\in\calM}\,\bfPhi_{m,n}$}.
\begin{enumerate}
\alphi

\item
$\widehat{\bfPhi}_{n_1}^{*}\widehat{\bfPhi}_{n_2}^{}
=\oplus_{m\in\calM}\,\bfPhi_{m,n_1}^{*}\bfPhi_{m,n_2}^{}$
for any $n_1,n_2\in\calN$.
In particular, each $\widehat{\bfPhi}_{n}$ is an isometry.

\item
$\max_{n_1\neq n_2}\norm{\widehat{\bfPhi}_{n_1}^{*}\widehat{\bfPhi}_{n_2}^{}}_2
=\max_{m\in\calM}\max_{n_1\neq n_2}\norm{\bfPhi_{m,n_1}^{*}\bfPhi_{m,n_2}^{}}_2$,
provided $\#(\calN)>1$.

\item
$(\widehat{\calF}_n)_{n\in\calN}$ is a TFF for $\widehat{\calE}$ if and only if every $(\calF_{m,n})_{n\in\calN}$ is a TFF for $\calE_m$ and
\begin{equation}
\label{eq.D/R consistency}
\tfrac{\dim(\calE_{m_1})}{\#(\calR_{m_1})}
=\tfrac{\dim(\calE_{m_2})}{\#(\calR_{m_2})}
\quad\forall\,m_1,m_2\in\calM.
\end{equation}
When this occurs and \smash{$\frac{\dim(\calE_{m})}{\#(\calR_{m})}<\#(\calN)$} for any $m\in\calM$,
each \smash{$(\bfPhi_{m,n})_{n\in\calN}$} has a Naimark complement \smash{$(\check{\bfPhi}_{m,n})_{n\in\calN}$}
and \smash{$(\oplus_{m\in\calM}\,\check{\bfPhi}_{m,n})_{n\in\calN}$} is a Naimark complement of
\smash{$(\widehat{\bfPhi}_n)_{n\in\calN}$}.

\item
$(\widehat{\calF}_n)_{n\in\calN}$ is EI if and only if every $(\calF_{m,n})_{n\in\calN}$ is EI with a common chordal constant $\varLambda\geq0$, namely if and only if
$\bfPhi_{m,n_1}^{*}\bfPhi_{m,n_2}^{}\bfPhi_{m,n_2}^{*}\bfPhi_{m,n_1}^{}=\varLambda\brI$ for any $m\in\calM$ and distinct $n_1,n_2\in\calN$.

\item
$(\widehat{\calF}_n)_{n\in\calN}$ is an EITFF for $\widehat{\calE}$ if and only if every $(\calF_{m,n})_{n\in\calN}$ is an EITFF for $\calE_m$ and~\eqref{eq.D/R consistency} holds.

\item
$(\widehat{\calF}_n)_{n\in\calN}$ is EC if $(\calF_{m,n})_{n\in\calN}$ is EC for any $m\in\calM$.

\item
$(\widehat{\calF}_n)_{n\in\calN}$ is an ECTFF for $\widehat{\calE}$ if every $(\calF_{m,n})_{n\in\calN}$ is an ECTFF for $\calE_m$ and~\eqref{eq.D/R consistency} holds.

\end{enumerate}
\end{lemma}

\begin{proof}[Proof.]
(a)
For any $n_1,n_2\in\calN$,
\begin{equation*}
\widehat{\bfPhi}_{n_1}^{*}\widehat{\bfPhi}_{n_2}^{}
=\Bigparen{\,\medoplus_{m\in\calM}\bfPhi_{m,n_1}^{}}^*\Bigparen{\,\medoplus_{m\in\calM}\bfPhi_{m,n_2}^{}}
=\Bigparen{\,\medoplus_{m\in\calM}\bfPhi_{m,n_1}^{*}}\Bigparen{\,\medoplus_{m\in\calM}\bfPhi_{m,n_2}^{}}
=\medoplus_{m\in\calM}\bfPhi_{m,n_1}^{*}\bfPhi_{m,n_2}^{}.
\end{equation*}
In particular, when $n_1=n_2=n$, this becomes
$\widehat{\bfPhi}_{n}^{*}\widehat{\bfPhi}_{n}^{}
=\oplus_{m\in\calM}\brI_{\calR_m}
=\brI_{\widehat{\calR}}$.

(b)
For any $n_1,n_2\in\calN$,
\begin{equation*}
\norm{\widehat{\bfPhi}_{n_1}^{*}\widehat{\bfPhi}_{n_2}^{}}_2
=\Bignorm{\medoplus_{m\in\calM}\bfPhi_{m,n_1}^{*}\bfPhi_{m,n_2}^{}}_2
=\max_{m\in\calM}\norm{\bfPhi_{m,n_1}^{*}\bfPhi_{m,n_2}^{}}_2
\end{equation*}
and so
$\displaystyle\max_{n_1\neq n_2}\norm{\widehat{\bfPhi}_{n_1}^{*}\widehat{\bfPhi}_{n_2}^{}}_2
=\max_{n_1\neq n_2}\max_{m\in\calM}\norm{\bfPhi_{m,n_1}^{*}\bfPhi_{m,n_2}^{}}_2
=\max_{m\in\calM}\max_{n_1\neq n_2}\norm{\bfPhi_{m,n_1}^{*}\bfPhi_{m,n_2}^{}}_2$.

(c)
For any $n\in\calN$,
\begin{equation*}
\widehat{\bfPhi}_{n}^{}\widehat{\bfPhi}_{n}^{*}
=\Bigparen{\,\medoplus_{m\in\calM}\bfPhi_{m,n}^{}}\Bigparen{\,\medoplus_{m\in\calM}\bfPhi_{m,n}^{}}^*
=\Bigparen{\,\medoplus_{m\in\calM}\bfPhi_{m,n}^{}}\Bigparen{\,\medoplus_{m\in\calM}\bfPhi_{m,n}^{*}}
=\medoplus_{m\in\calM}\bfPhi_{m,n}^{}\bfPhi_{m,n}^{*}
\end{equation*}
and so \smash{$(\widehat{\calF}_n)_{n\in\calN}$} is a TFF for \smash{$\widehat{\calE}$} if and only if there exists $A>0$ such that
\begin{equation*}
A\medoplus_{m\in\calM}\brI_{\calE_m}
=A\brI_{\widehat{\calE}}
=\sum_{n\in\calN}\widehat{\bfPhi}_{n}^{}\widehat{\bfPhi}_{n}^{*}
=\sum_{n\in\calN}\medoplus_{m\in\calM}\bfPhi_{m,n}^{}\bfPhi_{m,n}^{*}
=\medoplus_{m\in\calM}\sum_{n\in\calN}\bfPhi_{m,n}^{}\bfPhi_{m,n}^{*},
\end{equation*}
namely such that
\smash{$\sum_{n\in\calN}\bfPhi_{m,n}^{}\bfPhi_{m,n}^{*}=A\brI_{\calE_m}$} for any $m\in\calM$.
That is, \smash{$(\widehat{\calF}_n)_{n\in\calN}$} is a TFF for $\widehat{\calE}$ if and only if every $(\calF_{m,n})_{n\in\calN}$ is a TFF for $\calE_m$ with a common frame constant $A>0$.
Here since the frame constant of any such TFF \smash{$(\calF_{m,n})_{n\in\calN}$} is necessarily \smash{$\frac{NR_m}{D_m}$} where $N:=\#(\calN)$, $R_m:=\#(\calR_m)$ and $D_m:=\dim(\calE_m)$,
this occurs if and only if every \smash{$(\calF_{m,n})_{n\in\calN}$} is a TFF for \smash{$\calE_m$} and~\eqref{eq.D/R consistency} holds.
Assume for the moment that this occurs.
Since \smash{$(\widehat{\calF}_n)_{n\in\calN}$} is a TFF for \smash{$\widehat{\calE}$},
we moreover have \smash{$A=\frac{N\widehat{R}}{\widehat{D}}$} where
\begin{align*}
\widehat{D}
&:=\dim(\widehat{\calE})
=\dim(\oplus_{m\in\calM}\calE_m)
=\sum_{m\in\calM}\dim(\calE_m)
=\sum_{m\in\calM}D_m,\\
\widehat{R}
&:=\#(\widehat{\calR})
=\#\bigparen{\sqcup_{m\in\calM}(\set{m}\times\calR_m)}
=\sum_{m\in\calM}\#(\calR_m)
=\sum_{m\in\calM}R_m.
\end{align*}
Thus, \smash{$\frac{D_m}{R_m}=\frac{\widehat{D}}{\widehat{R}}$} for any $m\in\calM$.
Now further assume that \smash{$\frac{\widehat{D}}{\widehat{R}}<N$}.
For each $m\in\calM$, we thus have $NR_m>D_m$,
implying \smash{$(\bfPhi_{m,n})_{n\in\calN}$} has a Naimark complement \smash{$(\check{\bfPhi}_{m,n})_{n\in\calN}$}.
Here, for each $m\in\calM$,
\begin{equation}
\label{eq.pf.lem.FF direct sums.1}
\check{\bfPhi}_{m,n_1}^{*}\check{\bfPhi}_{m,n_2}^{}
=-\tfrac{D_m}{NR_m-D_m}\bfPhi_{m,n_1}^{*}\bfPhi_{m,n_2}^{}
=-\tfrac{\widehat{D}}{N\widehat{R}-\widehat{D}}\bfPhi_{m,n_1}^{*}\bfPhi_{m,n_2}^{}
\end{equation}
for any distinct $n_1,n_2\in\calN$,
and the images of \smash{$(\check{\bfPhi}_{m,n})_{n\in\calN}$} form a TFF for some Euclidean space \smash{$\check{\calE}_m$} over $\bbF$ of dimension \smash{$NR_m-D_m$}.
Here since
\begin{equation*}
\tfrac{\dim(\check{\calE}_{m})}{\#(\calR_{m})}
=\tfrac{NR_m-D_m}{R_m}
=\tfrac{N\widehat{R}-\widehat{D}}{\widehat{R}}
\end{equation*}
for any $m\in\calM$,
these isometries satisfy~\eqref{eq.D/R consistency}.
By the previous results,
\smash{$(\widehat{\check{\bfPhi}}_{n})_{n\in\calN}$},
\smash{$\widehat{\check{\bfPhi}}_{n}:=\oplus_{m\in\calM}\check{\bfPhi}_{m,n}$}
thus defines isometries of rank $\widehat{R}$ whose images form a TFF for the space \smash{$\oplus_{m\in\calM}\check{\calE}_m$} of dimension
\begin{equation*}
\dim(\oplus_{m\in\calM}\check{\calE}_m)
=\sum_{m\in\calM}\dim(\check{\calE}_m)
=\sum_{m\in\calM}(NR_m-D_m)
=N\widehat{R}-\widehat{D},
\end{equation*}
and for which, by~\eqref{eq.pf.lem.FF direct sums.1},
\begin{equation*}
\widehat{\check{\bfPhi}}_{n_1}^{*}\widehat{\check{\bfPhi}}_{n_2}^{}
=\oplus_{m\in\calM}\,\check{\bfPhi}_{m,n_1}^{*}\check{\bfPhi}_{m,n_2}^{}
=\oplus_{m\in\calM}\,\bigparen{-\tfrac{D_m}{NR_m-D_m}\bfPhi_{m,n_1}^{*}\bfPhi_{m,n_2}^{}}
=-\tfrac{\widehat{D}}{N\widehat{R}-\widehat{D}}\widehat{\bfPhi}_{n_1}^{*}\widehat{\bfPhi}_{n_2}^{}
\end{equation*}
for any distinct $n_1,n_2\in\calN$.
Thus, $(\widehat{\check{\bfPhi}}_n)_{n\in\calN}$ is indeed a Naimark complement of $(\widehat{\bfPhi}_n)_{n\in\calN}$.

(d)
Similarly,
\smash{$(\widehat{\calF}_n)_{n\in\calN}$} is EI if and only if there exists $\varLambda\geq0$ such that
\begin{equation*}
\varLambda\brI
=\widehat{\bfPhi}_{n_1}^{*}\widehat{\bfPhi}_{n_2}^{}\widehat{\bfPhi}_{n_2}^{*}\widehat{\bfPhi}_{n_1}^{}
=\Bigparen{\medoplus_{m\in\calM}\bfPhi_{m,n_1}^{*}\bfPhi_{m,n_2}^{}}
\Bigparen{\medoplus_{m\in\calM}\bfPhi_{m,n_2}^{*}\bfPhi_{m,n_1}^{}}
=\medoplus_{m\in\calM}\bfPhi_{m,n_1}^{*}\bfPhi_{m,n_2}^{}\bfPhi_{m,n_2}^{*}\bfPhi_{m,n_1}^{}
\end{equation*}
for any distinct $n_1,n_2\in\calN$,
namely such that \smash{$\bfPhi_{m,n_1}^{*}\bfPhi_{m,n_2}^{}\bfPhi_{m,n_2}^{*}\bfPhi_{m,n_1}^{}=\varLambda\brI$} for any $m\in\calM$ and distinct $n_1,n_2\in\calN$.

(e)
Here we assume $N>1$ since the result is immediate if $N=1$.
Recall that if \smash{$(\calF_{m,n})_{n\in\calN}$} is an EITFF for $\calE_m$ then its chordal constant is
\smash{$\frac{NR_m-D_m}{D_m(N-1)}
=\frac{1}{N-1}(N\frac{R_m}{D_m}-1)$}.
As such, by (c) and (d),
\smash{$(\widehat{\calF}_n)_{n\in\calN}$} is an EITFF for \smash{$\widehat{\calE}$}
if and only if every \smash{$(\calF_{m,n})_{n\in\calN}$} is an EITFF for \smash{$\calE_m$} and~\eqref{eq.D/R consistency} holds.

(f)
For any $n_1,n_2\in\calN$, (a) gives
\begin{equation*}
\norm{\widehat{\bfPhi}_{n_1}^{*}\widehat{\bfPhi}_{n_2}^{}}_\Fro^2
=\Bignorm{\medoplus_{m\in\calM}\bfPhi_{m,n_1}^{*}\bfPhi_{m,n_2}^{}}_\Fro^2
=\sum_{m\in\calM}\norm{\bfPhi_{m,n_1}^{*}\bfPhi_{m,n_2}^{}}_\Fro^2.
\end{equation*}
As such, if every finite sequence \smash{$(\calF_{m,n})_{n\in\calN}$} is EC with chordal constant $\varLambda_m$ then
\begin{equation*}
\tfrac1{\#(\widehat{\calR})}\norm{\widehat{\bfPhi}_{n_1}^{*}\widehat{\bfPhi}_{n_2}^{}}_\Fro^2
=\tfrac1{\widehat{R}}\sum_{m\in\calM}\norm{\bfPhi_{m,n_1}^{*}\bfPhi_{m,n_2}^{}}_\Fro^2
=\tfrac1{\widehat{R}}\sum_{m\in\calM}R_m\varLambda_m
\end{equation*}
for any distinct $n_1,n_2\in\calN$ and so \smash{$(\widehat{\calF}_n)_{n\in\calN}$} is EC.

(g)
If each \smash{$(\calF_{m,n})_{n\in\calN}$} is an ECTFF for $\calE_m$ and~\eqref{eq.D/R consistency} holds then (c) and (f) imply that \smash{$(\widehat{\calF}_n)_{n\in\calN}$} is an ECTFF for \smash{$\widehat{\calE}$}.
\end{proof}

Notably, the converses of parts (f) and (g) of Lemma~\ref{lem.FF direct sums} are not true in general.
Rather, Lemma~\ref{lem.FF direct sums}(a) implies that $(\widehat{\calF}_n)_{n\in\calN}$ is EC if and only if there exists $\varLambda\geq0$ such that
\smash{$\tfrac{1}{\#(\widehat{\calR})}\sum_{m\in\calM}\norm{\bfPhi_{m,n_1}^{*}\bfPhi_{m,n_2}^{}}_\Fro^2=\varLambda$}
for any distinct $n_1,n_2\in\calN$.

In this manuscript,
we focus on ECTFFs arising via Lemma~\ref{lem.FF direct sums} in the special case where each $\calF_{m,n}$ is one-dimensional,
and leave a more general exploitation of that approach for future research.
In this case,
each $\calR_m$ is a singleton set,
and so we streamline notation by identifying both $\calM$ and $\widehat{\calR}$ with $\calR$,
each $\bbF^{\calR_m}$ with $\bbF$,
and each $\bfPhi_{m,n}$ with an isometry $\bfPhi_{r,n}:\bbF\rightarrow\calE_r$,
namely a map for which there exists a unit norm vector $\bfphi_{r,n}\in\calE_r$ such that $\bfPhi_{r,n}x=\bfphi_{r,n}x$ for any $x\in\bbF$.
Under these simplifications,
some of the concepts above reduce to the following:

\begin{definition}
\label{def.SUNTF}
Let $\calN$ and $\calR$ be finite sets of cardinality $N>0$ and $R>0$, respectively.
For each $r\in\calR$,
let $(\bfphi_{r,n})_{n\in\calN}$ be a finite sequence of unit vectors in a Euclidean space $\calE_r$ over $\bbF$.
We say \smash{$\bigparen{(\bfphi_{r,n})_{n\in\calN}}_{r\in\calR}$} is \textit{stratified} if there exists $\varLambda\geq0$ such that
\begin{equation*}
\tfrac{1}{R}\sum_{r\in\calR}\abs{\ip{\bfphi_{r,n_1}}{\bfphi_{r,n_2}}}^2
=\varLambda
\end{equation*}
for any distinct $n_1,n_2\in\calN$.
If this holds and each \textit{layer} \smash{$(\bfphi_{r,n})_{n\in\calN}$} is a UNTF for \smash{$\calE_r$} where there exists $D>0$ such that \smash{$\dim(\calE_r)=D$} for any $r\in\calR$,
we call \smash{$\bigparen{(\bfphi_{r,n})_{n\in\calN}}_{r\in\calR}$} a SUNTF for $(\calE_r)_{r\in\calR}$ or a \smash{$\SUNTF_\bbF(D,N,R)$}.
\end{definition}

\begin{example}
\label{ex.SUNTF(2,5,2)}
A $\SUNTF_\bbF(D,N,R)$ with $R=1$ equates to an $\ETF_\bbF(D,N)$.
For an example with $R>1$,
consider the fifth roots of unity regarded as a finite sequence $(\bfphi_n)_{n\in\bbF_5}$ of five unit vectors in $\bbR^2$,
indexed by the finite field of order $5$,
namely the columns of the synthesis matrix
\begin{align*}
\bfPhi
=\left[\begin{array}{ccccc}
\bfphi_{0}&\bfphi_{1}&\bfphi_{2}&\bfphi_{3}&\bfphi_{4}
\end{array}\right]
=\left[\begin{array}{rrrrr}
1&u&v&v&u\\
0&\phm U&\phm V&-V&-U
\end{array}\right],
\end{align*}
where $u:=\cos(\frac{2\pi}{5})=\tfrac14(-1+\sqrt{5})$,
$v:=\cos(\frac{4\pi}{5})=\tfrac14(-1-\sqrt{5})$,
$U:=\sqrt{1-u^2}$, $V:=\sqrt{1-v^2}$.
A direct calculation gives that the frame operator and Gram matrix of $(\bfphi_n)_{n\in\bbF_5}$ are
\begin{equation*}
\bfPhi\bfPhi^*
=\left[\begin{array}{cc}\frac{5}{2}&0\smallskip\\0&\frac{5}{2}\end{array}\right],
\quad
\bfPhi^*\bfPhi
=\left[\begin{array}{ccccc}
1&u&v&v&u\\
u&1&u&v&v\\
v&u&1&u&v\\
v&v&u&1&u\\
u&v&v&u&1
\end{array}\right],
\end{equation*}
respectively.
In particular, $(\bfphi_n)_{n\in\bbF_5}$ is a UNTF for $\bbR^2$ but is not equiangular.
In fact,
Gerzon's bound implies that no five-vector ETF for $\bbR^2$ exists,
and the coherence of $(\bfphi_n)_{n\in\bbF_5}$ is minimal~\cite{BenedettoK06,FickusJM18}.
Now let $\calR:=\set{1,2}\subseteq\bbF_5$,
and for each $r\in\calR$, define $(\bfphi_{r,n})_{n\in\bbF_5}$, $\bfphi_{r,n}:=\bfphi_{rn}$,
namely the finite sequences with corresponding synthesis matrices
\begin{align}
\label{eq.ex.SUNTF(2,5,2).1}
\bfPhi_1
&=\left[\begin{array}{ccccc}
\bfphi_{1,0}&\bfphi_{1,1}&\bfphi_{1,2}&\bfphi_{1,3}&\bfphi_{1,4}
\end{array}\right]
=\left[\begin{array}{ccccc}
\bfphi_{0}&\bfphi_{1}&\bfphi_{2}&\bfphi_{3}&\bfphi_{4}
\end{array}\right]
=\left[\begin{array}{rrrrr}
1&u&v&v&u\\
0&\phm U&\phm V&-V&-U
\end{array}\right],\\
\label{eq.ex.SUNTF(2,5,2).2}
\bfPhi_2
&=\left[\begin{array}{ccccc}
\bfphi_{2,0}&\bfphi_{2,1}&\bfphi_{2,2}&\bfphi_{2,3}&\bfphi_{2,4}
\end{array}\right]
=\left[\begin{array}{ccccc}
\bfphi_{0}&\bfphi_{2}&\bfphi_{4}&\bfphi_{1}&\bfphi_{3}
\end{array}\right]
=\left[\begin{array}{rrrrr}
1&v&u&u&v\\
0&\phm V&-U&\phm U&-V
\end{array}\right].
\end{align}
Each \smash{$(\bfphi_{r,n})_{n\in\bbF_5}$} is thus a permutation of the UNTF \smash{$(\bfphi_n)_{n\in\bbF_5}$} for $\bbR^2$,
and so is itself a UNTF for $\bbR^2$.
Adding the entrywise squared moduli of their Gram matrices reveals that these two UNTFs are moreover stratified (Definition~\ref{def.SUNTF}) with
$\varLambda=\frac12(u^2+v^2)=\frac{3}{8}$:
\begin{equation*}
\abs{\bfPhi_1^*\bfPhi_1^{}}^2+\abs{\bfPhi_2^*\bfPhi_2^{}}^2
=\left[\begin{array}{ccccc}
  1&u^2&v^2&v^2&u^2\\
u^2&  1&u^2&v^2&v^2\\
v^2&u^2&  1&u^2&v^2\\
v^2&v^2&u^2&  1&u^2\\
u^2&v^2&v^2&u^2&  1
\end{array}\right]
+\left[\begin{array}{ccccc}
  1&v^2&u^2&u^2&v^2\\
v^2&  1&v^2&u^2&u^2\\
u^2&v^2&  1&v^2&u^2\\
u^2&u^2&v^2&  1&v^2\\
v^2&u^2&u^2&v^2&  1
\end{array}\right]
=2\left[\begin{array}{ccccc}
         1&\varLambda&\varLambda&\varLambda&\varLambda\\
\varLambda&         1&\varLambda&\varLambda&\varLambda\\
\varLambda&\varLambda&         1&\varLambda&\varLambda\\
\varLambda&\varLambda&\varLambda&         1&\varLambda\\
\varLambda&\varLambda&\varLambda&\varLambda&         1
\end{array}\right].
\end{equation*}
(Ignoring diagonal entries,
this is analogous to how a flat plain can arise from a stack of uneven strata of rock,
inspiring our terminology.)
Thus, \smash{$\bigparen{(\bfphi_{r,n})_{n\in\bbF_5}}_{r\in\set{1,2}}$} is a $\SUNTF_\bbR(2,5,2)$.
\end{example}

In the following result,
we construct ECTFFs from direct sums of the layers of a SUNTF:

\begin{theorem}
\label{thm.SUNTF to ECTFF}
Let $\bbF$ be $\bbR$ or $\bbC$ and let $\calN$ and $\calR$ be nonempty finite sets.
For each $r\in\calR$,
let $(\bfphi_{r,n})_{n\in\calN}$ be a finite sequence of unit vectors in a Euclidean space $\calE_r$ over $\bbF$.
Let \smash{$\widehat{\calE}:=\oplus_{r\in\calR}\calE_r$}.
For each $n\in\calN$,
let \smash{$\widehat{\calF}_n$} be the image of \smash{$\widehat{\bfPhi}_n:\bbF^\calR\rightarrow\widehat{\calE}$},
$\smash{\widehat{\bfPhi}_n\bfx:=\oplus_{r\in\calR}\bfphi_{r,n}\bfx(r)}$.
Then:
\begin{enumerate}
\alphi

\item
Each \smash{$\widehat{\bfPhi}_n$} is linear and each \smash{$\widehat{\bfPhi}_{n_1}^{*}\widehat{\bfPhi}_{n_2}^{}$} is a diagonal matrix in \smash{$\bbF^{\calR\times\calR}$} with
\smash{$\ip{\bfphi_{r,n_1}}{\bfphi_{r,n_2}}$} as its $r$th diagonal entry.
In particular, each \smash{$\widehat{\bfPhi}_n$} is an isometry.

\item
\smash{$\max_{n_1\neq n_2}\norm{\widehat{\bfPhi}_{n_1}^{*}\widehat{\bfPhi}_{n_2}^{}}_2
=\max_{r\in\calR}\max_{n_1\neq n_2}\abs{\ip{\bfphi_{r,n_1}}{\bfphi_{r,n_2}}}$},
provided $\#(\calN)>1$.

\item
\smash{$(\widehat{\calF}_n)_{n\in\calN}$} is a TFF for \smash{$\widehat{\calE}$} if and only if each \smash{$(\bfphi_{r,n})_{n\in\calN}$} is a UNTF for \smash{$\calE_r$} where \smash{$\dim(\calE_r)$} is constant over $r\in\calR$.
When this occurs and \smash{$\dim(\calE_r)<\#(\calN)$} for any $r\in\calR$,
each \smash{$(\bfphi_{r,n})_{n\in\calN}$} has a Naimark complementary UNTF \smash{$(\check{\bfphi}_{r,n})_{n\in\calN}$} for some Euclidean space $\check{\calE}_r$,
and letting
\begin{equation*}
\widehat{\check{\bfPhi}}_n:\bbF^\calR\rightarrow\widehat{\check{\calE}}:=\oplus_{r\in\calR}\check{\calE}_r,
\quad
\widehat{\check{\bfPhi}}_n\bfx:=\oplus_{r\in\calR}\check{\bfphi}_{r,n}\bfx(r),
\end{equation*}
for each $n\in\calN$ yields a Naimark complement $(\widehat{\check{\bfPhi}}_n)_{n\in\calN}$ of
\smash{$(\widehat{\bfPhi}_n)_{n\in\calN}$}.

\item
\smash{$(\widehat{\calF}_n)_{n\in\calN}$} is EI if and only if there exists $\varLambda\geq0$ such that \smash{$\abs{\ip{\bfphi_{r,n_1}}{\bfphi_{r,n_2}}}^2=\varLambda$} for any $r\in\calR$ and any distinct \smash{$n_1,n_2\in\calN$}.

\item
\smash{$(\widehat{\calF}_n)_{n\in\calN}$} is an EITFF for \smash{$\widehat{\calE}$} if and only if each \smash{$(\bfphi_{r,n})_{n\in\calN}$} is an ETF for \smash{$\calE_r$} where \smash{$\dim(\calE_r)$} is constant over $r\in\calR$.

\item
\smash{$(\widehat{\calF}_n)_{n\in\calN}$} is EC if and only if
\smash{$\bigparen{(\bfphi_{r,n})_{n\in\calN}}_{r\in\calR}$} is stratified.

\item
\smash{$(\widehat{\calF}_n)_{n\in\calN}$} is an ECTFF for \smash{$\widehat{\calE}$} if and only if \smash{$\bigparen{(\bfphi_{r,n})_{n\in\calN}}_{r\in\calR}$} is a SUNTF for $(\calE_r)_{r\in\calR}$.
Here, \smash{$(\widehat{\calF}_n)_{n\in\calN}$} is an $\ECTFF_\bbF(DR,R,N)$ when \smash{$\bigparen{(\bfphi_{r,n})_{n\in\calN}}_{r\in\calR}$} is a $\SUNTF_\bbF(D,N,R)$.
\end{enumerate}
\end{theorem}

\begin{proof}[Proof.]
For each $r\in\calR$ and $n\in\calN$,
define
\smash{$\bfPhi_{r,n}:\bbF\rightarrow\calE_r$},
\smash{$\bfPhi_{r,n}x:=\bfphi_{r,n}x$} and let \smash{$\calF_{r,n}$} be its one-dimensional image.
Then \smash{$\widehat{\bfPhi}_n=\oplus_{r\in\calR}\bfPhi_{r,n}$} provided \smash{$\bbF^\calR$} is identified with \smash{$\oplus_{r\in\calR}\bbF$}.
For any $r\in\calR$ and $n_1,n_2\in\calN$, the map \smash{$\bfPhi_{r,n_1}^{*}\bfPhi_{r,n_2}^{}:\bbF\rightarrow\bbF$} is represented by the $1\times 1$ matrix with entry \smash{$\ip{\bfphi_{r,n_1}}{\bfphi_{r,n_2}}$}.
Claims (a)--(e) then follow from Lemma~\ref{lem.FF direct sums}(a)--(e), respectively.
We leave those details to the interested reader and focus our remaining attention on (f) and (g).
For (f), it suffices to let $R:=\#(\calR)$,
recall the definitions of EC and stratified (Definition~\ref{def.SUNTF}),
and note that (a) gives
\begin{equation}
\label{eq.pf.thm.SUNTF to ECTFF.1}
\tfrac{1}{R}\norm{\widehat{\bfPhi}_{n_1}^{*}\widehat{\bfPhi}_{n_2}^{}}_\Fro^2
=\tfrac{1}{R}\sum_{r_1\in\calR}\sum_{r_2\in\calR}\abs{(\widehat{\bfPhi}_{n_1}^{*}\widehat{\bfPhi}_{n_2}^{})(r_1,r_2)}^2
=\tfrac{1}{R}\sum_{r\in\calR}\abs{\ip{\bfphi_{r,n_1}}{\bfphi_{r,n_2}}}^2
\end{equation}
for any $n_1,n_2\in\calN$.
Combining (c) and (f) then gives (g).
\end{proof}

We remark that if \smash{$\bigparen{(\bfphi_{r,n})_{n\in\calN}}_{r\in\calR}$} is stratified then~\eqref{eq.pf.thm.SUNTF to ECTFF.1} gives that the constant $\varLambda$ of Definition~\ref{def.SUNTF} is the chordal constant of the resulting EC subspaces.
For such \smash{$\bigparen{(\bfphi_{r,n})_{n\in\calN}}_{r\in\calR}$},
we thus refer to this number $\varLambda$ as the chordal constant of \smash{$\bigparen{(\bfphi_{r,n})_{n\in\calN}}_{r\in\calR}$}.

\begin{example}
\label{ex.ECTFF(4,2,5)}
The $\SUNTF_\bbR(2,5,2)$ of Example~\ref{ex.SUNTF(2,5,2)} yields an $\ECTFF_\bbR(4,2,5)$ via Theorem~\ref{thm.SUNTF to ECTFF},
that is, an optimal packing of five planes in $\bbR^4$ with respect to the chordal distance.
For each $n\in\bbF_5$, the corresponding rank-$2$ isometry \smash{$\widehat{\bfPhi}_n$} is the block diagonal matrix with $\bfphi_{1,n}$ and $\bfphi_{2,n}$ as its diagonal blocks:
\begin{align}
\nonumber
\widehat{\bfPhi}
&=\left[\begin{array}{ccccc}
\widehat{\bfPhi}_0&\widehat{\bfPhi}_1&\widehat{\bfPhi}_2&\widehat{\bfPhi}_3&\widehat{\bfPhi}_4
\end{array}\right]\\
\nonumber
&=\left[\begin{array}{cc|cc|cc|cc|cc}
\bfphi_{1,0}&\bfzero&\bfphi_{1,1}&\bfzero&\bfphi_{1,2}&\bfzero&\bfphi_{1,3}&\bfzero&\bfphi_{1,4}&\bfzero\\
\bfzero&\bfphi_{2,0}&\bfzero&\bfphi_{2,1}&\bfzero&\bfphi_{2,2}&\bfzero&\bfphi_{2,3}&\bfzero&\bfphi_{2,4}
\end{array}\right]\\
\label{eq.ex.ECTFF(4,2,5).1}
&=\left[\begin{array}{rr|rr|rr|rr|rr}
\phm 1&     0\phm&     u&     0\phm&     v& 0\phm& v&     0\phm& u& 0\phm\\
     0&     0\phm&\phm U&     0\phm&\phm V& 0\phm&-V&     0\phm&-U& 0\phm\\
     0&\phm 1\phm&     0&     v\phm&     0& u\phm& 0&     u\phm& 0& v\phm\\
     0&     0\phm&     0&\phm V\phm&     0&-U\phm& 0&\phm U\phm& 0&-V\phm
\end{array}\right].
\end{align}
Every block of the associated fusion Gram matrix is thus diagonal:
\begin{equation*}
\widehat{\bfPhi}^*\widehat{\bfPhi}
=\left[\begin{array}{ccccc}
\widehat{\bfPhi}_{0}^{*}\widehat{\bfPhi}_{0}^{}&\widehat{\bfPhi}_{0}^{*}\widehat{\bfPhi}_{1}^{}&\widehat{\bfPhi}_{0}^{*}\widehat{\bfPhi}_{2}^{}&\widehat{\bfPhi}_{0}^{*}\widehat{\bfPhi}_{3}^{}&\widehat{\bfPhi}_{0}^{*}\widehat{\bfPhi}_{4}^{}\\
\widehat{\bfPhi}_{1}^{*}\widehat{\bfPhi}_{0}^{}&\widehat{\bfPhi}_{1}^{*}\widehat{\bfPhi}_{1}^{}&\widehat{\bfPhi}_{1}^{*}\widehat{\bfPhi}_{2}^{}&\widehat{\bfPhi}_{1}^{*}\widehat{\bfPhi}_{3}^{}&\widehat{\bfPhi}_{1}^{*}\widehat{\bfPhi}_{4}^{}\\
\widehat{\bfPhi}_{2}^{*}\widehat{\bfPhi}_{0}^{}&\widehat{\bfPhi}_{2}^{*}\widehat{\bfPhi}_{1}^{}&\widehat{\bfPhi}_{2}^{*}\widehat{\bfPhi}_{2}^{}&\widehat{\bfPhi}_{2}^{*}\widehat{\bfPhi}_{3}^{}&\widehat{\bfPhi}_{2}^{*}\widehat{\bfPhi}_{4}^{}\\
\widehat{\bfPhi}_{3}^{*}\widehat{\bfPhi}_{0}^{}&\widehat{\bfPhi}_{3}^{*}\widehat{\bfPhi}_{1}^{}&\widehat{\bfPhi}_{3}^{*}\widehat{\bfPhi}_{2}^{}&\widehat{\bfPhi}_{3}^{*}\widehat{\bfPhi}_{3}^{}&\widehat{\bfPhi}_{3}^{*}\widehat{\bfPhi}_{4}^{}\\
\widehat{\bfPhi}_{4}^{*}\widehat{\bfPhi}_{0}^{}&\widehat{\bfPhi}_{4}^{*}\widehat{\bfPhi}_{1}^{}&\widehat{\bfPhi}_{4}^{*}\widehat{\bfPhi}_{2}^{}&\widehat{\bfPhi}_{4}^{*}\widehat{\bfPhi}_{3}^{}&\widehat{\bfPhi}_{4}^{*}\widehat{\bfPhi}_{4}^{}
\end{array}\right]
=\left[\begin{array}{cc|cc|cc|cc|cc}
1&0&u&0&v&0&v&0&u&0\\
0&1&0&v&0&u&0&u&0&v\\
\hline
u&0&1&0&u&0&v&0&v&0\\
0&v&0&1&0&v&0&u&0&u\\
\hline
v&0&u&0&1&0&u&0&v&0\\
0&u&0&v&0&1&0&v&0&u\\
\hline
v&0&v&0&u&0&1&0&u&0\\
0&u&0&u&0&v&0&1&0&v\\
\hline
u&0&v&0&v&0&u&0&1&0\\
0&v&0&u&0&u&0&v&0&1
\end{array}\right].
\end{equation*}
Moreover, for any $n_1,n_2\in\bbF_5$,
the two diagonal entries of
\smash{$\widehat{\bfPhi}_{n_1}^{*}\widehat{\bfPhi}_{n_2}^{}$} are
\smash{$\ip{\bfphi_{1,n_1}}{\bfphi_{1,n_2}}$} and
\smash{$\ip{\bfphi_{2,n_1}}{\bfphi_{2,n_2}}$},
implying
\smash{$\norm{\widehat{\bfPhi}_{n_1}^{*}\widehat{\bfPhi}_{n_2}^{}}_\Fro^2
=\abs{\ip{\bfphi_{1,n_1}}{\bfphi_{1,n_2}}}^2+\abs{\ip{\bfphi_{2,n_1}}{\bfphi_{2,n_2}}}^2$}.
That is, the fusion Gram matrix \smash{$\widehat{\bfPhi}^{*}\widehat{\bfPhi}^{}$} is obtained by perfectly shuffling the $2\times 2$ block diagonal array that has $\bfPhi_1^*\bfPhi_1^{}$ and $\bfPhi_2^*\bfPhi_2^{}$ as its $5\times5$ diagonal blocks into a $5\times 5$ array of $2\times 2$ diagonal blocks.
In particular, the fact that \smash{$\bigparen{(\bfphi_{r,n})_{n\in\bbF_5}}_{r\in\set{1,2}}$} is stratified implies that the $2$-dimensional images of these isometries are EC.
They moreover form a TFF for $\bbR^4$ since the associated fusion frame operator is block diagonal with $\bfPhi_{r}^{}\bfPhi_{r}^{*}=\frac{5}{2}\brI$ as its $r$th diagonal block.
\end{example}

In Section~4,
we explain how the $\SUNTF_\bbR(2,5,2)$ of Example~\ref{ex.SUNTF(2,5,2)} is equivalent to one whose layers are harmonic,
how the latter equates to a difference family for a finite abelian group,
and how the ECTFF that arises from a harmonic SUNTF via Theorem~\ref{thm.SUNTF to ECTFF} is equivalent to one that arises from the corresponding difference family via an extension~\cite{FickusMW25} of a method~\cite{King16} for constructing ECTFFs from \textit{semiregular divisible difference sets}.
For context, a (complex, equi-isoclinic) $\EITFF_\bbC(4,2,5)$ arises from a (quaternionic) $\ETF_\bbH(2,5)$~\cite{CohnKM16,EtTaoui20,Waldron25} and from the block-circulant approach of~\cite{FickusIJM23}.
That said, no $\EITFF_\bbC(4,2,5)$ arises via Theorem~\ref{thm.SUNTF to ECTFF},
as that would require a $\SUNTF_\bbC(2,5,2)$ each of whose layers is an $\ETF_\bbC(2,5)$,
which is precluded by the Gerzon bound.
Notably, Theorem~\ref{thm.SUNTF to ECTFF} only yields isometries \smash{$(\widehat{\bfPhi}_n)_{n\in\calN}$} whose cross-Gram matrices are diagonal and thus commute.
Some previously known methods for constructing TFFs also yield isometries whose cross-Gram matrices commute,
including those based on filter banks~\cite{ChebiraFM11} and harmonic frames~\cite{King16,FickusMW25}.
In the following result,
we show that complex isometries whose cross-Gram matrices commute are equivalent to ones that arise from Theorem~\ref{thm.SUNTF to ECTFF} via a constant sequence of right-unitaries.

\begin{theorem}
\label{thm.commuting}
Let $\calN$ and $\widetilde{\calR}$ be nonempty finite sets,
let \smash{$\widetilde{\calE}$} be a complex Euclidean space of dimension \smash{$\widetilde{D}$} and,
for each $n\in\calN$, let \smash{$\widetilde{\bfPhi}_n:\bbC^{\widetilde{\calR}}\rightarrow\widetilde{\calE}$} be an isometry.
Then
\begin{equation}
\label{eq.cross-Grams commute}
\widetilde{\bfPhi}_{n_1}^*\widetilde{\bfPhi}_{n_2}^{}\widetilde{\bfPhi}_{n_3}^*\widetilde{\bfPhi}_{n_4}^{}
=\widetilde{\bfPhi}_{n_3}^*\widetilde{\bfPhi}_{n_4}^{}\widetilde{\bfPhi}_{n_1}^*\widetilde{\bfPhi}_{n_2}^{}
\end{equation}
for any $n_1,n_2,n_3,n_4\in\calN$ if and only if there exist unitaries \smash{$\bfU:\widehat{\calE}\rightarrow\widetilde{\calE}$} and \smash{$\bfV\in\bbC^{\widetilde{\calR}\times\calR}$} and isometries \smash{$(\widehat{\bfPhi}_n)_{n\in\calN}$} arising from \smash{$\bigparen{(\bfphi_{r,n})_{n\in\calN}}_{r\in\calR}$} via Theorem~\ref{thm.SUNTF to ECTFF} such that $\widetilde{\bfPhi}_n=\bfU\widehat{\bfPhi}_n\bfV^*$ for any $n\in\calN$.
Moreover,
when this occurs, letting $N:=\#(\calN)$, $R:=\#(\calR)$, \smash{$\widetilde{D}:=\dim(\widetilde{\calE})$} and $\widetilde{\calF}_n:=\im(\widetilde{\bfPhi}_n)$ for each $n\in\calN$, the following hold:

\begin{enumerate}
\alphi

\item
If $N>1$,
the spectral coherence of \smash{$(\widetilde{\bfPhi}_n)_{n\in\calN}$} is at least the minimum coherence of $N$ unit vectors in $\bbC^{\lfloor\frac{\widetilde{D}}{R}\rfloor}$,
achieved if each \smash{$(\bfphi_{r,n})_{n\in\calN}$} has that coherence.
\smallskip

\item
\smash{$(\widetilde{\calF}_n)_{n\in\calN}$} is a TFF or EITFF for \smash{$\widetilde{\calE}$} if and only if each \smash{$(\bfphi_{r,n})_{n\in\calN}$} is a $\UNTF_\bbC(D,N)$ or $\ETF_\bbC(D,N)$, respectively,
and is an ECTFF for \smash{$\widetilde{\calE}$} if and only if \smash{$\bigparen{(\bfphi_{r,n})_{n\in\calN}}_{r\in\calR}$} is a $\SUNTF_\bbC(D,N,R)$,
where $R$ necessarily divides \smash{$\widetilde{D}$} and \smash{$D:=\frac{\widetilde{D}}{R}$}.

\end{enumerate}
\end{theorem}

\begin{proof}[Proof.]
($\Leftarrow$)
For any $n_1,n_2,n_3,n_4\in\calN$,
\begin{equation*}
\widetilde{\bfPhi}_{n_1}^*\widetilde{\bfPhi}_{n_2}^{}\widetilde{\bfPhi}_{n_3}^*\widetilde{\bfPhi}_{n_4}^{}
=\bfV\widehat{\bfPhi}_{n_1}^{*}\bfU^*\bfU\widehat{\bfPhi}_{n_2}^{}\bfV^*
\bfV\widehat{\bfPhi}_{n_3}^{*}\bfU^*\bfU\widehat{\bfPhi}_{n_4}^{}\bfV^*
=\bfV\widehat{\bfPhi}_{n_1}^{*}\widehat{\bfPhi}_{n_2}^{}\widehat{\bfPhi}_{n_3}^{*}\widehat{\bfPhi}_{n_4}^{}\bfV^*.
\end{equation*}
By Theorem~\ref{thm.SUNTF to ECTFF}(a),
$\widehat{\bfPhi}_{n_1}^{*}\widehat{\bfPhi}_{n_2}^{}$
and $\widehat{\bfPhi}_{n_3}^{*}\widehat{\bfPhi}_{n_4}^{}$ are diagonal complex matrices and thus commute, implying
\begin{equation*}
\widetilde{\bfPhi}_{n_1}^*\widetilde{\bfPhi}_{n_2}^{}\widetilde{\bfPhi}_{n_3}^*\widetilde{\bfPhi}_{n_4}^{}
=\bfV\widehat{\bfPhi}_{n_1}^{*}\widehat{\bfPhi}_{n_2}^{}\widehat{\bfPhi}_{n_3}^{*}\widehat{\bfPhi}_{n_4}^{}\bfV^*
=\bfV\widehat{\bfPhi}_{n_3}^{*}\widehat{\bfPhi}_{n_4}^{}\widehat{\bfPhi}_{n_1}^{*}\widehat{\bfPhi}_{n_2}^{}\bfV^*
=\widetilde{\bfPhi}_{n_3}^*\widetilde{\bfPhi}_{n_4}^{}\widetilde{\bfPhi}_{n_1}^*\widetilde{\bfPhi}_{n_2}^{}.
\end{equation*}
($\Rightarrow$)
For any $n_1,n_2\in\calN$,
taking $n_3=n_2$ and $n_4=n_1$ in~\eqref{eq.cross-Grams commute} gives that
\smash{$\widetilde{\bfPhi}_{n_1}^*\widetilde{\bfPhi}_{n_2}^{}$} commutes with its adjoint and so is normal.
Thus, \smash{$\set{\widetilde{\bfPhi}_{n_1}^*\widetilde{\bfPhi}_{n_2}^{} : n_1,n_2\in\calN}$} consists of commuting normal operators on \smash{$\bbC^{\widetilde{\calR}}$},
implying there exists an ONB $(\bfv_r)_{r\in\calR}$ for \smash{$\bbC^{\widetilde{\calR}}$} such that each $\bfv_r$ is an eigenvector for each \smash{$\widetilde{\bfPhi}_{n_1}^*\widetilde{\bfPhi}_{n_2}^{}$}.
(Here, $\calR$ can be any set whose cardinality equals that of \smash{$\widetilde{\calR}$}.
As detailed later on, this flexibility is convenient when, for example, each \smash{$\widetilde{\bfPhi}_{n_1}^*\widetilde{\bfPhi}_{n_2}^{}$} is circulant over a finite abelian group $\widetilde{\calR}$, as it lets us choose $\calR$ to be its Pontryagin dual.)
The synthesis matrix \smash{$\bfV\in\bbC^{\widetilde{\calR}\times\calR}$} of this ONB is thus unitary and,
for any $n_1,n_2\in\calN$,
\smash{$\bfLambda_{n_1,n_2}
:=\bfV^*\widetilde{\bfPhi}_{n_1}^*\widetilde{\bfPhi}_{n_2}^{}\bfV\in\bbC^{\calR\times\calR}$} is diagonal.
Now let \smash{$\widehat{\bfG}:=\widehat{\bfV}^*\widetilde{\bfPhi}^*\widetilde{\bfPhi}\widehat{\bfV}$}
where $\widehat{\bfV}$ is a block diagonal matrix that has $\bfV$ as each of its diagonal blocks and
\smash{$\widetilde{\bfPhi}$} is the fusion synthesis map of \smash{$(\widetilde{\bfPhi}_n)_{n\in\calN}$}.
Thus, \smash{$\widehat{\bfG}$} is PSD and is a block matrix,
having $\bfLambda_{n_1,n_2}$ as its $(n_1,n_2)$th block where $\bfLambda_{n,n}=\brI$ for any $n\in\calN$.

Conjugating \smash{$\widehat{\bfG}$} by the (perfect shuffle) permutation matrix $\bfP:\bbC^{\calN\times\calR}\rightarrow\bbC^{\calR\times\calN}$ defined by having $\bfP\bfdelta_{n,r}:=\bfdelta_{r,n}$ thus gives a block-diagonal PSD matrix \smash{$\bfG:=\bfP\widehat{\bfG}\bfP^*\in\bbC^{(\calR\times\calN)\times(\calR\times\calN)}$} whose $r$th diagonal block is the matrix \smash{$\bfG_r\in\bbC^{\calN\times\calN}$} defined by \smash{$\bfG_r(n_1,n_2):=\bfLambda_{n_1,n_2}(r,r)$}.
Since $\bfG$ is block-diagonal and PSD,
each of its $R$ diagonal blocks $\bfG_r$ is PSD,
and is thus the Gram matrix $\bfPhi_{r}^{*}\bfPhi_{r}^{}$ of some finite sequence $(\bfphi_{r,n})_{n\in\calN}$ of vectors in any given complex Euclidean space $\calE_r$ of dimension at least $\rank(\bfG_r)$.
Here,
\begin{equation*}
\sum_{r\in\calR}\rank(\bfG_r)
=\rank(\bfG)
=\rank(\widehat{\bfG})
=\rank(\widetilde{\bfPhi}^*\widetilde{\bfPhi})
=\rank(\widetilde{\bfPhi})
=\dim\bigparen{\Span(\widetilde{\calF}_n)_{n\in\calN}}
\leq\dim(\widetilde{\calE}),
\end{equation*}
permitting us to without loss of generality choose such \smash{$(\calE_r)_{r\in\calR}$} so that \smash{$\sum_{r\in\calR}\dim(\calE_r)=\dim(\widetilde{\calE})$},
and so $\widehat{\calE}:=\oplus_{r\in\calR}\calE_r$ has $\dim(\widehat{\calE})=\dim(\widetilde{\calE})$.
For any $r\in\calR$ and $n_1,n_2\in\calN$,
\begin{equation*}
\ip{\bfphi_{r,n_1}}{\bfphi_{r,n_2}}
=(\bfPhi_{r}^{*}\bfPhi_{r}^{})(n_1,n_2)
=\bfG_r(n_1,n_2)
=\bfLambda_{n_1,n_2}(r,r).
\end{equation*}
Each $\bfphi_{r,n}$ is thus unit norm.
Define \smash{$(\widehat{\bfPhi}_n)_{n\in\calN}$} from
\smash{$\bigparen{(\bfphi_{r,n})_{n\in\calN}}_{r\in\calR}$} \`{a} la Theorem~\ref{thm.SUNTF to ECTFF}.
By Theorem~\ref{thm.SUNTF to ECTFF}(a),
each \smash{$\widehat{\bfPhi}_n$} is an isometry and each
\smash{$\widehat{\bfPhi}_{n_1}^{*}\widehat{\bfPhi}_{n_2}^{}$} is diagonal with
$(\widehat{\bfPhi}_{n_1}^{*}\widehat{\bfPhi}_{n_2}^{})(r,r)
=\ip{\bfphi_{r,n_1}}{\bfphi_{r,n_2}}
=\bfLambda_{n_1,n_2}(r,r)$
for any $r\in\calR$.
Since each \smash{$\bfLambda_{n_1,n_2}$} is diagonal,
this implies
$\widehat{\bfPhi}_{n_1}^{*}\widehat{\bfPhi}_{n_2}^{}
=\bfLambda_{n_1,n_2}
=\bfV^*\widetilde{\bfPhi}_{n_1}^*\widetilde{\bfPhi}_{n_2}^{}\bfV$
for any $n_1,n_2\in\calN$.
By Lemma~\ref{lem.equivalence}, there exists a unitary
\smash{$\bfU:\widehat{\calE}\rightarrow\widetilde{\calE}$}
such that $\widetilde{\bfPhi}_n=\bfU\widehat{\bfPhi}_n\bfV^*$ for any $n\in\calN$.

(a) Since $\widetilde{\bfPhi}_n=\bfU\widehat{\bfPhi}_n\bfV^*$ for any $n\in\calN$,
\begin{equation*}
\norm{\widetilde{\bfPhi}_{n_1}^{*}\widetilde{\bfPhi}_{n_2}^{}}_2
=\norm{\bfV\widehat{\bfPhi}_{n_1}^{*}\bfU^*\bfU\widehat{\bfPhi}_{n_2}^{}\bfV^*}_2
=\norm{\widehat{\bfPhi}_{n_1}^{*}\widehat{\bfPhi}_{n_2}^{}}_2
\end{equation*}
for any $n_1,n_2\in\calN$.
Provided $N>1$,
this implies via Theorem~\ref{thm.SUNTF to ECTFF}(b) that
\begin{equation}
\label{eq.pf.cor.coherence of commuting.1}
\max_{n_1\neq n_2}\norm{\widetilde{\bfPhi}_{n_1}^{*}\widetilde{\bfPhi}_{n_2}^{}}_2
=\max_{n_1\neq n_2}\norm{\widehat{\bfPhi}_{n_1}^{*}\widehat{\bfPhi}_{n_2}^{}}_2
=\max_{r\in\calR}\max_{n_1\neq n_2}\abs{\ip{\bfphi_{r,n_1}}{\bfphi_{r,n_2}}}.
\end{equation}
Recall each $(\bfphi_{r,n})_{n\in\calN}$ is a finite sequence of $N$ unit norm vectors in some complex Euclidean space $\calE_r$.
Letting $D_r:=\dim(\calE_r)$ for each $r\in\calR$,
\begin{equation}
\label{eq.pf.cor.coherence of commuting.2}
\widetilde{D}
=\dim(\widetilde{\calE})
=\dim(\widehat{\calE})
=\dim(\oplus_{r\in\calR}\calE_r)
=\sum_{r\in\calR}D_r
\geq R\min_{r\in\calR}D_r
\end{equation}
and so $\lfloor\frac{\widetilde{D}}{R}\rfloor\geq\min_{r\in\calR}D_r$.
In general, if $\bbF$ is $\bbR$ or $\bbC$ and $D$ and $N$ are positive integers with $N>1$,
the coherence \smash{$\max_{n_1\neq n_2}\abs{\ip{\bfphi_{n_1}}{\bfphi_{n_2}}}$} of $N$ unit vectors in \smash{$\bbF^D$} is a real-valued continuous function over a compact set (the Cartesian product of $N$ copies of the unit sphere in $\bbF^D$) and so has a minimum, which we denote here as $\mu_\bbF(D,N)$.
Note $\mu_\bbF(D+1,N)\leq\mu_\bbF(D,N)$ since $\bbF^D$ isometrically embeds into $\bbF^{D+1}$.
(This inequality is known to be strict when $\bbF=\bbR$ and $D<N$~\cite{FickusJM18}.)
Since coherence is preserved by equivalence,
\eqref{eq.pf.cor.coherence of commuting.1} continues as
\begin{equation*}
\max_{n_1\neq n_2}\norm{\widetilde{\bfPhi}_{n_1}^{*}\widetilde{\bfPhi}_{n_2}^{}}_2
=\max_{r\in\calR}\max_{n_1\neq n_2}\abs{\ip{\bfphi_{r,n_1}}{\bfphi_{r,n_2}}}
\geq\max_{r\in\calR}\mu_\bbC(D_r,N)
=\mu_\bbC(\min_{r\in\calR}D_r,N)
\geq\mu_\bbC(\lfloor\tfrac{\widetilde{D}}{R}\rfloor,N),
\end{equation*}
as claimed.
Moreover, equality holds above if
\smash{$\max_{n_1\neq n_2}\abs{\ip{\bfphi_{r,n_1}}{\bfphi_{r,n_2}}}=\mu_\bbC(\lfloor\tfrac{\widetilde{D}}{R}\rfloor,N)$}
for each $r\in\calR$.

(b) Since \smash{$(\widetilde{\bfPhi}_n)_{n\in\calN}$} is equivalent to \smash{$(\widehat{\bfPhi}_n)_{n\in\calN}$},
\smash{$(\widetilde{\calF}_n)_{n\in\calN}$} is EI, EC or a TFF for \smash{$\widetilde{\calE}$} if and only if \smash{$(\widehat{\calF}_n)_{n\in\calN}$} is EI, EC or a TFF for \smash{$\widehat{\calE}$}, respectively.
Here, note that by~\eqref{eq.pf.cor.coherence of commuting.2},
\smash{$\dim(\calE_r)$} is constant over $r\in\calR$ if and only if $R$ divides \smash{$\widetilde{D}$} and $\dim(\calE_r)=D$ for any $r\in\calR$.
As such, all remaining claims follow from (c), (e) and (g) of Theorem~\ref{thm.SUNTF to ECTFF}.
\end{proof}

\begin{remark}
Let \smash{$(\widetilde{\bfPhi}_n)_{n=1}^{5}$} be a finite sequence of $N=5$ isometries into a complex space \smash{$\widetilde{\calE}$} of dimension \smash{$\widetilde{D}=4$},
each of rank $R=2$.
By~\eqref{eq.Welch bounds},
its squared spectral coherence
\smash{$\mu_{\rms}^2$}
is at least
\smash{$\frac{NR-\widetilde{D}}{\widetilde{D}(N-1)}
=\frac{5(2)-4}{4(5-1)}
=\frac38$}.
Moreover, as noted above,
an $\EITFF_\bbC(4,2,5)$ exists,
and any corresponding isometries achieve this lower bound.
That said, if we further require that the cross-Gram matrices of such isometries commute,
then Theorem~\ref{thm.commuting}(a) gives that \smash{$\mu_{\rms}^2$} is bounded below by the minimal squared coherence of five unit vectors in a complex space of dimension \smash{$D=\lfloor\frac{\widetilde{D}}{R}\rfloor=\lfloor\frac{4}{2}\rfloor=2$}.
In cases like this where $\bbF$, $D$ and $N$ do not satisfy the Gerzon bound~\eqref{eq.Gerzon},
the squared coherence of $N$ unit norm vectors $(\bfphi_n)_{n\in\calN}$ more generally satisfies the \textit{orthoplex bound}~\cite{ConwayHS96,BodmannH16},
namely that $\max_{n_1\neq n_2}\abs{\ip{\bfphi_{n_1}}{\bfphi_{n_2}}}^2\geq\frac1D$.
In particular, if the cross-Gram matrices of \smash{$(\widetilde{\bfPhi}_n)_{n=1}^{5}$} commute then
\smash{$\mu_{\rms}^2\geq\frac12$}.
Notably, for any prime power $Q$, \cite{BodmannH16} constructs a $\UNTF_\bbC(D,N)$ where $D=Q+1$ and $N=D^2+1$ that meets the orthoplex bound.
As a degenerate case of this with $Q=1$, consider
\begin{equation}
\label{eq.UNTF(2,5)}
\bfPhi
=\left[\begin{array}{ccccc}
\bfphi_1&\bfphi_2&\bfphi_3&\bfphi_4&\bfphi_5
\end{array}\right]
=\tfrac1{\sqrt{2}}\left[\begin{array}{cclll}
\sqrt{2}&0&1&\omega&\omega^2\\
0&\sqrt{2}&1&\omega^2&\omega
\end{array}\right],
\quad\omega=\exp(\tfrac{2\pi\rmi}{3}).
\end{equation}
The isometries $(\widehat{\bfPhi}_n)_{n=1}^{5}$ constructed via Theorem~\ref{thm.SUNTF to ECTFF} from two copies of this UNTF
thus yield a $\TFF_\bbC(4,2,5)$ that is neither EI nor EC but whose spectral coherence $\mu_\rms$ is minimal subject to the constraint of having commuting cross-Gram matrices.
\end{remark}

\begin{remark}
The ``only if'' result of Theorem~\ref{thm.commuting} does not directly generalize to the real setting.
In particular, there exist isometries
$(\widetilde{\bfPhi}_n)_{n\in\calN}$ into a real Euclidean space whose cross-Gram matrices commute but for which the remaining conclusions of Theorem~\ref{thm.commuting} do not hold.
For example, let \smash{$z=-\frac12+\frac{\sqrt{3}}{2}\rmi$} and consider
\begin{equation}
\label{eq.ETF(2,4)}
\bfPhi=\tfrac{1}{\sqrt{3}}\left[\begin{array}{cccc}
\sqrt{3}&                  1&        1&                   1\\
       0&\sqrt{2}\phantom{z}&\sqrt{2}z&\sqrt{2}\overline{z}
\end{array}\right],\quad
\bfPhi\bfPhi^*
=\left[\begin{array}{cc}2&0\\0&2\end{array}\right],
\quad
\bfPhi^*\bfPhi
=\tfrac1{\sqrt{3}}\left[\begin{array}{cccc}
\sqrt{3}&   \phm1&   \phm1&   \phm1\\
       1&\sqrt{3}&\phm\rmi&   -\rmi\\
       1&   -\rmi&\sqrt{3}&\phm\rmi\\
       1&\phm\rmi&   -\rmi&\sqrt{3}
\end{array}\right].
\end{equation}
Clearly, the columns of $\bfPhi$ form an $\ETF_\bbC(2,4)$,
namely the isometries of an $\EITFF_\bbC(2,1,4)$.
As mentioned earlier,
applying (the injective ring homomorphism) \smash{$u+\rmi v\mapsto\left[\begin{smallmatrix}u&-v\\v&\phm u\end{smallmatrix}\right]$} entrywise to these matrices produces the analogous objects for an $\EITFF_\bbR(4,2,4)$.
Each of its cross-Gram matrices is one of the following four matrices:
\begin{equation*}
\left[\begin{smallmatrix}1&0\\0&1\end{smallmatrix}\right],\quad
\tfrac1{\sqrt{3}}\left[\begin{smallmatrix}1&0\\0&1\end{smallmatrix}\right],\quad
\tfrac1{\sqrt{3}}\left[\begin{smallmatrix}0&-1\\1&\phm0\end{smallmatrix}\right],\quad
\tfrac1{\sqrt{3}}\left[\begin{smallmatrix}\phm0&1\\-1&0\end{smallmatrix}\right].
\end{equation*}
These matrices commute.
Indeed, they arise by applying a ring homomorphism to elements of a field.
However, in contrast to Theorem~\ref{thm.commuting},
these real isometries are not equivalent to ones that arise via Theorem~\ref{thm.SUNTF to ECTFF} from two $\ETF_\bbR(2,4)$,
as the Gerzon bound gives that these do not exist.
Of course, Theorem~\ref{thm.commuting} does apply if these real isometries are regarded as complex isometries with real entries,
and they are equivalent to ones that arise via Theorem~\ref{thm.SUNTF to ECTFF} from two $\ETF_\bbC(2,4)$.
In fact, they arise from the ETF of~\eqref{eq.ETF(2,4)} and its entrywise conjugate since $\left[\begin{smallmatrix}u&-v\\v&\phm u\end{smallmatrix}\right]
=\frac1{\sqrt{2}}\left[\begin{smallmatrix}\phm1&1\\-\rmi&\rmi\end{smallmatrix}\right]
\left[\begin{smallmatrix}u+\rmi v&0\\0&u-\rmi v\end{smallmatrix}\right]
\frac1{\sqrt{2}}\left[\begin{smallmatrix}1&\phm\rmi\\1&-\rmi\end{smallmatrix}\right]$
for any $u,v\in\bbR$.
\end{remark}

\subsection{Filter banks}

We devote the remainder of this section to a special case of Theorem~\ref{thm.commuting} of particular mathematical and historical significance, namely where the cross-Gram matrices are circulant with respect to a finite abelian group.
As we shall discuss, this is intimately related to the classical concept of a filter bank and its associated polyphase matrix.

Let $\calG$ be a finite abelian group of order $G$ whose operation is denoted as addition.
For any $g_1\in\calG$, define a linear \textit{translation} operator $\brT^{g_1}$ on $\bbC^{\calG}$ by $\brT^{g_1}\bfdelta_{g_2}:=\bfdelta_{g_1+g_2}$ for each $g_2\in\calG$.
Thus, $g\mapsto\brT^g$ is a homomorphism.
Each $\brT^g$ maps the standard basis for $\bbC^\calG$ to another ONB and so is unitary, implying $(\brT^g)^*=(\brT^g)^{-1}=\brT^{-g}$.
It follows that $\calA:=\Span_\bbC(\brT^g)_{g\in\calG}$ is a commutative complex $*$-algebra.
We call its members \textit{filters (over $\calG$)}.

Since $\sum_{g\in\calG}\brT^g\bff(g)\bfdelta_0=\sum_{g\in\calG}\bfdelta_g\bff(g)=\bff$ for any $\bff\in\bbC^\calG$,
the (synthesis) map $\bff\mapsto\sum_{g\in\calG}\brT^g\bff(g)$ from $\bbC^\calG$ into $\calA$ has $\bfF\mapsto\bfF\bfdelta_0$ as an inverse.
For any $\bff\in\bbC^\calG$ and $\bfF\in\calA$,
we thus have $\bfF=\sum_{g\in\calG}\brT^g\bff(g)$ if and only if $\bff=\bfF\bfdelta_0$ is the \textit{impulse response} of $\bfF$.
(This implies $(\brT^g)_{g\in\calG}$ is linearly independent;
in fact, it is equal-norm orthogonal since $\ip{\brT^{g_1}}{\brT^{g_2}}_\Fro=\Tr(\brT^{g_2-g_1})$ for any $g_1,g_2\in\calG$.)
Since $\calA$ is a commutative $*$-algebra, this vector space isomorphism induces a commutative $*$-algebra on $\bbC^\calG$.
The induced product of $\bff_1,\bff_2\in\bbC^\calG$ is called their \textit{convolution} $\bff_1*\bff_2\in\bbC^\calG$.
Multiplying $\sum_{g_1\in\calG}\brT^{g_1}\bff_1(g_1)$ and $\sum_{g_2\in\calG}\brT^{g_2}\bff_2(g_2)$ reveals that
\begin{equation*}
(\bff_1*\bff_2)(g)
=\smash{\sum_{\substack{g_1,g_2\in\calG\\g_1+g_2=g}}}\bff_1(g_1)\bff_2(g_2)
=\sum_{g_1\in\calG}\bff_1(g_1)\bff_2(g_{}-g_1)
=\sum_{g_2\in\calG}\bff_1(g_{}-g_2)\bff_2(g_2).
\end{equation*}
The induced multiplicative identity is $\brI\bfdelta_0=\bfdelta_0$.
Taking the adjoint of $\sum_{g\in\calG}\brT^g\bff(g)$ gives that the induced involution of $\bff\in\bbC^\calG$ is its \textit{conjugate reversal} \smash{$\bff^{\star}\in\bbC^\calG$},
\smash{$\bff^{\star}(g):=\overline{\bff(-g)}$}.
If $\bfF,\bfX\in\calA$ have $\bff,\bfx\in\bbC^\calG$ as their respective impulse responses, $\bfF\bfx=\bfF\bfX\bfdelta_0=\bff*\bfx$.
Filters thus equate to \textit{convolutional operators} on $\bbC^\calG$.
Thus,
$(\brT^{g_1}\bfx)(g_2)
=(\bfdelta_{g_1}*\bfx)(g_2)
=\bfx(g_2-g_1)$
and
$\bfF(g_1,g_2)
=(\bff*\bfdelta_{g_2})(g_1)
=\bff(g_1-g_2)$
for any $g_1,g_2\in\calG$.
It follows that $\bfF\in\bbC^{\calG\times\calG}$ is a filter if and only if it is \textit{($\calG$-)circulant},
that is, $\bfF(g_1+g,g_2+g)=\bfF(g_1,g_2)$ for any $g_1,g_2,g\in\calG$.

Now let $\calG$ be a subgroup of a finite abelian group $\calH$ and let \smash{$(\widetilde{\bfphi}_n)_{n\in\calN}$} be a finite sequence in \smash{$\bbC^\calH$}.
For each $n\in\calN$, let \smash{$\widetilde{\bfPhi}_n$} be the synthesis map of \smash{$(\brT^g\widetilde{\bfphi}_n)_{g\in\calG}$} and let \smash{$\widetilde{\bfPhi}$} be the synthesis map of their concatenation \smash{$(\brT^g\widetilde{\bfphi}_n)_{(n,g)\in\calN\times\calG}$}.
This map \smash{$\widetilde{\bfPhi}$} and its adjoint \smash{$\widetilde{\bfPhi}^*$} are called the \textit{synthesis} and \textit{analysis filter banks} of \smash{$(\widetilde{\bfphi}_n)_{n\in\calN}$} (with respect to $\calG$), respectively.
Applying the latter to any \smash{$\bfy\in\bbC^\calH$} yields \smash{$\widetilde{\bfPhi}^*\bfy\in\bbC^{\calN\times\calG}$} where
\begin{equation*}
(\widetilde{\bfPhi}^*\bfy)(n,g)
=\ip{\brT^g\widetilde{\bfphi}_n}{\bfy}
=\sum_{h\in\calH}\overline{\widetilde{\bfphi}_n(h-g)}\bfy(h)
=(\widetilde{\bfphi}_n^{\star}*\bfy)(g)
\end{equation*}
for any $n\in\calN$ and $g\in\calG$,
and so $\widetilde{\bfPhi}^*\bfy=\oplus_{n\in\calN}\dsample(\widetilde{\bfphi}_n^{\star}*\bfy)$ where $\dsample$ is the \textit{downsampling} map from $\bbC^\calH$ into $\bbC^\calG$, $(\dsample\bfy)(g):=\bfy(g)$.
Similarly, $\widetilde{\bfPhi}(\bfx_n)_{n\in\calN}=\sum_{n\in\calN}\widetilde{\bfphi}_n*(\usample\bfx_n)$ for any $(\bfx_n)_{n\in\calN}\in\bbC^{\calN\times\calG}$ where $\usample=\dsample^*$ is the \textit{upsampling} map defined by $(\usample\bfx_n)(g)=\bfx_n(g)$ for each $g\in\calG$ and $(\usample\bfx_n)(h)=0$ for each $h\in\calH\backslash\calG$.
That is, $\widetilde{\bfPhi}(\bfx_n)_{n\in\calN}=\sum_{n\in\calN}\widetilde{\bfPhi}_n\bfx_n$ where each map $\bfx_n\mapsto\widetilde{\bfPhi}_n\bfx_n=\widetilde{\bfphi}_n*(\usample\bfx_n)$ is a \textit{channel} of the synthesis filter bank.
Such filter banks are fundamental to the study of \textit{wavelets}.
For any $n_1,n_2\in\calN$, the matrix \smash{$\widetilde{\bfPhi}_{n_1}^{*}\widetilde{\bfPhi}_{n_2}^{}\in\bbC^{\calG\times\calG}$} has
\begin{equation*}
(\widetilde{\bfPhi}_{n_1}^{*}\widetilde{\bfPhi}_{n_2}^{})(g_1,g_2)
=\ip{\brT^{g_1}\widetilde{\bfphi}_{n_1}}{\brT^{g_2}\widetilde{\bfphi}_{n_2}}
=\ip{\widetilde{\bfphi}_{n_1}}{\brT^{g_2-g_1}\widetilde{\bfphi}_{n_2}}
\end{equation*}
for any $g_1,g_2\in\calG$,
and so is $\calG$-circulant.
Since $\calA$ is commutative,
it follows that if \smash{$(\brT^g\widetilde{\bfphi}_n)_{g\in\calG}$} is orthonormal for each $n\in\calN$,
then \smash{$(\widetilde{\bfPhi}_n)_{n\in\calN}$} are isometries to which Theorem~\ref{thm.commuting} applies.
As we now explain, when this occurs, the matrices $\bfU$ and $\bfV$ of Theorem~\ref{thm.commuting} can be made explicit.

A \textit{character} of $\calG$ is a homomorphism $\gamma$ from $\calG$ into the complex unit circle $\bbT$.
The \textit{Pontryagin dual (group)} of $\calG$ is the set $\widehat{\calG}$ of all such $\gamma$ equipped with pointwise multiplication.
Its identity $1$ has $1(g):=1$ for each $g\in\calG$.
We identify each $\gamma\in\widehat{\calG}$ with a vector $\bfgamma\in\bbC^\calG$ whose $g$th entry $\bfgamma(g):=\gamma(g)$ is unimodular.
These vectors are orthogonal since, for any distinct \smash{$\gamma_1,\gamma_2\in\widehat{\calG}$},
the value $\ip{\bfgamma_1}{\bfgamma_2}=\sum_{g\in\calG}(\gamma_1^{-1}\gamma_2^{})(g)$ is preserved by multiplication by $(\gamma_1^{-1}\gamma_2^{})(g_0)\notin\set{0,1}$ for some $g_0\in\calG$.
Thus, \smash{$\#(\widehat{\calG})\leq G$} and $\bfGamma^*\bfGamma=G\brI$ where \smash{$\bfGamma\in\bbC^{\calG\times\widehat{\calG}}$}, $\bfGamma(g,\gamma):=\gamma(g)$ is the \textit{character table} of $\calG$,
namely the synthesis matrix of \smash{$(\bfgamma)_{\gamma\in\widehat{\calG}}$}.
Each such $\bfgamma$ is moreover an eigenvector for any filter $\sum_{g\in\calG}\brT^g\bff(g)$ with eigenvalue $(\bfGamma^*\bff)(\gamma)=\ip{\bfgamma}{\bff}$:
\begin{equation*}
(\bff*\bfgamma)(g)
=\sum_{g'\in\calG}\bff(g')\bfgamma(g-g')
=\bfgamma(g)\sum_{g'\in\calG}\overline{\bfgamma(g')}\bff(g')
=\bfgamma(g)(\bfGamma^*\bff)(\gamma)
\end{equation*}
for any $g\in\calG$.
Thus,
$\bfF\bfGamma=\bfGamma\diag(\bfGamma^*\bff)$ where $\diag(\bfGamma^*\bff)$ is the diagonal complex \smash{$\widehat{\calG}\times\widehat{\calG}$} matrix with $(\bfGamma^*\bff)(\gamma)$ as its $\gamma$th diagonal entry.

Notably, for any finite abelian group $\calG$,
one can construct an injective homomorphism \smash{$\sigma:\calG\rightarrow\widehat{\calG}$},
implying $\#(\widehat{\calG})=G$ and that $\sigma$ is an isomorphism.
This isomorphism is not natural in the categorical sense.
Rather, when $\calG=\bbZ_N$, we can take $\bigparen{\sigma(n)}(m):=\exp(\frac{2\pi\rmi mn}{N})$,
and more generally write $\calG$ as a direct product of such cyclic groups and take tensor products of such characters.
In the proof of Lemma~\ref{lem.autocor of mult char}, we revisit this issue in the special case where $\calG$ is the additive group of a finite field.
Since $\#(\widehat{\calG})=G$, the character table $\bfGamma$ is a complex Hadamard matrix,
implying $\bfGamma^{-1}=\frac{1}{G}\bfGamma^*$.
Here, \smash{$(\bfgamma)_{\gamma\in\widehat{\calG}}$} is the \textit{Fourier basis} for $\bbC^{\calG}$ and its analysis matrix $\bfGamma^*$ is the \textit{discrete Fourier transform (over $\calG$)}.
Since $\bbC^\calG$ is a commutative $*$-algebra,
the vector space isomorphism $\bfGamma^*$ induces a commutative $*$-algebra on \smash{$\bbC^{\widehat{\calG}}$}.
Since any filter diagonalizes as \smash{$\sum_{g\in\calG}\brT^g\bff(g)
=\bfGamma\diag(\bfGamma^*\bff)\bfGamma^{-1}
=\frac1G\bfGamma\diag(\bfGamma^*\bff)\bfGamma^{*}$},
the induced multiplication and involution on \smash{$\bbC^{\widehat{\calG}}$} are pointwise multiplication and conjugation, respectively.
We refer to the three isomorphic complex $*$-algebras $\calA$, $\bbC^\calG$ and \smash{$\bbC^{\widehat{\calG}}$} as \textit{Fourier algebras}.

Under the above conventions, we have the following technical lemma,
which characterizes those synthesis channels which convolve an upsampled $\bfx\in\bbC^{\calG}$ with some $\widetilde{\bfphi}_n\in\bbC^{\calH}$.

\begin{lemma}
\label{lem.polyphase}
Let $\calK$ be a transversal of a finite abelian group $\calH$ with respect to some subgroup $\calG$ of order $G$.
Define unitaries \smash{$\bfU:=\bfP\tfrac1{\sqrt{G}}\bfGamma^\calK\bfQ$} and \smash{$\bfV:=\frac1{\sqrt{G}}\bfGamma$} where \smash{$\bfGamma^\calK$} is the direct sum of $\calK$-indexed copies of the \smash{$\calG\times\widehat{\calG}$} character table $\bfGamma$ of $\calG$ and where \smash{$\bfP:\bbC^{\calK\times\calG}\rightarrow\bbC^{\calH}$} and
\smash{$\bfQ:\bbC^{\widehat{\calG}\times\calK}\rightarrow\bbC^{\calK\times\widehat{\calG}}$} are permutation matrices such that
\smash{$\bfP\bfdelta_{(k,g)}:=\bfdelta_{k+g}$} and \smash{$\bfQ\bfdelta_{\gamma,k}:=\bfdelta_{k,\gamma}$}.

Let $\widetilde{\bfPhi}_n:\bbC^\calG\rightarrow\bbC^\calH$ and \smash{$\widehat{\bfPhi}_n:\bbC^{\widehat{\calG}}\rightarrow\bbC^{\widehat{\calG}\times\calK}$} be linear maps such that \smash{$\widetilde{\bfPhi}_n=\bfU\widehat{\bfPhi}_n\bfV^*$}.
Then \smash{$\widetilde{\bfPhi}_n$} is the synthesis map of \smash{$(\brT^g\widetilde{\bfphi}_n)_{g\in\calG}$} for some \smash{$\widetilde{\bfphi}_n\in\bbC^\calH$} if and only if there exists a finite sequence
\smash{$(\bfphi_{\gamma,n})_{\gamma\in\widehat{\calG}}$} in \smash{$\bbC^{\calK}$} such that \smash{$(\widehat{\bfPhi}_n\bfx)(\gamma,k)=\bfphi_{\gamma,n}(k)\bfx(\gamma)$} for any
\smash{$\bfx\in\bbC^{\widehat{\calG}}$},
\smash{$\gamma\in\widehat{\calG}$} and $k\in\calK$.
When this occurs,
\smash{$\bfphi_{\gamma,n}(k)=(\bfGamma^*\dsample\brT^{-k}\widetilde{\bfphi}_n)(\gamma)$} for any \smash{$\gamma\in\widehat{\calG}$} and $k\in\calK$.
When this occurs and there moreover exists \smash{$\bfphi_n\in\bbC^\calK$} such that \smash{$\bfphi_{\gamma,n}=\bfphi_n$} for any \smash{$\gamma\in\widehat{\calG}$} then
\begin{equation}
\label{eq.constant polyphase}
\widetilde{\bfphi}_n(h)
=\left\{\begin{array}{cl}
\bfphi_{n}(h),&\ h\in\calK,\\
0,&\ h\notin\calK.
\end{array}\right.
\end{equation}
\end{lemma}

\begin{proof}[Proof.]
To clarify, the ``$n$'' subscripts here are unnecessary but help maintain consistent notation.

($\Rightarrow$)
Assume \smash{$\widetilde{\bfPhi}_n$} is the synthesis map of \smash{$(\brT^g\widetilde{\bfphi}_n)_{g\in\calG}$}.
For any $\bfy\in\bbC^{\calG}$, $k\in\calK$ and $g\in\calG$,
\begin{equation*}
(\bfP^*\widetilde{\bfPhi}_n\bfy)(k,g)
=\ip{\bfP\bfdelta_{k,g}}{\widetilde{\bfPhi}_n\bfy}
=(\widetilde{\bfPhi}_n\bfy)(k+g)
=\sum_{g'\in\calG}(\brT^{g'}\widetilde{\bfphi}_n)(k+g)\bfy(g').
\end{equation*}
Since $(\brT^{g'}\widetilde{\bfphi}_n)(k+g)
=\widetilde{\bfphi}_n(k+g-g')
=(\brT^{-k}\widetilde{\bfphi}_n)(g-g')$ where $g-g'\in\calG$,
this continues as
\begin{equation*}
(\bfP^*\widetilde{\bfPhi}_n\bfy)(k,g)
=\sum_{g'\in\calG}(\brT^{-k}\widetilde{\bfphi}_n)(g-g')\bfy(g')
=\sum_{g'\in\calG}(\dsample\brT^{-k}\widetilde{\bfphi}_n)(g-g')\bfy(g')
=\bigparen{(\dsample\brT^{-k}\widetilde{\bfphi}_n)*\bfy}(g).
\end{equation*}
Identifying \smash{$\bbC^{\calK\times\calG}$} with the direct sum of $\calK$-indexed copies of \smash{$\bbC^\calG$},
this gives \smash{$\bfP^*\widetilde{\bfPhi}_n\bfy=\oplus_{k\in\calK}\bfF_{n,k}\bfy$} where, for each $k\in\calK$,
\smash{$\bfF_{n,k}$} is the filter on $\bbC^{\calG}$ that has \smash{$\bff_{n,k}:=\dsample\brT^{-k}\widetilde{\bfphi}_n$} as its impulse response.
(That is, \smash{$\bfP^*\widetilde{\bfPhi}_n$} is the ``vertical concatenation'' of these filters \smash{$(\bfF_{n,k})_{k\in\calK}$}.)
Now recall that any filter $\bfF$ on $\bbC^\calG$ has \smash{$\tfrac1G\bfGamma^*\bfF\bfGamma=\diag(\bfGamma^*\bff)$} where $\bff$ is the impulse response of $\bfF$.
Letting $\bfy=\bfGamma\bfx$ thus gives
\begin{equation*}
\tfrac{1}{G}(\bfGamma^{\calK})^*\bfP^*\widetilde{\bfPhi}_n\bfGamma\bfx
=\tfrac{1}{G}(\bfGamma^*)^{\calK}(\oplus_{k\in\calK}\bfF_{n,k}\bfGamma\bfx)
=\oplus_{k\in\calK}(\tfrac{1}{G}\bfGamma^*\bfF_{n,k}\bfGamma\bfx)
=\oplus_{k\in\calK}\bigparen{(\bfGamma^*\bff_{n,k})\odot\bfx}
\end{equation*}
for any \smash{$\bfx\in\bbC^{\widehat{\calG}}$},
where $\odot$ denotes entrywise multiplication.
As such, identifying the direct sum of $\calK$-indexed copies of \smash{$\bbC^{\widehat{\calG}}$} with \smash{$\bbC^{\calK\times\widehat{\calG}}$},
\begin{equation*}
\bigparen{\tfrac{1}{G}\bfQ^*(\bfGamma^\calK)^*\bfP^*\widetilde{\bfPhi}_n\bfGamma\bfx}(\gamma,k)
=\bigparen{\tfrac{1}{G}(\bfGamma^\calK)^*\bfP^*\widetilde{\bfPhi}_n\bfGamma\bfx}(k,\gamma)\\
=\bigparen{(\bfGamma^*\bff_{n,k})\odot\bfx}(\gamma)
\end{equation*}
for any $\bfx\in\bbC^{\widehat{\calG}}$, $\gamma\in\widehat{\calG}$ and $k\in\calK$.
Here since
\begin{equation*}
\tfrac{1}{G}\bfQ^*(\bfGamma^\calK)^*\bfP^*\widetilde{\bfPhi}_n\bfGamma
=\bigparen{\bfP\tfrac1{\sqrt{G}}\bfGamma^\calK\bfQ}^*\widetilde{\bfPhi}_n(\tfrac{1}{\sqrt{G}}\bfGamma)
=\bfU^*\widetilde{\bfPhi}_n\bfV
=\widehat{\bfPhi}_n,
\end{equation*}
this reduces to the following for any such $\bfx$, $\gamma$ and $k$,
\begin{equation*}
(\widehat{\bfPhi}_n\bfx)(\gamma,k)
=(\bfGamma^*\bff_{n,k})(\gamma)\bfx(\gamma)
=(\bfGamma^*\dsample\brT^{-k}\widetilde{\bfphi}_n)(\gamma)\bfx(\gamma).
\end{equation*}
As such, for any finite sequence \smash{$(\bfphi_{\gamma,n})_{\gamma\in\widehat{\calG}}$} in \smash{$\bbC^{\calK}$},
we have \smash{$(\widehat{\bfPhi}_n\bfx)(\gamma,k)=\bfphi_{\gamma,n}(k)\bfx(\gamma)$} for any such $\bfx$, $\gamma$ and $k$ if and only if
$\bfphi_{\gamma,n}(k)\bfx(\gamma)=(\bfGamma^*\dsample\brT^{-k}\widetilde{\bfphi}_n)(\gamma)\bfx(\gamma)$ for any such $\bfx$, $\gamma$ and $k$,
namely if and only if \smash{$\bfphi_{\gamma,n}(k)=(\bfGamma^*\dsample\brT^{-k}\widetilde{\bfphi}_n)(\gamma)$}
for any \smash{$\gamma\in\widehat{\calG}$} and $k\in\calK$.

When this occurs and there moreover exists $\bfphi_n\in\bbC^\calK$ such that $\bfphi_{\gamma,n}=\bfphi_n$ for any $\gamma\in\widehat{\calG}$ then
\begin{equation*}
(\bfGamma^*\dsample\brT^{-k}\widetilde{\bfphi}_n)(\gamma)
=\bfphi_{\gamma,n}(k)
=\bfphi_{n}(k)
=\bfphi_{n}(k)\bfchi_{\widehat{\calG}}(\gamma)
=\bigparen{\bfchi_{\widehat{\calG}}\bfphi_{n}(k)}(\gamma)
\end{equation*}
for any $k\in\calK$ and $\gamma\in\widehat{\calG}$, where \smash{$\bfchi_{\widehat{\calG}}\in\bbC^{\widehat{\calG}}$} is the ``all-ones'' characteristic function of \smash{$\widehat{\calG}$}.
Since \smash{$\bfchi_{\widehat{\calG}}=\bfGamma^*\bfdelta_0$},
this implies $\dsample\brT^{-k}\widetilde{\bfphi}_n=\bfdelta_0\bfphi_{n}(k)$ for any $k\in\calK$.
Thus,
\begin{equation*}
\widetilde{\bfphi}_n(g+k)
=(\brT^{-k}\widetilde{\bfphi}_n)(g)
=(\dsample\brT^{-k}\widetilde{\bfphi}_n)(g)
=\bfdelta_0(g)\bfphi_{n}(k)
\end{equation*}
for any $g\in\calG$ and $k\in\calK$.
Since $\calK$ is a transversal of $\calH$ with respect to $\calG$,
any $h\in\calH$ can be uniquely decomposed as $h=g+k$ where $(g,k)\in\calG\times\calK$,
and so this implies~\eqref{eq.constant polyphase}.

($\Leftarrow$)
Let $\calT$ be the set of linear maps
\smash{$\widetilde{\bfPhi}_n:\bbC^\calG\rightarrow\bbC^\calH$} that are the synthesis maps of \smash{$(\brT^g\widetilde{\bfphi}_n)_{g\in\calG}$} for some \smash{$\widetilde{\bfphi}_n\in\bbC^\calH$}.
Here since \smash{$\widetilde{\bfphi}_n\in\bbC^\calH$} is arbitrary and uniquely determines such a map \smash{$\widetilde{\bfPhi}_n$},
this set $\calT$ is a subspace of \smash{$\bbC^{\calH\times\calG}$} of dimension $H=\#(\calH)$.
Similarly, let $\calS$ be the set of linear maps
\smash{$\widehat{\bfPhi}_n:\bbC^{\widehat{\calG}}\rightarrow\bbC^{\widehat{\calG}\times\calK}$}
for which there exists a finite sequence
\smash{$(\bfphi_{\gamma,n})_{\gamma\in\widehat{\calG}}$} in \smash{$\bbC^{\calK}$} such that \smash{$(\widehat{\bfPhi}_n\bfx)(\gamma,k)=\bfphi_{\gamma,n}(k)\bfx(\gamma)$} for any
\smash{$\bfx\in\bbC^{\widehat{\calG}}$},
\smash{$\gamma\in\widehat{\calG}$} and $k\in\calK$,
that is, block-diagonal $\widehat{\calG}\times\widehat{\calG}$ arrays of column vectors in $\bbC^\calK$.
This set $\calS$ is a subspace of \smash{$\bbC^{(\widehat{\calG}\times\calK)\times\widehat{\calG}}$} of dimension $\#(\widehat{\calG})\#(\calK)=G\frac{H}{G}=H$.
Now note that since $\bfU$ and $\bfV$ are unitary, \smash{$\bfL:\bbC^{\calH\times\calG}\rightarrow\bbC^{(\widehat{\calG}\times\calK)\times\widehat{\calG}}$},
\smash{$\bfL(\widetilde{\bfPhi}):=\bfU^*\widetilde{\bfPhi}\bfV$} is an injective linear map between spaces of dimension $HG$,
and is thus invertible.
By the ``only if'' direction already proven above, we further have $\bfL(\calT)\subseteq\calS$.
Since $\dim(\bfL(\calT))=\dim(\calT)=H=\dim(\calS)$,
this implies $\bfL(\calT)=\calS$.
In particular, if \smash{$\widehat{\bfPhi}_n:\bbC^{\widehat{\calG}}\rightarrow\bbC^{\widehat{\calG}\times\calK}$} is linear and there exists a finite sequence
\smash{$(\bfphi_{\gamma,n})_{\gamma\in\widehat{\calG}}$} in \smash{$\bbC^{\calK}$} such that \smash{$(\widehat{\bfPhi}_n\bfx)(\gamma,k)=\bfphi_{\gamma,n}(k)\bfx(\gamma)$} for any
\smash{$\bfx\in\bbC^{\widehat{\calG}}$},
\smash{$\gamma\in\widehat{\calG}$} and $k\in\calK$,
then \smash{$\widehat{\bfPhi}_n\in\calS=\bfL(\calT)$},
implying \smash{$\widehat{\bfPhi}_n=\bfU^*\widetilde{\bfPhi}_n\bfV$} for some $\widetilde{\bfPhi}_n$ that is the synthesis map of \smash{$(\brT^g\widetilde{\bfphi}_n)_{g\in\calG}$} for some \smash{$\widetilde{\bfphi}_n\in\bbC^\calH$}.
\end{proof}

\begin{example}
\label{ex.protopolyphase}
Let $\calH=\bbZ_4$, $\calG=\set{0,2}$ and $\calK=\set{0,1}$.
For any $\widetilde{\bfphi}_n\in\bbC^{\bbZ_4}$,
the corresponding map $\widetilde{\bfPhi}_n$ of Lemma~\ref{lem.polyphase} is
\begin{equation}
\label{ex.protopolyphase.1}
\widetilde{\bfPhi}_n
=\left[\begin{array}{cc}
\widetilde{\bfphi}_n(0)&\widetilde{\bfphi}_n(2)\\
\widetilde{\bfphi}_n(1)&\widetilde{\bfphi}_n(3)\\
\widetilde{\bfphi}_n(2)&\widetilde{\bfphi}_n(0)\\
\widetilde{\bfphi}_n(3)&\widetilde{\bfphi}_n(1)
\end{array}\right].
\end{equation}
Permuting its rows according to coset representative yields a vertical concatenation of two matrices,
each of which is circulant with respect to the cyclic group $\calG$,
and so expressible as a polynomial in terms of the translation operator $\brT=\left[\begin{smallmatrix}0&1\\1&0\end{smallmatrix}\right]$ of its generator:
\begin{equation}
\label{eq.ex.protopolyphase.2}
\bfP^*\widetilde{\bfPhi}_n
=\left[\begin{array}{cc}
\widetilde{\bfphi}_n(0)&\widetilde{\bfphi}_n(2)\\
\widetilde{\bfphi}_n(2)&\widetilde{\bfphi}_n(0)\\
\widetilde{\bfphi}_n(1)&\widetilde{\bfphi}_n(3)\\
\widetilde{\bfphi}_n(3)&\widetilde{\bfphi}_n(1)
\end{array}\right]
=\left[\begin{array}{c}
\widetilde{\bfphi}_n(0)\left[\begin{smallmatrix}1&0\\0&1\end{smallmatrix}\right]+
\widetilde{\bfphi}_n(2)\left[\begin{smallmatrix}0&1\\1&0\end{smallmatrix}\right]\smallskip\\
\widetilde{\bfphi}_n(1)\left[\begin{smallmatrix}1&0\\0&1\end{smallmatrix}\right]+
\widetilde{\bfphi}_n(3)\left[\begin{smallmatrix}0&1\\1&0\end{smallmatrix}\right]
\end{array}\right]
=\left[\begin{array}{c}
\widetilde{\bfphi}_n(0)\brI+\widetilde{\bfphi}_n(2)\brT\smallskip\\
\widetilde{\bfphi}_n(1)\brI+\widetilde{\bfphi}_n(3)\brT
\end{array}\right].
\end{equation}
We now exploit the fact that translation operators, like all filters, are diagonalized by the Fourier basis.
Here in particular,
$\frac1G\bfGamma^*\brT\bfGamma
=\tfrac12
\left[\begin{smallmatrix}1&\phm 1\\1&-1\end{smallmatrix}\right]
\left[\begin{smallmatrix}0&1\\1&0\end{smallmatrix}\right]
\left[\begin{smallmatrix}1&\phm 1\\1&-1\end{smallmatrix}\right]
=\left[\begin{smallmatrix}1&\phm0\\0&-1\end{smallmatrix}\right]$,
and so multiplying the previous matrix on the left and right by $\frac1G(\bfGamma^\calK)^*$ and $\bfGamma$, respectively, gives
\begin{equation*}
\tfrac1G(\bfGamma^\calK)^*\bfP^*\widetilde{\bfPhi}_n\bfGamma
=\tfrac12\left[\begin{array}{cc}\bfGamma^*&\bfzero\smallskip\\\bfzero&\bfGamma^*\end{array}\right]
\left[\begin{array}{c}
\widetilde{\bfphi}_n(0)\brI+\widetilde{\bfphi}_n(2)\brT\smallskip\\
\widetilde{\bfphi}_n(1)\brI+\widetilde{\bfphi}_n(3)\brT
\end{array}\right]\bfGamma
=\left[\begin{array}{cc}
\widetilde{\bfphi}_n(0)+\widetilde{\bfphi}_n(2)&0\\
0&\widetilde{\bfphi}_n(0)-\widetilde{\bfphi}_n(2)\\
\widetilde{\bfphi}_n(1)+\widetilde{\bfphi}_n(3)&0\\
0&\widetilde{\bfphi}_n(1)-\widetilde{\bfphi}_n(3)
\end{array}\right].
\end{equation*}
Perfectly shuffling the rows then gives a block-diagonal array of column vectors in $\bbC^\calK$,
namely
\begin{equation*}
\widehat{\bfPhi}_n
=\bfU^*\widetilde{\bfPhi}_n\bfV
=\left[\begin{array}{cc}
\widetilde{\bfphi}_n(0)+\widetilde{\bfphi}_n(2)&0\\
\widetilde{\bfphi}_n(1)+\widetilde{\bfphi}_n(3)&0\\
0&\widetilde{\bfphi}_n(0)-\widetilde{\bfphi}_n(2)\\
0&\widetilde{\bfphi}_n(1)-\widetilde{\bfphi}_n(3)
\end{array}\right]
=\left[\begin{array}{cc}
\bfphi_{1,n}&\bfzero\\
\bfzero&\bfphi_{-1,n}
\end{array}\right]
\end{equation*}
where, for any $z\in\set{1,-1}$,
the vector $\bfphi_{z,n}$ in $\bbC^\calK$ is obtained by converting each ``power of $\brT$'' in the right-hand quantity of~\eqref{eq.ex.protopolyphase.2} into a power of $z^{-1}$, namely
\begin{equation}
\label{eq.ex.protopolyphase.3}
\bfphi_{z,n}
=\left[\begin{array}{c}
\bfphi_{z,n}(0)\\
\bfphi_{z,n}(1)
\end{array}\right]
=\left[\begin{array}{c}
\widetilde{\bfphi}_n(0)+\widetilde{\bfphi}_n(2)z^{-1}\\
\widetilde{\bfphi}_n(1)+\widetilde{\bfphi}_n(3)z^{-1}
\end{array}\right].
\end{equation}
That is, $\bfphi_{z,n}$ is a column vector whose $k$th entry is the classical \textit{$z$-transform} of the portion of $\widetilde{\bfphi}_n$ indexed by the corresponding $k$th coset $k+\calG$ of $\calG$ in $\calH$.
This is the classical polyphase representation~\cite{Vaidyanathan93} of $\widetilde{\bfPhi}_n$.
To illustrate the final conclusion of Lemma~\ref{lem.polyphase},
note that there exists $\bfphi_n\in\bbC^\calK$ such that
\begin{equation*}
\left[\begin{array}{c}
\widetilde{\bfphi}_n(0)+\widetilde{\bfphi}_n(2)z^{-1}\\
\widetilde{\bfphi}_n(1)+\widetilde{\bfphi}_n(3)z^{-1}
\end{array}\right]
=\left[\begin{array}{c}
\bfphi_n(0)\\
\bfphi_n(1)
\end{array}\right]
\end{equation*}
for every $z\in\set{1,-1}$ if and only if
$\widetilde{\bfphi}_n(0)=\bfphi_n(0)$,
$\widetilde{\bfphi}_n(1)=\bfphi_n(1)$ and
$\widetilde{\bfphi}_n(2)=\widetilde{\bfphi}_n(3)=0$.
In this degenerate case,
$\widetilde{\bfPhi}_n$ is essentially the tensor product of an identity matrix and $\bfphi_n$:
\begin{equation*}
\widetilde{\bfPhi}_n=\left[\begin{array}{cc}
\bfphi_n(0)&0\\
\bfphi_n(1)&0\\
0&\bfphi_n(0)\\
0&\bfphi_n(1)
\end{array}\right].
\end{equation*}
\end{example}

Lemma~\ref{lem.polyphase} leads to the following generalization of the classical polyphase matrix of a filter bank over a cyclic group to that over an arbitrary finite abelian group:

\begin{definition}
\label{def.polyphase}
Under the hypotheses of Lemma~\ref{lem.polyphase},
the \textit{polyphase matrix} of a finite sequence \smash{$(\widetilde{\bfphi}_n)_{n\in\calN}$} in $\bbC^\calH$ is the function that maps any \smash{$\gamma\in\widehat{\calG}$} to the synthesis matrix $\bfPhi_\gamma$ of the associated finite sequence \smash{$(\bfphi_{\gamma,n})_{n\in\calN}$} in $\bbC^\calK$,
namely $\bfPhi_\gamma\in\bbC^{\calK\times\calN}$,
\begin{equation*}
\bfPhi_\gamma(k,n)
=\bfphi_{\gamma,n}(k)
:=(\bfGamma^*\dsample\brT^{-k}\widetilde{\bfphi}_n)(\gamma)
=\sum_{g\in\calG}\overline{\gamma(g)}\widetilde{\bfphi}_n(k+g).
\end{equation*}
\end{definition}

Combining this classical concept with the novel machinery introduced earlier in this section yields the following result concerning the maps $(\widetilde{\bfPhi}_n)_{n\in\calN}$ of the synthesis filter bank of $(\widetilde{\bfphi}_n)_{n\in\calN}$.
It generalizes---to the setting of noncyclic finite abelian groups---a known characterization~\cite{ChebiraFM11} of when such maps $(\widetilde{\bfPhi}_n)_{n\in\calN}$ are isometries whose images form a TFF for $\bbC^\calH$,
and provides novel characterizations of when such a TFF is EC or EI.

\begin{theorem}
\label{thm.filter banks}
Assume the hypotheses of Lemma~\ref{lem.polyphase} and Definition~\ref{def.polyphase},
and let $\calG$ and $\calH$ have order $G$ and $H$, respectively, and so \smash{$K:=\#(\calK)=\tfrac{H}{G}$}.
Let $\widetilde{\calF}_n:=\im(\widetilde{\bfPhi}_n)$ for each $n\in\calN$ and let $N:=\#(\calN)$.
Then every $\widetilde{\bfPhi}_n$ is an isometry if and only if $\bfphi_{\gamma,n}$ is unit norm for every \smash{$\gamma\in\widehat{\calG}$} and $n\in\calN$.
When this occurs and $N>1$, \smash{$(\widetilde{\bfPhi}_n)_{n\in\calN}$} has minimal spectral coherence (over this class) if and only if each \smash{$(\bfphi_{\gamma,n})_{n\in\calN}$} has minimal coherence,
\smash{$(\widetilde{\calF}_n)_{n\in\calN}$} is a TFF or EITFF for \smash{$\bbC^\calH$} if and only if each \smash{$(\bfphi_{\gamma,n})_{n\in\calN}$} is a UNTF or ETF for \smash{$\bbC^{\calK}$}, respectively,
and \smash{$(\widetilde{\calF}_n)_{n\in\calN}$} is an ECTFF for \smash{$\bbC^\calH$} if and only if
\smash{$\bigparen{(\bfphi_{\gamma,n})_{n\in\calN}}_{\gamma\in\widehat{\calG}}$} is a \smash{$\SUNTF_\bbC(K,N,G)$}.
\end{theorem}

\begin{proof}[Proof.]
Clearly, \smash{$\widetilde{\bfPhi}_n=\bfU\widehat{\bfPhi}_n\bfV^*$} is an isometry if and only if \smash{$\widehat{\bfPhi}_n$} is an isometry.
Here, Lemma~\ref{lem.polyphase} gives
\begin{equation*}
(\widehat{\bfPhi}_{n}^{*}\widehat{\bfPhi}_{n}^{})(\gamma_1,\gamma_2)
=\sum_{\gamma\in\widehat{\calG}}\sum_{k\in\calK}
\overline{\bfphi_{\gamma,n}(k)\bfdelta_{\gamma_1}(\gamma)}\bfphi_{\gamma,n}(k)\bfdelta_{\gamma_2}(\gamma)
=\sum_{\gamma\in\widehat{\calG}}\bfdelta_{\gamma_1}(\gamma)\bfdelta_{\gamma_2}(\gamma)\norm{\bfphi_{\gamma,n}}^2
\end{equation*}
for any \smash{$\gamma_1,\gamma_2\in\widehat{\calG}$},
and so \smash{$\widehat{\bfPhi}_{n}^{*}\widehat{\bfPhi}_{n}^{}$} is diagonal with
\smash{$(\widehat{\bfPhi}_{n}^{*}\widehat{\bfPhi}_{n}^{})(\gamma,\gamma)
=\norm{\bfphi_{\gamma,n}}^2$} for any \smash{$\gamma\in\widehat{\calG}$}.
As such, every \smash{$\widetilde{\bfPhi}_n$} is an isometry if and only if every \smash{$\bfphi_{\gamma,n}$} is unit norm,
as claimed.
For the remaining claims, assume that this holds.
By Lemma~\ref{lem.polyphase},
\smash{$(\widetilde{\bfPhi}_n)_{n\in\calN}$} is equivalent to \smash{$(\widehat{\bfPhi}_n)_{n\in\calN}$}
where the latter arises from \smash{$\bigparen{(\bfphi_{\gamma,n})_{n\in\calN}}_{\gamma\in\widehat{\calG}}$} in the manner of Theorem~\ref{thm.SUNTF to ECTFF} in the special case where $\bbF=\bbC$,
\smash{$\calR=\widehat{\calG}$}, \smash{$\calE_\gamma=\bbC^{\calK}$} for each $\gamma\in\widehat{\calG}$ and we identify
\smash{$\widehat{\calE}=(\bbC^\calK)^{\widehat{\calG}}$} with \smash{$\bbC^{\widehat{\calG}\times\calK}$}.
As such, the spectral coherence of \smash{$(\widetilde{\bfPhi}_n)_{n\in\calN}$} equals that of \smash{$(\widehat{\bfPhi}_n)_{n\in\calN}$} and \smash{$(\widetilde{\calF}_n)_{n\in\calN}$} is a (EC/EI) TFF for \smash{$\bbC^\calH$} if and only if \smash{$(\widehat{\calF}_n)_{n\in\calN}$} is a (EC/EI) TFF for \smash{$\bbC^{\widehat{\calG}\times\calK}$} where \smash{$\widehat{\calF}_n:=\im(\widehat{\bfPhi}_n)$} for each $n\in\calN$.
The remaining claims then follow from parts (b), (c), (e) and (g) of Theorem~\ref{thm.SUNTF to ECTFF}.
\end{proof}

\begin{example}
\label{ex.polyphase}
Continuing Example~\ref{ex.protopolyphase},
now let $\calN=\bbF_5$ and depict the synthesis filter bank of an arbitrary finite sequence $(\widetilde{\bfphi}_n)_{n\in\bbF_5}$ in $\bbC^{\bbZ_4}$ as the horizontal concatenation of the corresponding five matrices $(\widetilde{\bfPhi}_n)_{n\in\bbF_5}$ of type~\eqref{ex.protopolyphase.1}:
\begin{equation}
\label{eq.ex.polyphase.1}
\widetilde{\bfPhi}
=\left[\begin{array}{cc|cc|cc|cc|cc}
\widetilde{\bfphi}_0(0)&\widetilde{\bfphi}_0(2)&\widetilde{\bfphi}_1(0)&\widetilde{\bfphi}_1(2)&\widetilde{\bfphi}_2(0)&\widetilde{\bfphi}_2(2)&\widetilde{\bfphi}_3(0)&\widetilde{\bfphi}_3(2)&\widetilde{\bfphi}_4(0)&\widetilde{\bfphi}_4(2)\\
\widetilde{\bfphi}_0(1)&\widetilde{\bfphi}_0(3)&\widetilde{\bfphi}_1(1)&\widetilde{\bfphi}_1(3)&\widetilde{\bfphi}_2(1)&\widetilde{\bfphi}_2(3)&\widetilde{\bfphi}_3(1)&\widetilde{\bfphi}_3(3)&\widetilde{\bfphi}_4(1)&\widetilde{\bfphi}_4(3)\\
\widetilde{\bfphi}_0(2)&\widetilde{\bfphi}_0(0)&\widetilde{\bfphi}_1(2)&\widetilde{\bfphi}_1(0)&\widetilde{\bfphi}_2(2)&\widetilde{\bfphi}_2(0)&\widetilde{\bfphi}_3(2)&\widetilde{\bfphi}_3(0)&\widetilde{\bfphi}_4(2)&\widetilde{\bfphi}_4(0)\\
\widetilde{\bfphi}_0(3)&\widetilde{\bfphi}_0(1)&\widetilde{\bfphi}_1(3)&\widetilde{\bfphi}_1(1)&\widetilde{\bfphi}_2(3)&\widetilde{\bfphi}_2(1)&\widetilde{\bfphi}_3(3)&\widetilde{\bfphi}_3(1)&\widetilde{\bfphi}_4(3)&\widetilde{\bfphi}_4(1)
\end{array}\right].
\end{equation}
Identifying the identity and nonidentity members of $\widehat{\calG}$ with the square roots of unity $1$ and $-1$, respectively,
the polyphase matrix of Definition~\ref{def.polyphase} is the function that maps any $z\in\set{1,-1}$ to the synthesis matrix of the vectors $(\bfphi_{z,n})_{n\in\bbF_5}$ of~\eqref{eq.ex.protopolyphase.3}, namely
\begin{align*}
\bfPhi_{z}
&=\left[\begin{array}{ccccc}
\widetilde{\bfphi}_0(0)&\widetilde{\bfphi}_1(0)&\widetilde{\bfphi}_2(0)&\widetilde{\bfphi}_3(0)&\widetilde{\bfphi}_4(0)\\
\widetilde{\bfphi}_0(1)&\widetilde{\bfphi}_1(1)&\widetilde{\bfphi}_2(1)&\widetilde{\bfphi}_3(1)&\widetilde{\bfphi}_4(1)
\end{array}\right]
+\left[\begin{array}{ccccc}
\widetilde{\bfphi}_0(2)&\widetilde{\bfphi}_1(2)&\widetilde{\bfphi}_2(2)&\widetilde{\bfphi}_3(2)&\widetilde{\bfphi}_4(2)\\
\widetilde{\bfphi}_0(3)&\widetilde{\bfphi}_1(3)&\widetilde{\bfphi}_2(3)&\widetilde{\bfphi}_3(3)&\widetilde{\bfphi}_4(3)
\end{array}\right]z^{-1}.
\end{align*}
A careful read of Lemma~\ref{lem.polyphase} reveals that,
in general,
perfectly shuffling the columns of $\bfU^*\widetilde{\bfPhi}\bfV^\calN$ yields a block diagonal matrix with $\bfPhi_\gamma$ as its $\gamma$th diagonal block,
leading to the equivalences of Theorem~\ref{thm.filter banks}.
Here, for example,
it gives that \smash{$(\widetilde{\bfPhi}_n)_{n\in\bbF_5}$} are the isometries of an EITFF if and only if \smash{$\bfPhi_z$} is (the synthesis matrix of) an $\ETF_\bbC(2,5)$ for every $z\in\set{1,-1}$,
which, by the Gerzon bound, is not possible.

In lieu of that, Theorem~\ref{thm.filter banks} more generally gives that $(\widetilde{\bfPhi}_n)_{n\in\bbF_5}$ of this form are isometries of minimal spectral coherence if and only if each $\bfPhi_z$ has unit norm columns of minimal coherence for any $z\in\set{1,-1}$.
In general, the equivalence of Lemma~\ref{lem.polyphase} allows us to define the polyphase matrix \smash{$(\bfPhi_\gamma)_{\gamma\in\widehat{\calG}}$} arbitrarily and recover corresponding impulse responses \smash{$(\widetilde{\bfphi}_n)_{n\in\calN}$}, with~\eqref{eq.constant polyphase} holding in the special case where \smash{$\bfPhi_\gamma$} is constant over $\gamma$.
In particular, taking \smash{$\bfPhi_z$} to be the matrix of~\eqref{eq.UNTF(2,5)} for both $z=1$ and $z=-1$ yields such $(\widetilde{\bfPhi}_n)_{n\in\bbF_5}$,
and the corresponding analysis filter bank \smash{$\widetilde{\bfPhi}^*$} acts on a given $\bfy\in\bbC^{\bbZ_4}$ in a manner analogous to a windowed Fourier transform (whose windows do not overlap),
decomposing \smash{$\bfy=\bigparen{\bfy(0),\bfy(1),\bfy(2),\bfy(3)}$} into first and second halves $\bfy_0=\bigparen{\bfy(0),\bfy(1)}$ and $\bfy_2=\bigparen{\bfy(2),\bfy(3)}$,
and taking the inner products of both halves with each of the five columns of~\eqref{eq.UNTF(2,5)}.

For users confined to such a filter bank architecture and willing to trade an increase in spectral coherence for a decrease in chordal coherence, Theorem~\ref{thm.filter banks} gives that
\smash{$(\widetilde{\bfPhi}_n)_{n\in\bbF_5}$} are the isometries of an ECTFF for $\bbC^{\bbZ_4}$ if and only if $(\bfPhi_z)_{z\in\set{1,-1}}$ is a $\SUNTF_\bbC(2,5,2)$.
This can be accomplished by taking $\bfPhi_1$ and $\bfPhi_{-1}$ to be the matrices of~\eqref{eq.ex.SUNTF(2,5,2).1} and~\eqref{eq.ex.SUNTF(2,5,2).2} of Example~\ref{ex.SUNTF(2,5,2)}, respectively.
The resulting (block-diagonal) isometries $(\widehat{\bfPhi}_n)_{n\in\bbF_5}$ are given in~\eqref{eq.ex.ECTFF(4,2,5).1},
and the corresponding synthesis filter bank can be obtained by reversing the process of Example~\ref{ex.protopolyphase}:
perfectly shuffling the rows of the matrix $\widehat{\bfPhi}$ of~\eqref{eq.ex.ECTFF(4,2,5).1} gives a matrix of diagonal blocks, namely
\begin{equation*}
\bfQ\widehat{\bfPhi}
=\left[\begin{array}{rr|rr|rr|rr|rr}
\phm 1&     0\phm&     u&     0\phm&     v& 0\phm& v&     0\phm& u& 0\phm\\
     0&\phm 1\phm&     0&     v\phm&     0& u\phm& 0&     u\phm& 0& v\phm\\
     0&     0\phm&\phm U&     0\phm&\phm V& 0\phm&-V&     0\phm&-U& 0\phm\\
     0&     0\phm&     0&\phm V\phm&     0&-U\phm& 0&\phm U\phm& 0&-V\phm
\end{array}\right],
\end{equation*}
and conjugating each by $\tfrac{1}{\sqrt{G}}\bfGamma$ yields a matrix $\tfrac{1}{G}\bfGamma^\calK\bfQ\widehat{\bfPhi}(\bfGamma^\calN)^{*}$ of circulant blocks,
namely
\begin{equation*}
\tfrac{1}{2}\left[\begin{array}{rr|rr|rr|rr|rr}
\phm 2&     0\phm&u+v&u-v&v+u&v-u& v+u& v-u& u+v& u-v\\
     0&\phm 2\phm&u-v&u+v&v-u&v+u& v-u& v+u& u-v& u+v\\
     0&     0\phm&U+V&U-V&V-U&V+U&-V+U&-V-U&-U-V&-U+V\\
     0&     0\phm&U-V&U+V&V+U&V-U&-V-U&-V+U&-U+V&-U-V
\end{array}\right],
\end{equation*}
and perfectly shuffling its rows gives a matrix of form~\eqref{eq.ex.polyphase.1} where
\begin{equation*}
\widetilde{\bfphi}_0=\left[\begin{array}{c}1\\0\\0\\0\end{array}\right],\quad
\widetilde{\bfphi}_1=\tfrac14\left[\begin{array}{r}
-1\\
\mu\\
\sqrt{5}\\
\nu
\end{array}\right],\quad
\widetilde{\bfphi}_2=\tfrac14\left[\begin{array}{r}
-1\\
-\nu\\
-\sqrt{5}\\
\mu
\end{array}\right],\quad
\widetilde{\bfphi}_3=\tfrac14\left[\begin{array}{r}
-1\\
 \nu\\
-\sqrt{5}\\
-\mu
\end{array}\right],\quad
\widetilde{\bfphi}_4=\tfrac14\left[\begin{array}{r}
-1\\
-\mu\\
\sqrt{5}\\
-\nu
\end{array}\right],
\end{equation*}
provided $\mu:=2(U+V)=\sqrt{5+2\sqrt{5}}$ and $\nu:=2(U-V)=\sqrt{5-2\sqrt{5}}$.
In contrast to cases where $\bfPhi_\gamma$ is constant over $\gamma$,
we note that here, for any $n\in\bbF_5\backslash\set{0}$,
the $\calG$-translates of $\widetilde{\bfphi}_n$ have overlapping support.
\end{example}

Altogether, we see that as far as TFFs, ECTFFs and EITFFs are concerned,
complex isometries arising via Theorem~\ref{thm.SUNTF to ECTFF} are equivalent to ones whose cross-Gram matrices commute (Theorem~\ref{thm.commuting}),
and in particular are equivalent to ones that arise as filter banks in the manner of Theorem~\ref{thm.filter banks}.
We caution that, for a given finite sequence of impulse responses \smash{$(\widetilde{\bfphi}_n)$} in \smash{$\bbC^\calH$},
the latter is only true when each isometry \smash{$\widetilde{\bfPhi}_n$} is the synthesis map of \smash{$(\brT^g\widetilde{\bfphi}_n)_{g\in\calG}$},
implying the fusion Gram matrix \smash{$\widetilde{\bfPhi}^*\widetilde{\bfPhi}^{}$} has circulant blocks.
This equivalence does not more generally apply when each isometry \smash{$\widetilde{\bfPhi}_g$} is the synthesis map of \smash{$(\brT^g\widetilde{\bfphi}_n)_{n\in\calN}$},
namely when the fusion Gram matrix \smash{$\widetilde{\bfPhi}^*\widetilde{\bfPhi}^{}$} is block circulant.
Rather, the latter are characterized by the theory of~\cite{FickusIJM23}.
That theory does not apply here in general:
though any Gram matrix with circulant blocks can be perfectly shuffled into one that is block circulant,
such shuffling does not preserve the ``every diagonal block is an identity'' property needed for TFFs.
For example, whereas \smash{$(\brT^g\widetilde{\bfphi}_n)_{n\in\calN}$} is orthonormal for every $g\in\calG$ if and only if \smash{$(\widetilde{\bfphi}_n)_{n\in\calN}$} is orthonormal,
\smash{$(\brT^g\widetilde{\bfphi}_n)_{g\in\calG}$} is orthonormal for a given $n\in\calN$ if and only if the corresponding polyphase vectors of Lemma~\ref{lem.polyphase} form a unit norm sequence \smash{$(\bfphi_{\gamma,n})_{\gamma\in\widehat{\calG}}$} in \smash{$\bbC^\calK$}.
And whereas any \smash{$\TFF_\bbC(\widetilde{D},R,N)$} of Theorem~\ref{thm.commuting} necessarily has $R$ dividing \smash{$\widetilde{D}$},
the block circulant TFFs of~\cite{FickusIJM23} are not restricted in this way,
yielding \smash{$\EITFF_\bbC(5,2,5)$}, for example.
That said, both~\cite{FickusIJM23} and the theory here apply to fusion Gram matrices that are BCCB,
that is, block circulant arrays of circulant blocks,
such as those arising from harmonic frames in the manner of~\cite{King16,FickusMW25}.

\section{Stratified unit norm tight frames}

In Theorems~\ref{thm.SUNTF to ECTFF}, \ref{thm.commuting} and~\ref{thm.filter banks},
we showed that if we restrict ourselves to isometries whose cross-Gram matrices commute,
such as when each is the channel of a synthesis filter bank,
the problem of constructing an $\EITFF_\bbC(\widetilde{D},R,N)$ reduces to that of constructing an \smash{$\ETF_\bbC(D,N)$} where \smash{$D=\frac{\widetilde{D}}{R}$},
and moreover that the problem of constructing isometries with minimal spectral coherence reduces to that of constructing unit norm vectors with minimal coherence.
Both of these problems thus reduce to ones with mature literature,
and we do not discuss them further here.
Rather, we focus our remaining attention on the fact that in this same context,
the problem of constructing an $\ECTFF_\bbC(\widetilde{D},R,N)$ reduces to that of constructing a $\SUNTF_\bbC(D,N,R)$.
Recalling Definition~\ref{def.SUNTF},
the latter is a novel yet straightforward ``multilayer'' generalization of an $\ETF_\bbF(D,N)$.
We begin this section with a discussion of how standard results about ETFs,
including the Welch and Gerzon bounds, generalize to the SUNTF setting.

\begin{theorem}
\label{thm.SUNTF}
Let $\bbF$ be $\bbR$ or $\bbC$,
let $D$, $N$ and $R$ be positive integers such that $N\geq 2$,
and let $\calN$ and $\calR$ be sets of cardinality $N$ and $R$, respectively.
For each $r\in\calR$,
let $(\bfphi_{r,n})_{n\in\calN}$ be a finite sequence of unit norm vectors in some Euclidean space $\calE_r$ over $\bbF$ of dimension $D$. Then:

\begin{enumerate}
\alphi

\item
$\displaystyle
\max_{n_1\neq n_2}\tfrac{1}{R}\sum_{r\in\calR}\abs{\ip{\bfphi_{r,n_1}}{\bfphi_{r,n_2}}}^2
\geq\tfrac{N-D}{D(N-1)}$.

\noindent
Moreover, equality holds here if and only if $\bigparen{(\bfphi_{r,n})_{n\in\calN}}_{r\in\calR}$ is a SUNTF for $(\calE_r)_{r\in\calR}$.
\smallskip

\item
If $\bigparen{(\bfphi_{r,n})_{n\in\calN}}_{r\in\calR}$ is stratified with chordal constant $\varLambda<1$,
such as when it is a $\SUNTF_\bbF(D,N,R)$ with $D>1$, then
\begin{equation*}
\tfrac{N-1}{R}
\leq(D-1)+\tfrac{D(D-1)}{2}[\bbF:\bbR].
\end{equation*}
Moreover, if equality holds here then $\bigparen{(\bfphi_{r,n})_{n\in\calN}}_{r\in\calR}$ is a $\SUNTF_\bbF(D,N,R)$.
\smallskip

\item
If $\bigparen{(\bfphi_{r,n})_{n\in\calN}}_{r\in\calR}$ is a $\SUNTF_\bbF(D,N,R)$ with $D<N$ then its layerwise Naimark complement $\bigparen{(\check{\bfphi}_{r,n})_{n\in\calN}}_{r\in\calR}$ is a $\SUNTF_\bbF(N-D,N,R)$.
\smallskip

\item
If \smash{$\bigparen{(\bfphi_{r,n}^{(m)})_{n\in\calN}}_{r\in\calR_m}$} is a \smash{$\SUNTF_\bbF(D,N,R_m)$} for each member $m$ of some nonempty finite set $\calM$ then \smash{$\bigparen{(\bfphi_{(m,r),n})_{n\in\calN}}_{(m,r)\in\widehat{\calR}}$}
is a \smash{$\SUNTF_\bbF(D,N,\sum_{m\in\calM}R_m)$}
provided \smash{$\bfphi_{(m,r),n}:=\bfphi_{r,n}^{(m)}$} for any \smash{$(m,r)\in\widehat{\calR}:=\sqcup_{m\in\calM}(\set{m}\times\calR_m)$} and $n\in\calN$.
\smallskip

\item
If \smash{$(\bfphi_n)_{n\in\calN}$} is any $\UNTF_\bbF(D,N)$ then \smash{$\bigparen{(\bfphi_{\sigma(n)})_{n\in\calN}}_{\sigma\in\Sym(\calN)}$} is a \smash{$\SUNTF_\bbF(D,N,N!)$},
where \smash{$\Sym(\calN)$} is the symmetric group of all permutations of $\calN$.

\end{enumerate}
\end{theorem}

\begin{proof}[Proof.]
(a)
Construct isometries $(\widehat{\bfPhi}_n)_{n\in\calN}$ into $\widehat{\calE}:=\oplus_{r\in\calR}\calE_r$ from $\bigparen{(\bfphi_{r,n})_{n\in\calN}}_{r\in\calR}$ in the manner of Theorem~\ref{thm.SUNTF to ECTFF}.
Here,
\smash{$\widehat{D}
:=\dim(\widehat{\calE})
=\dim(\oplus_{r\in\calR}\calE_r)
=\sum_{r\in\calR}\dim(\calE_r)
=DR$} and
\begin{equation*}
\tfrac{1}{R}\norm{\widehat{\bfPhi}_{n_1}^{*}\widehat{\bfPhi}_{n_2}^{}}_\Fro^2
=\tfrac{1}{R}\sum_{r_1\in\calR}\sum_{r_2\in\calR}
\abs{(\widehat{\bfPhi}_{n_1}^{*}\widehat{\bfPhi}_{n_2}^{})(r_1,r_2)}^2
=\tfrac{1}{R}\sum_{r\in\calR}\abs{\ip{\bfphi_{r,n_1}}{\bfphi_{r,n_2}}}^2
\end{equation*}
for any $n_1,n_2\in\calN$.
As such, the chordal Welch bound, given in~\eqref{eq.Welch bounds}, implies
\begin{equation*}
\max_{n_1\neq n_2}\tfrac{1}{R}\sum_{r\in\calR}\abs{\ip{\bfphi_{r,n_1}}{\bfphi_{r,n_2}}}^2
=\max_{n_1\neq n_2}\tfrac{1}{R}\norm{\widehat{\bfPhi}_{n_1}^{*}\widehat{\bfPhi}_{n_2}^{}}_\Fro^2
\geq\tfrac{NR-\widehat{D}}{\widehat{D}(N-1)}
=\tfrac{NR-DR}{DR(N-1)}
=\tfrac{N-D}{D(N-1)}.
\end{equation*}
Moreover, equality holds above if and only if it holds in the chordal Welch bound,
namely if and only if the images of $(\widehat{\bfPhi}_n)_{n\in\calN}$ form an ECTFF for $\widehat{\calE}$.
By Theorem~\ref{thm.SUNTF to ECTFF}(g), this occurs if and only if $\bigparen{(\bfphi_{r,n})_{n\in\calN}}_{r\in\calR}$ is a SUNTF for $(\calE_r)_{r\in\calR}$.

(b)
For any $r\in\calR$ and $n\in\calN$ we have
\smash{$\Tr(\bfphi_{r,n}^{}\bfphi_{r,n}^{*}-\tfrac1D\brI)
=\norm{\bfphi_{r,n}}^2-1
=0$} and so
\smash{$\bfphi_{r,n}^{}\bfphi_{r,n}^{*}-\tfrac1D\brI$} belongs to the real Euclidean space
$\calS_r:=\set{\bfL:\calE_r\rightarrow\calE_r \mid \bfL\text{ is linear},\, \bfL^*=\bfL,\, \Tr(\bfL)=0}$.
Since $\calE_r$ is a Euclidean space over $\bbF$ of dimension $D$,
$\calS_r$ has (real) dimension
$\dim(\calS_r)=(D-1)+\tfrac{D(D-1)}{2}[\bbF:\bbR]$.
For each $n\in\calN$,
the vector \smash{$\bfpsi_n:=\frac1{\sqrt{R}}\oplus_{r\in\calR}(\bfphi_{r,n}^{}\bfphi_{r,n}^{*}-\tfrac1D\brI)$} thus belongs to the space $\oplus_{r\in\calR}\calS_r$ of dimension \smash{$R\bigparen{(D-1)+\tfrac{D(D-1)}{2}[\bbF:\bbR]}$}.
This dimension is thus an upper bound on the rank of the Gram matrix $\bfPsi^*\bfPsi$ of $(\bfpsi_n)_{n\in\calN}$.
To compute the latter,
note that for any $n_1,n_2\in\calN$,
\begin{align*}
\ip{\bfpsi_{n_1}}{\bfpsi_{n_2}}
&=\tfrac{1}{R}\sum_{r\in\calR}\bigip{\bfphi_{r,n_1}^{}\bfphi_{r,n_1}^{*}-\tfrac1D\brI}{\bfphi_{r,n_2}^{}\bfphi_{r,n_2}^{*}-\tfrac1D\brI}_\Fro\\
&=\tfrac{1}{R}\sum_{r\in\calR}\Tr\Bigparen{\bigparen{\bfphi_{r,n_1}^{}\bfphi_{r,n_1}^{*}-\tfrac1D\brI}\bigparen{\bfphi_{r,n_2}^{}\bfphi_{r,n_2}^{*}-\tfrac1D\brI}}\\
&=\tfrac{1}{R}\sum_{r\in\calR}\Bigparen{\Tr\bigparen{\bfphi_{r,n_2}^{*}\bfphi_{r,n_1}^{}\bfphi_{r,n_1}^{*}\bfphi_{r,n_2}^{}}
-\tfrac1D\Tr\bigparen{\bfphi_{r,n_1}^{*}\bfphi_{r,n_1}^{}}-\tfrac1D\Tr\bigparen{\bfphi_{r,n_2}^{*}\bfphi_{r,n_2}^{}}+\tfrac1{D^2}\Tr(\brI)}\\
&=\tfrac{1}{R}\sum_{r\in\calR}\abs{\ip{\bfphi_{r,n_1}}{\bfphi_{r,n_2}}}^2-\tfrac1{D}.
\end{align*}
Since \smash{$\bigparen{(\bfphi_{r,n})_{n\in\calN}}_{r\in\calR}$} is stratified with chordal constant $\varLambda$,
this implies
$\bfPsi^*\bfPsi
=(1-\varLambda)\brI+(\varLambda-\tfrac{1}{D})\brJ$
where $\brJ\in\bbF^{\calN\times\calN}$ is an all-ones matrix.
The spectrum of $\bfPsi^*\bfPsi$ thus consists of \smash{$(1-\varLambda)+N(\varLambda-\tfrac{1}{D})=\varLambda(N-1)-\tfrac{N-D}{D}$}
and $N-1$ copies of $(1-\varLambda)$.
Now recall that $\varLambda<1$ by assumption,
such as when \smash{$\bigparen{(\bfphi_{r,n})_{n\in\calN}}_{r\in\calR}$} is a $\SUNTF_\bbF(D,N,R)$ with $D>1$.
As such, $1-\varLambda>0$ and so $\rank(\bfPsi^*\bfPsi)\in\set{N-1,N}$ where $\rank(\bfPsi^*\bfPsi)=N-1$ if and only if \smash{$\varLambda=\tfrac{N-D}{D(N-1)}$}.
By (a), the latter occurs if and only if \smash{$\bigparen{(\bfphi_{r,n})_{n\in\calN}}_{r\in\calR}$} is a SUNTF for $(\calE_r)_{r\in\calR}$,
that is, a $\SUNTF_\bbF(D,N,R)$.
Altogether,
\begin{equation*}
N-1
\leq\rank(\bfPsi^*\bfPsi)
\leq\dim(\oplus_{r\in\calR}\calS_r)
=R\bigparen{(D-1)+\tfrac{D(D-1)}{2}[\bbF:\bbR]},
\end{equation*}
where equality holds throughout only if \smash{$\bigparen{(\bfphi_{r,n})_{n\in\calN}}_{r\in\calR}$} is a $\SUNTF_\bbF(D,N,R)$.

(c)
For a direct proof,
for each $r\in\calR$,
let \smash{$(\check{\bfphi}_{r,n})_{n\in\calN}$} be a Naimark complement of \smash{$(\bfphi_{r,n})_{n\in\calN}$},
namely a UNTF for some space $\check{\calE}_r$ of dimension $N-D$ that satisfies
$\ip{\check{\bfphi}_{r,n_1}}{\check{\bfphi}_{r,n_2}}=-\tfrac{D}{N-D}\ip{\bfphi_{r,n_1}}{\bfphi_{r,n_2}}$ for any distinct $n_1,n_2\in\calN$.
Here, since \smash{$\bigparen{(\bfphi_{r,n})_{n\in\calN}}_{r\in\calR}$} is a $\SUNTF_\bbF(D,N,R)$,
(a) gives
\begin{equation*}
\tfrac{1}{R}\sum_{r\in\calR}\abs{\ip{\check{\bfphi}_{r,n_1}}{\check{\bfphi}_{r,n_2}}}^2
=\bigparen{\tfrac{D}{N-D}}^2\tfrac{1}{R}\sum_{r\in\calR}\abs{\ip{\bfphi_{r,n_1}}{\bfphi_{r,n_2}}}^2
=\bigparen{\tfrac{D}{N-D}}^2\tfrac{N-D}{D(N-1)}
=\tfrac{D}{(N-D)(N-1)}
\end{equation*}
for any distinct $n_1,n_2\in\calN$,
implying by (a) that \smash{$\bigparen{(\check{\bfphi}_{r,n})_{n\in\calN}}_{r\in\calR}$} is a $\SUNTF_\bbF(N-D,N,R)$.
For an alternative proof, combine (c) and (g) of Theorem~\ref{thm.SUNTF to ECTFF} with the fact that the Naimark complement of an $\ECTFF_\bbF(DR,R,N)$ with $NR>DR$ is an $\ECTFF_\bbF((N-D)R,R,N)$.
Here, Theorem~\ref{thm.SUNTF to ECTFF}(c) moreover gives that the finite sequence of isometries
\smash{$(\widehat{\check{\bfPhi}}_n)_{n\in\calN}$} constructed from the layerwise Naimark complement  \smash{$\bigparen{(\check{\bfphi}_{r,n})_{n\in\calN}}_{r\in\calR}$} of \smash{$\bigparen{(\bfphi_{r,n})_{n\in\calN}}_{r\in\calR}$} in the manner of Theorem~\ref{thm.SUNTF to ECTFF} is a Naimark complement of the finite sequence \smash{$(\widehat{\bfPhi}_n)_{n\in\calN}$} that arises from the latter in this same manner.

(d)
If each $\bigparen{(\bfphi_{r,n}^{(m)})_{n\in\calN}}_{r\in\calR_m}$ is stratified,
then letting $\varLambda_m$ be its chordal constant,
note that
\begin{equation*}
\sum_{m\in\calM}\sum_{r\in\calR_m}\abs{\ip{\bfphi_{(m,r),n_1}}{\bfphi_{(m,r),n_2}}}^2
=\sum_{m\in\calM}\sum_{r\in\calR_m}\abs{\ip{\bfphi_{r,n_1}^{(m)}}{\bfphi_{r,n_2}^{(m)}}}^2
=\sum_{m\in\calM}\#(\calR_m)\varLambda_m
\end{equation*}
for any distinct $n_1,n_2\in\calN$, and so
\smash{$\bigparen{(\bfphi_{(m,r),n})_{n\in\calN}}_{(m,r)\in\widehat{\calR}}$} is stratified.
Thus, if each $\bigparen{(\bfphi_{r,n}^{(m)})_{n\in\calN}}_{r\in\calR_m}$ is a $\SUNTF_\bbF(D,N,R_m)$,
each $(\bfphi_{(m,r),n})_{n\in\calN}$ is a $\UNTF_\bbF(D,N)$,
and so \smash{$\bigparen{(\bfphi_{(m,r),n})_{n\in\calN}}_{(m,r)\in\widehat{\calR}}$} is a \smash{$\SUNTF_\bbF(D,N,\sum_{m\in\calM}R_m)$}.

(e)
Since \smash{$(\bfphi_n)_{n\in\calN}$} is any $\UNTF_\bbF(D,N)$,
every layer of \smash{$\bigparen{(\bfphi_{\sigma(n)})_{n\in\calN}}_{\sigma\in\Sym(\calN)}$} is a $\UNTF_\bbF(D,N)$.
To see that they are stratified, let $\bfPhi$ be the synthesis matrix of \smash{$(\bfphi_n)_{n\in\calN}$} and note that for any distinct $n_1,n_2\in\calN$ and any distinct $\hat{n}_1,\hat{n}_2\in\calN$,
there are exactly $(N-2)!$ permutations $\sigma$ of $\calN$ such that
$\bigparen{\sigma(n_1),\sigma(n_2)}=(\hat{n}_1,\hat{n}_2)$ and so
\begin{align*}
\tfrac1{\#(\Sym(\calN))}\sum_{\sigma\in\Sym(\calN)}\abs{\ip{\bfphi_{\sigma(n_1)}}{\bfphi_{\sigma(n_2)}}}^2
&=\tfrac{(N-2)!}{N!}\sum_{\hat{n}_1\in\calN}\sum_{\hat{n}_2\in\calN\backslash\set{\hat{n}_1}}
\abs{\ip{\bfphi_{\hat{n}_1}}{\bfphi_{\hat{n}_2}}}^2\\
&=\tfrac{1}{N(N-1)}\sum_{\hat{n}_1\in\calN}\Bigparen{\,\sum_{\hat{n}_2\in\calN}
\abs{\ip{\bfphi_{\hat{n}_1}}{\bfphi_{\hat{n}_2}}}^2-1}\\
&=\tfrac{1}{N(N-1)}\Bigparen{\,\sum_{\hat{n}_2\in\calN}\norm{\bfPhi^*\bfphi_{\hat{n}_2}}^2-N}.
\end{align*}
Since
\smash{$\norm{\bfPhi^*\bfphi_{\hat{n}_2}}^2
=\ip{\bfphi_{\hat{n}_2}}{\bfPhi\bfPhi^*\bfphi_{\hat{n}_2}}
=\tfrac{N}{D}\norm{\bfphi_{\hat{n}_2}}^2
=\tfrac{N}{D}$}
for any $\hat{n}_2\in\calN$,
this reduces to
\smash{$\tfrac{N-D}{D(N-1)}$}
for any distinct $n_1,n_2\in\calN$.
\end{proof}

\begin{remark}
When $\bbF=\bbR$, $D=2$ and $R=2$,
the \textit{stratified Gerzon bound} given in Theorem~\ref{thm.SUNTF}(b) gives that
$N\leq 5$.
Notably, this bound is achieved by the $\SUNTF_\bbR(2,5,2)$ of Example~\ref{ex.SUNTF(2,5,2)}.
This suggests that this bound is not overly coarse.
In contrast, directly applying the chordal Gerzon bound~\eqref{eq.Gerzon} to the $\ECTFF_\bbF(DR,R,N)$ that arises from a $\SUNTF_\bbF(D,N,R)$ with $D>1$ via Theorem~\ref{thm.SUNTF to ECTFF} only yields the weaker bound
$N\leq DR+\tfrac{DR(DR-1)}{2}[\bbF:\bbR]$,
which for example only gives $N\leq 10$ when $\bbF=\bbR$, $D=2$ and $R=2$.
For that matter,
if one mimics the standard proof of~\eqref{eq.Gerzon} (given in the appendix) too closely,
and instead computes the Gram matrix of
\smash{$\bigparen{\oplus_{r\in\calR}\bfphi_{r,n}^{}\bfphi_{r,n}^{*}}_{n\in\calN}$}
to show that it is a linearly independent finite sequence in the direct sum over $r\in\calR$ of the real spaces of self-adjoint operators on $\calE_r$,
one only obtains the weaker bound $N\leq DR+\tfrac{DR(D-1)}{2}[\bbF:\bbR]$,
which for example only gives $N\leq 6$ when $\bbF=\bbR$, $D=2$ and $R=2$.
We further note that in the special case where $R=1$,
the careful argument of Theorem~\ref{thm.SUNTF}(b) recovers a result from the literature surrounding Zauner's conjecture, namely that any $D^2$ noncollinear equiangular vectors in $\bbC^D$ necessarily form an $\ETF_\bbC(D,D^2)$.
\end{remark}

In regard to Theorem~\ref{thm.SUNTF}(c),
our only additional remark is that if $(\calF_n)_{n\in\calN}$ is an $\ECTFF_\bbF(D,R,N)$ for $\calE$ with $R<D$ then its \textit{spatial complement}~\cite{CasazzaFMYZ11,FickusMW25} $(\calF_n^\perp)_{n\in\calN}$ is an $\ECTFF_\bbF(D,D-R,N)$ for $\calE$,
but the latter need not arise from a SUNTF via Theorem~\ref{thm.SUNTF to ECTFF} even when the former does.
For example, the unit-norm vertices of a regular tetrahedron centered at the origin form an $\ETF_\bbR(3,4)$,
namely a $\SUNTF_\bbR(3,4,1)$,
which yields an $\EITFF_\bbR(3,1,4)$ via Theorem~\ref{thm.SUNTF to ECTFF}.
Its spatial complement is an $\ECTFF_\bbR(3,2,4)$,
which does not arise from a SUNTF since $2$ does not divide $3$.

Much of the remaining work below is inspired by parts (b), (d) and (e) of Theorem~\ref{thm.SUNTF}.
In particular, given $\bbF$ to be either $\bbR$ or $\bbC$,
and integers $N\geq D> 1$,
consider the set of positive integers $R$ for which a $\SUNTF_\bbF(D,N,R)$ exists.
By Theorem~\ref{thm.SUNTF}(e) this set contains $N!$ and so is nonempty.
Theorem~\ref{thm.SUNTF}(d) gives that this set is closed under linear combinations whose coefficients are nonnegative integers, not identically zero.
Moreover, whereas the traditional Gerzon bound precludes the existence of $\ETF_\bbF(D,N)$ for a fixed $D>1$ and sufficiently large $N$,
the necessary condition of Theorem~\ref{thm.SUNTF}(b) can be satisfied for any $N\geq D>1$ provided $R$ is taken sufficiently large.
This suggests the following problem:

\begin{problem}
\label{prob.SUNTF}
Given $\bbF$ to be either $\bbR$ or $\bbC$ and integers $N\geq D>1$,
does there exist a positive integer $R_\bbF(D,N)$ such that a $\SUNTF_\bbF(D,N,R)$ exists for any integer $R\geq R_\bbF(D,N)$?
\end{problem}

To be clear, the answer to Problem~\ref{prob.SUNTF} is clearly ``yes, with $R_\bbF(D,N)=1$'' if and only if an $\ETF_\bbF(D,N)$ exists.
From this perspective, Problem~\ref{prob.SUNTF} is a generalization of the ETF existence problem.
Our geometric intuition of the algebraic variety of all $\UNTF_\bbF(D,N)$,
based on the results of~\cite{BenedettoF03,DykemaS06,CahillFMPS13,CahillMS17,NeedhamS23} for example,
leads us to suspect that the answer is ``yes'' for any such $\bbF$, $D$ and $N$,
and that with the exception of the degenerate cases with $N\in\set{D,D+1}$,
there is enough ``freedom'' in each layer that any stack of sufficiently many of them can be made flat.
Later on in this section, we prove that the answer is indeed ``yes'' when $\bbF=\bbC$ and $N$ is a prime power.
Our approach relies on Fourier analysis and Gauss sums over finite fields.
The following example illustrates both this problem and this approach.

\begin{example}
\label{ex.SUNTF(2,5,5)}
Recall the $\SUNTF_\bbR(2,5,2)$ of Example~\ref{ex.SUNTF(2,5,2)} which,
via Theorem~\ref{thm.SUNTF to ECTFF},
yields the $\ECTFF_\bbR(4,2,5)$ of Example~\ref{ex.ECTFF(4,2,5)}.
For any positive integer $M$,
applying Theorem~\ref{thm.SUNTF}(d) to $M$ copies of this $\SUNTF_\bbR(2,5,2)$ yields a $\SUNTF_\bbR(2,5,2M)$.
(Laying $M$ two-layer ``flat stacks'' on top of each other yields a single such stack with $2M$ layers.)
The interested reader may verify that the $\ECTFF_\bbR(4M,2M,5)$ that arises from this $\SUNTF_\bbR(2,5,2M)$ via Theorem~\ref{thm.SUNTF to ECTFF} is the direct sum of $M$ copies of this $\ECTFF_\bbR(4,2,5)$ in the manner of Lemma~\ref{lem.FF direct sums}.
It follows that when $\bbF=\bbR$, $D=2$ and $N=5$,
the answer to Problem~\ref{prob.SUNTF} is ``yes'' if and only if there exists a $\SUNTF_\bbR(2,5,R)$ for some odd $R$.
Here, Theorem~\ref{thm.SUNTF}(e) does not help as $5!=120$ is even.

Now let $u:=\sqrt{\frac{3}{8}}$, $v:=\sqrt{\frac{5}{8}}$,
and consider the following UNTF $(\bfphi_n)_{n\in\bbF_5}$ for $\bbR^2$:
\begin{equation*}
\bfPhi=\left[\begin{array}{rrrrr}
1&\phm u&\phm u&\phm u&\phm u\\
0&\phm v&    -v&    -v&\phm v
\end{array}\right],
\quad
\bfPhi\bfPhi^*
=\left[\begin{array}{cc}\frac{5}{2}&0\smallskip\\0&\frac{5}{2}\end{array}\right],
\quad
\abs{\bfPhi^*\bfPhi}^2
=\left[\begin{array}{ccccc}
          1& \frac{3}{8}& \frac{3}{8}& \frac{3}{8}& \frac{3}{8}\smallskip\\
\frac{3}{8}&           1&\frac{1}{16}&\frac{1}{16}&           1\smallskip\\
\frac{3}{8}&\frac{1}{16}&           1&           1&\frac{1}{16}\smallskip\\
\frac{3}{8}&\frac{1}{16}&           1&           1&\frac{1}{16}\smallskip\\
\frac{3}{8}&           1&\frac{1}{16}&\frac{1}{16}&           1
\end{array}\right].
\end{equation*}
Notably, the entries of any cyclic subdiagonal of \smash{$\abs{\bfPhi^*\bfPhi}^2$} average to the squared Welch bound,
that is,
\smash{$\tfrac15(2\tfrac{3}{8}+2\tfrac{1}{16}+1)
=\tfrac{3}{8}
=\tfrac{5-2}{2(5-1)}$}.
As such, cycling the columns of $\bfPhi$ yields a \smash{$\SUNTF_\bbR(2,5,5)$},
namely \smash{$\bigparen{(\bfphi_{r+n})_{n\in\bbF_5}}_{r\in\bbF_5}$}.
Here, the synthesis matrices of the layers are
\begin{align*}
\bfPhi_0&=\left[\begin{array}{rrrrr}
\phm 1&\phm u&\phm u&\phm u&\phm u\\
\phm 0&\phm v&    -v&    -v&\phm v
\end{array}\right],&
\bfPhi_1&=\left[\begin{array}{rrrrr}
\phm u&\phm u&\phm u&\phm u&\phm 1\\
\phm v&    -v&    -v&\phm v&\phm 0
\end{array}\right],&
\bfPhi_2&=\left[\begin{array}{rrrrr}
\phm u&\phm u&\phm u&\phm 1&\phm u\\
    -v&    -v&\phm v&\phm 0&\phm v
\end{array}\right],\\
\bfPhi_3&=\left[\begin{array}{rrrrr}
\phm u&\phm u&\phm 1&\phm u&\phm u\\
    -v&\phm v&\phm 0&\phm v&    -v
\end{array}\right],&
\bfPhi_4&=\left[\begin{array}{rrrrr}
\phm u&\phm 1&\phm u&\phm u&\phm u\\
\phm v&\phm 0&\phm v&    -v&    -v
\end{array}\right].
\end{align*}
As such, Theorem~\ref{thm.SUNTF}(d) implies that a $\SUNTF_\bbR(2,5,R)$ exists for any positive integer $R$ that is a nonnegative integer linear combination of $2$ and $5$.
This is a resolved instance of the \textit{Frobenius coin problem},
and so here Problem~\ref{prob.SUNTF} is answered in the affirmative with $R_\bbR(2,5)=4$.
We do not yet know whether a $\SUNTF_\bbR(2,5,3)$ exists, and we leave this for future research.
This example does highlight one approach to Problem~\ref{prob.SUNTF}:
for a given $\bbF$, $D$ and $N$,
its answer is ``yes'' if and only if there exist a $\SUNTF_\bbF(D,N,R_1)$ and a $\SUNTF_\bbF(D,N,R_2)$ where $R_1$ and $R_2$ are relatively prime.
\end{example}

\subsection{Harmonic stratified unit norm tight frames}

For any subset $\calD$ of $\calG$ of cardinality $D>0$,
the corresponding \textit{harmonic frame} is the finite sequence $(\bfphi_\gamma)_{\gamma\in\widehat{\calG}}$, \smash{$\bfphi_\gamma(d):=D^{-\frac12}\gamma(d)$}.
It is a UNTF for \smash{$\bbC^\calD$} since $\bfGamma$ is a complex Hadamard matrix.
For any \smash{$\gamma_1,\gamma_2\in\widehat{\calG}$},
the corresponding entry of its Gram matrix is naturally expressed in terms of the discrete Fourier transform of the characteristic function $\bfchi_\calD$ of $\calD$:
\begin{equation}
\label{eq.harmonic Gram}
\ip{\bfphi_{\gamma_1}}{\bfphi_{\gamma_2}}
=\tfrac{1}{D}\sum_{d\in\calD}\overline{\gamma_1(d)}\gamma_2(d)
=\tfrac{1}{D}\sum_{g\in\calG}\overline{(\gamma_1^{}\gamma_2^{-1})(g)}\bfchi_\calD(g)
=\tfrac{1}{D}(\bfGamma^*\bfchi_\calD)(\gamma_1^{}\gamma_2^{-1}).
\end{equation}
That is, $\bfPhi^*\bfPhi$ is the $\widehat{\calG}$-circulant matrix whose impulse response is $\frac1D\bfGamma^*\bfchi_\calD$.
In particular, $\bfPhi^*\bfPhi$ is real-valued if and only if $\bfGamma^*\bfchi_\calD$ is real-valued,
namely if and only if $\bfchi_\calD$ equals its conjugate reversal $\bfchi_\calD^{\star}$.
Thus, a harmonic frame is (unitarily equivalent to) a UNTF with a real-valued synthesis matrix if and only if $\calD$ is \textit{reversible}, namely $\calD=-\calD:=\set{-g : g\in\calD}$.
When this occurs,
we can use the following process to convert the (complex) $\calD\times\widehat{\calG}$ synthesis matrix $\bfPhi$ for the corresponding harmonic frame into a unitarily equivalent real-valued one.
First, note that for any $g\in\calD$ for which $-g=g$,
the corresponding row of $\bfPhi$ is already real-valued since
\smash{$\overline{\gamma(g)}
=\gamma(-g)
=\gamma(g)$}
for each $\gamma\in\widehat{\calG}$.
(We caution that $\calD$ may contain many such $g$,
such as when $\calG$ is an elementary abelian \mbox{$2$-group}.)
When $\calD$ is reversible,
$\set{g\in\calD: -g\neq g}$ can be partitioned into opposite pairs and since $\gamma(-g)=\overline{\gamma(g)}$ for each $\gamma\in\widehat{\calG}$,
the corresponding pair of rows of $\bfPhi$ can be unitarily transformed into a pair of real-valued rows as follows:
\begin{equation}
\label{eq.conjugate pairs}
\tfrac1{\sqrt{2}}\left[\begin{array}{rr}1&1\\-\rmi&\rmi\end{array}\right]
\left[\begin{array}{l}\bfphi_\gamma(g)\\\bfphi_\gamma(-g)\end{array}\right]
=\tfrac1{\sqrt{2D}}\left[\begin{array}{rr}1&1\\-\rmi&\rmi\end{array}\right]
\left[\begin{array}{l}\gamma(g)\smallskip\\\overline{\gamma(g)}\end{array}\right]
=\sqrt{2}\left[\begin{array}{c}\real(\bfphi_\gamma(g))\\\imag(\bfphi_\gamma(g)) \end{array}\right].
\end{equation}

In general, since $\bfGamma^*$ transforms convolution and conjugate reversal into pointwise multiplication and conjugation, respectively,
the squared modulus of \eqref{eq.harmonic Gram} is naturally expressed in terms of the \textit{autocorrelation} \smash{$\bfchi_\calD^{}*\bfchi_\calD^{\star}$} of \smash{$\bfchi_\calD$}:
\begin{equation}
\label{eq.DFT of char mod squared}
\abs{\ip{\bfphi_{\gamma_1}}{\bfphi_{\gamma_2}}}^2
=\tfrac{1}{D^2}\abs{(\bfGamma^*\bfchi_\calD)(\gamma_1^{}\gamma_2^{-1})}^2
=\tfrac1{D^2}\bigparen{\bfGamma^*(\bfchi_\calD^{}*\bfchi_\calD^{\star})}(\gamma_1^{}\gamma_2^{-1}).
\end{equation}
For each $g\in\calG$,
the corresponding value of this autocorrelation gives the number of ways $g$ can be expressed as a difference of members of $\calD$:
\begin{equation}
\label{eq.auto of char}
(\bfchi_\calD^{}*\bfchi_\calD^{\star})(g)
=\sum_{\substack{g_1,g_2\in\calG\\g_1+g_2=g}}\bfchi_\calD(g_1)\overline{\bfchi_\calD(-g_2)}
=\#\set{(d_1,d_2)\in\calD\times\calD: g=d_1-d_2}.
\end{equation}
These same values appear in the classical algebro-combinatorial concept of a difference set and its various generalizations.
Specifically, $\calD$ is a \textit{difference set} for $\calG$ if there exists $C\geq0$ such that $\#\set{(d_1,d_2)\in\calD\times\calD: g=d_1-d_2}=C$ for any $g\in\calG\backslash\set{0}$.
More generally, a finite sequence $(\calD_r)_{r\in\calR}$ of $R=\#(\calR)$ subsets of $\calG$ of equal cardinality $D>0$ is a \textit{difference family} for $\calG$ if there exists $C\geq0$ such that
$\sum_{r\in\calR}\#\set{(d_1,d_2)\in\calD_r\times\calD_r: g=d_1-d_2}=C$
for any $g\in\calG\backslash\set{0}$.
For example, $(\set{1,4},\set{2,3})$ is a difference family for $\bbZ_5$ with $R=2$, $D=2$ and $C=1$ since each of the four nonzero elements $\set{1,2,3,4}$ of $\bbZ_5$ appears exactly once in the combined difference tables of $\set{1,4}$ and $\set{2,3}$:
\begin{equation}
\label{eq.DF(5,2,1)}
\begin{array}{c|cc}
-&1&4\\\hline
1&0&2\\
4&3&0
\end{array}\,,
\quad
\begin{array}{c|cc}
-&2&3\\\hline
2&0&4\\
3&1&0
\end{array}\,.
\end{equation}
Here in general,
since a finite group of order $G>1$ has $G-1$ nonidentity elements whereas for each $r\in\calR$,
there are $D(D-1)$ pairs $(d_1,d_2)\in\calD_r\times\calD_r$ with $d_1\neq d_2$,
we necessarily have \smash{$C=\frac{RD(D-1)}{G-1}$}.
For instance,  \smash{$C=\frac{2(2)(2-1)}{(5-1)}=1$} in the prior example.
In summary, a finite sequence $(\calD_r)_{r\in\calR}$ of $R>0$ subsets of a finite abelian group $\calG$ of order $G>1$, each of cardinality $D>0$, is a difference family for $\calG$ if and only if
\begin{equation}
\label{eq.def of difference family}
\sum_{r\in\calR}\#\set{(d_1,d_2)\in\calD_r\times\calD_r: g=d_1-d_2}
=\tfrac{RD(D-1)}{G-1},
\quad\forall g\in\calG\backslash\set{0}.
\end{equation}
We now show that difference families equate to SUNTFs whose layers are harmonic frames:

\begin{theorem}
\label{thm.DF to SUNTF}
Let $\calG$ be a finite abelian group of order $G>1$,
let $D$ and $R$ be positive integers,
and let $\calR$ be a finite set of cardinality $R$.
For each $r\in\calR$,
let $\calD_r$ be a subset of $\calG$ of cardinality $D$,
and let \smash{$(\bfphi_{r,\gamma})_{\gamma\in\widehat{\calG}}$} be the corresponding harmonic frame for $\bbC^{\calD_r}$,
\smash{$\bfphi_{r,\gamma}(d):=D^{-\frac12}\bfgamma(d)$}.
Then \smash{$\bigparen{(\bfphi_{r,\gamma})_{\gamma\in\widehat{\calG}}}_{r\in\calR}$} is a \smash{$\SUNTF_\bbC(D,G,R)$} if and only if \smash{$(\calD_r)_{r\in\calR}$} is a difference family for $\calG$.
When this occurs and each \smash{$\calD_r$} is reversible, a \smash{$\SUNTF_\bbR(D,G,R)$} exists.
\end{theorem}

\begin{proof}[Proof.]
For each $r\in\calR$, \smash{$(\bfphi_{r,\gamma})_{\gamma\in\widehat{\calG}}$} is a finite sequence of $\#(\widehat{\calG})=\#(\calG)=G$ unit vectors in the space $\bbC^{\calD_r}$ of dimension $\#(\calD_r)=D$.
By Theorem~\ref{thm.SUNTF}(a), $\bigparen{(\bfphi_{r,\gamma})_{\gamma\in\widehat{\calG}}}_{r\in\calR}$ is thus a $\SUNTF_\bbC(D,G,R)$ if and only if
\begin{equation*}
\tfrac1R\sum_{r\in\calR}\abs{\ip{\bfphi_{r,\gamma_1}}{\bfphi_{r,\gamma_2}}}^2
=\left\{\begin{array}{cl}1,&\ \gamma_1=\gamma_2\\ \frac{G-D}{D(G-1)},&\ \gamma_1\neq\gamma_2\end{array}\right\}
=\tfrac{(D-1)G}{D(G-1)}\bfdelta_1(\gamma_1^{}\gamma_2^{-1})+\tfrac{G-D}{D(G-1)}
\end{equation*}
for any $\gamma_1,\gamma_2\in\widehat{\calG}$.
Moreover, since each finite sequence \smash{$(\bfphi_{r,\gamma})_{\gamma\in\widehat{\calG}}$} is a harmonic frame,
\eqref{eq.DFT of char mod squared} gives
\begin{equation*}
\tfrac1R\sum_{r\in\calR}\abs{\ip{\bfphi_{r,\gamma_1}}{\bfphi_{r,\gamma_2}}}^2
=\tfrac1R\sum_{r\in\calR}\tfrac1{D^2}\bigparen{\bfGamma^*(\bfchi_{\calD_r}^{}*\bfchi_{\calD_r}^{\star})}(\gamma_1^{}\gamma_2^{-1})
=\tfrac1{D^2R}\biggparen{\bfGamma^*\sum_{r\in\calR}\bfchi_{\calD_r}^{}*\bfchi_{\calD_r}^{\star}}(\gamma_1^{}\gamma_2^{-1})
\end{equation*}
for any $\gamma_1,\gamma_2\in\widehat{\calG}$.
Together, these imply that $\bigparen{(\bfphi_{r,\gamma})_{\gamma\in\widehat{\calG}}}_{r\in\calR}$ is a $\SUNTF_\bbC(D,G,R)$ if and only if
\begin{equation*}
\tfrac1{D^2R}\bfGamma^*\sum_{r\in\calR}\bfchi_{\calD_r}^{}*\bfchi_{\calD_r}^{\star}
=\tfrac{(D-1)G}{D(G-1)}\bfdelta_1+\tfrac{G-D}{D(G-1)}\bfchi_{\widehat{\calG}}.
\end{equation*}
Since $\bfGamma$ is invertible with $\bfGamma^{-1}=\frac1G\bfGamma^*$ where $\bfGamma\bfdelta_1=\bfchi_\calG$ and $\bfGamma^*\bfdelta_0=\bfchi_{\widehat{\calG}}$, this is equivalent to having
\begin{equation*}
\tfrac{1}{R}\sum_{r\in\calR}\bfchi_{\calD_r}^{}*\bfchi_{\calD_r}^{\star}
=\tfrac{D^2}{G}\Bigparen{\tfrac{(D-1)G}{D(G-1)}\bfchi_{\calG}+\tfrac{G(G-D)}{D(G-1)}\bfdelta_0}
=\tfrac{D(G-D)}{G-1}\bfdelta_0+\tfrac{D(D-1)}{G-1}\bfchi_{\calG}.
\end{equation*}
Evaluating this at any $g\in\calG$ equivalently gives via~\eqref{eq.auto of char} that
\begin{equation*}
\tfrac{1}{R}\sum_{r\in\calR}\#\set{(d_1,d_2)\in\calD_r\times\calD_r: g=d_1-d_2}
=\left\{\begin{array}{cl}
D,&\ g=0,\\
\tfrac{D(D-1)}{G-1},&\ g\neq0.
\end{array}\right.
\end{equation*}
Now note that $\#\set{(d_1,d_2)\in\calD_r\times\calD_r: 0=d_1-d_2}=\#(\calD_r)=D$ for each $r\in\calR$,
and so the above condition at $g=0$ is superfluous.
Thus, $\bigparen{(\bfphi_{r,\gamma})_{\gamma\in\widehat{\calG}}}_{r\in\calR}$ is a $\SUNTF_\bbC(D,G,R)$ if and only if~\eqref{eq.def of difference family} holds, namely if and only if $(\calD_r)_{r\in\calR}$ is a difference family for $\calG$.
The final conclusion immediately follows from the previous discussion regarding~\eqref{eq.conjugate pairs}.
\end{proof}

\begin{remark}
Theorem~\ref{thm.SUNTF to ECTFF} converts the $\SUNTF_\bbC(D,N,R)$ that arises from a difference family via Theorem~\ref{thm.DF to SUNTF} into an $\ECTFF_\bbC(DR,R,N)$.
Notably, Theorem~4.2 of~\cite{FickusMW25} constructs an ECTFF with the same parameters from such a difference family.
As we now explain, this is not a coincidence:
our Theorems~\ref{thm.SUNTF to ECTFF} and~\ref{thm.DF to SUNTF} can be regarded as a streamlined generalization of this construction of~\cite{FickusMW25}.

Applying Theorem~\ref{thm.SUNTF to ECTFF} to the vectors \smash{$\bigparen{(\bfphi_{r,\gamma})_{\gamma\in\widehat{\calG}}}_{r\in\calR}$}
of Theorem~\ref{thm.DF to SUNTF} yields isometries
\smash{$(\widehat{\bfPhi}_\gamma)_{\gamma\in\widehat{\calG}}$} from \smash{$\bbC^\calR$} into \smash{$\oplus_{r\in\calR}\bbC^{\calD_r}$},
which we identify with \smash{$\bbC^{\widehat{\calD}}$} where \smash{$\widehat{\calD}:=\sqcup_{r\in\calR}(\set{r}\times\calD_r)$}.
For any $r,r'\in\calR$ and $g\in\calD_r$,
the corresponding entry of each \smash{$\widehat{\calD}\times\calR$} matrix \smash{$\widehat{\bfPhi}_\gamma$} is
\begin{equation*}
\widehat{\bfPhi}_\gamma\bigparen{(r,g),r'}
=\left\{\begin{array}{cl}
\bfphi_{r,\gamma}(g),&\ r'=r\smallskip\\
0,&\ r'\neq r
\end{array}\right\}
=\tfrac{1}{\sqrt{D}}\left\{\begin{array}{cl}
\gamma(g),&\ r'=r,\smallskip\\
0,&\ r'\neq r.
\end{array}\right.
\end{equation*}
Now assume without loss of generality that the indexing set $\calR$ of the difference family is itself a finite abelian group,
such as $\bbZ_R$,
and let \smash{$\widehat{\calR}$} be its Pontryagin dual.
Identify the Pontryagin dual of $\calG\times\calR$ with \smash{$\widehat{\calG}\times\widehat{\calR}$} in the usual way,
namely so that its character table \smash{$\bfGamma_{\calG\times\calR}$} is the tensor product of the character tables $\bfGamma_{\calG}$ and $\bfGamma_{\calR}$ of $\calG$ and $\calR$, respectively,
implying
\smash{$\bfGamma_{\calG\times\calR}\bigparen{(g,r),(\gamma,\rho)}
=\gamma(g)\rho(r)$}.
The harmonic frame \smash{$(\widetilde{\bfphi}_{\gamma,\rho})_{(\gamma,\rho)\in\widehat{\calG}\times\widehat{\calR}}$}
arising from the subset \smash{$\widetilde{\calD}:=\sqcup_{r\in\calR}(\calD_r\times\set{r})$} of $\calG\times\calR$ thus has \smash{$\widetilde{\bfphi}_{\gamma,\rho}(g,r)=\frac1{\sqrt{DR}}\gamma(g)\rho(r)$} for any $r\in\calR$, \smash{$g\in\calD_r$}.
Letting \smash{$\bfP:\bbC^{\widehat{\calD}}\rightarrow\bbC^{\widetilde{\calD}}$} be the permutation matrix defined by \smash{$\bfP\bfdelta_{r,g}:=\bfdelta_{g,r}$} for any such $r$ and $g$,
and letting \smash{$\widetilde{\bfPhi}_\gamma:=\bfP\widehat{\bfPhi}_\gamma\tfrac1{\sqrt{R}}\bfGamma_\calR$} for each \smash{$\gamma\in\widehat{\calG}$}
thus yields isometries \smash{$(\widetilde{\bfPhi}_\gamma)_{\gamma\in\widehat{\calG}}$}
from \smash{$\bbC^{\widehat{\calR}}$} into \smash{$\bbC^{\widetilde{\calD}}$} that
are equivalent to \smash{$(\widehat{\bfPhi}_\gamma)_{\gamma\in\widehat{\calG}}$} with
\begin{equation*}
\widetilde{\bfPhi}_\gamma\bigparen{(g,r),\rho}
=\tfrac1{\sqrt{R}}\sum_{r'\in\calR}\widehat{\bfPhi}_\gamma\bigparen{(r,g),r'}\bfGamma_\calR(r',\rho)
=\tfrac1{\sqrt{DR}}\gamma(g)\rho(r)
=\widetilde{\bfphi}_{\gamma,\rho}(g,r)
\end{equation*}
for any $(g,r)\in\widetilde{\calD}$ and $\rho\in\widehat{\calR}$.
Thus, each $\widetilde{\bfPhi}_\gamma$ is the synthesis map for $(\widetilde{\bfphi}_{\gamma,\rho})_{\rho\in\widehat{\calR}}$.
Accounting for changes in notation,
the careful reader will recognize this to be the instance of the construction of Theorem~4.2 of~\cite{FickusMW25} in the special case where its ``$\calG$'' and ``$\calH$'' are $\calG\times\calR$ and $\calG\times\set{0}$, respectively,
with the annihilator ``$\calH^\perp$'' of the latter being $\set{1}\times\widehat{\calR}$.
In general, this result of~\cite{FickusMW25} characterizes when partitioning the vectors of a harmonic frame according to cosets of the annihilator of a subgroup of the underlying group yields the isometries of a (EC/EI) TFF.
\end{remark}

\begin{example}
Recall that $\calD_1:=\set{1,4}$ and $\calD_2:=\set{2,3}$ form a difference family for $\bbZ_5$,
as verified in~\eqref{eq.DF(5,2,1)}.
The synthesis matrices $\bfPhi_1$ and $\bfPhi_2$ of the associated harmonic frames are obtained by extracting the corresponding rows from the character table of $\bbZ_5$ and then normalizing columns,
namely
\begin{equation*}
\bfPhi_1
=\tfrac{1}{\sqrt{2}}\left[\begin{array}{cllll}
1&\omega  &\omega^2&\omega^3&\omega^4\\
1&\omega^4&\omega^3&\omega^2&\omega
\end{array}\right],
\quad
\bfPhi_2
=\tfrac{1}{\sqrt{2}}\left[\begin{array}{cllll}
1&\omega^2&\omega^4&\omega  &\omega^3\\
1&\omega^3&\omega  &\omega^4&\omega^2
\end{array}\right],
\quad
\omega:=\exp(\tfrac{2\pi\rmi}{5}).
\end{equation*}
By Theorem~\ref{thm.DF to SUNTF}, these two harmonic frames form a $\SUNTF_\bbC(2,5,2)$ which in turn yields an $\ECTFF_\bbC(4,2,5)$ via Theorem~\ref{thm.SUNTF to ECTFF}.
In fact, since $\calD_1$ and $\calD_2$ are both reversible,
these two UNTFs are unitarily equivalent to real-valued ones:
multiplying $\bfPhi_1$ and $\bfPhi_2$ above by
\smash{$\tfrac1{\sqrt{2}}\left[\begin{smallmatrix}\phm1&\phm1\\-\rmi&\phm\rmi\end{smallmatrix}\right]$} yields the synthesis matrices~\eqref{eq.ex.SUNTF(2,5,2).1} and~\eqref{eq.ex.SUNTF(2,5,2).2}, respectively, of the $\SUNTF_\bbR(2,5,2)$ of Example~\ref{ex.SUNTF(2,5,2)},
which yields the $\ECTFF_\bbR(4,2,5)$ of Example~\ref{ex.ECTFF(4,2,5)} via Theorem~\ref{thm.SUNTF to ECTFF}.
This highlights another advantage of the generalized approach of Theorems~\ref{thm.SUNTF to ECTFF} and~\ref{thm.DF to SUNTF} over that of~\cite{FickusMW25}:
the former sometimes yields real ECTFFs.

In the following result, we show that this example is not a fluke:
in general, for any nonempty subset $\calD$ of a finite field $\bbF_Q$,
\smash{$(x\calD)_{x\in\bbF_Q^\times}$} is a difference family for this field's additive group.
In particular, taking $\calD=\set{1,4}$ yields the difference family $(\calD_1,\calD_2,\calD_3,\calD_4)$ where,
due to symmetries in $\calD$ itself, $\calD_4=\calD_1=\calD=\set{1,4}$ and $\calD_3=\calD_2=2\calD=\set{2,3}$.
Here, since the difference family $(\calD_1,\calD_2,\calD_3,\calD_4)$ consists of two copies of $(\calD_1,\calD_2)$,
this directly implies that the latter is itself a difference family,
obviating the need for the calculations of~\eqref{eq.DF(5,2,1)}.
Moreover, since this particular $\calD$ is reversible, the same is true for $x\calD$ for any \smash{$x\in\bbF_Q^\times$}.
\end{example}

\begin{theorem}
\label{thm.DF construction}
A $\SUNTF_\bbR(D,Q,Q-1)$ exists for any prime power $Q$ and positive integer $D\leq Q$.
\end{theorem}

\begin{proof}[Proof.]
We claim that the additive group of the finite field $\bbF_Q$ of order $Q$ contains a reversible subset $\calD$ of cardinality $D$.
Indeed, if $Q$ is even then every subset $\calD$ of cardinality $D$ is reversible.
Else, $Q$ is odd and $\bbF_Q$ partitions as $\bbF_Q=\set{0}\sqcup\calS\sqcup(-\calS)$ for some subset $\calS$ of \smash{$\bbF_Q$} of cardinality \smash{$\frac{Q-1}{2}$}.
In this case, if $D$ is odd, let $\calD=\set{0}\sqcup\calC\sqcup(-\calC)$ where $\calC$ is any subset of $\calS$ of cardinality $\frac{D-1}{2}$, while if $D$ is even, $\calD=\calC\sqcup(-\calC)$ where $\calC$ is any subset of $\calS$ of cardinality $\frac{D}{2}$;
the latter exists since
\smash{$\frac{D}{2}
=\lfloor\frac{D}{2}\rfloor
\leq\lfloor\frac{Q}{2}\rfloor
=\frac{Q-1}{2}$}.

Having such a set $\calD$, for each $x\in\bbF_Q^\times$,
let $\calD_x:=x\calD=\set{xy: y\in\calD}$.
Then each $\calD_x$ is reversible since $-\calD_x=(-x)\calD=x(-\calD)=x\calD=\calD_x$.
Moreover, since $x\neq0$, the map $y\mapsto xy$ is injective over $\bbF_Q$,
and so $\#(\calD_x)=\#(\calD)=D$.
By Theorem~\ref{thm.DF to SUNTF},
it thus suffices to show that \smash{$(\calD_x)_{x\in\bbF_Q^\times}$} is a difference family for the additive group of $\bbF_Q$.
This follows via~\eqref{eq.def of difference family}: for any \smash{$y\in\bbF_Q^\times$},
\begin{align*}
\sum_{x\in\bbF_Q^\times}\#\set{(d_1,d_2)\in\calD_x\times\calD_x: y=d_1-d_2}
&=\sum_{x\in\bbF_Q^\times}\#\set{(z_1,z_2)\in\calD\times\calD: y=x(z_1-z_2)}\\
&=\sum_{x\in\bbF_Q^\times}\#\set{(z_1,z_2)\in\calD\times\calD: \tfrac{z_1-z_2}{y}=\tfrac1x}\\
&=\#\set{(z_1,z_2)\in\calD\times\calD: \tfrac{z_1-z_2}{y}\neq0}\\
&=\#\set{(z_1,z_2)\in\calD\times\calD: z_1\neq z_2}\\
&=D(D-1),
\end{align*}
which indeed equals
``$\tfrac{RD(D-1)}{G-1}$'' when $R=\#(\calR)=\#(\bbF_Q^\times)=Q-1$ and $G=\#(\calG)=\#(\bbF_Q)=Q$.
\end{proof}

\subsection{Stratified unit norm tight frames from Gauss sums}

Recall that a character of the additive group of a finite field $\bbF_Q$ is a function $\gamma:\bbF_Q\rightarrow\bbT$ such that $\gamma(x+y)=\gamma(x)\gamma(y)$ for any $x,y\in\bbF_Q$.
We call any such $\gamma$ an \textit{additive character} of $\bbF_Q$,
and denote the set of them as \smash{$\widehat{\bbF}_Q$}.
As before, we identify any \smash{$\gamma\in\widehat{\bbF}_Q$} with the vector \smash{$\bfgamma\in\bbC^{\bbF_Q}$} with $\bfgamma(x):=\gamma(x)$ for any $x\in\bbF_Q$.
Similarly, a character of the (cyclic) multiplicative group \smash{$\bbF_Q^{\times}$} of $\bbF_Q$ is a function
\smash{$\theta:\bbF_Q^\times\rightarrow\bbT$} such that $\theta(xy)=\theta(x)\theta(y)$ for any \smash{$x,y\in\bbF_Q^\times$}.
We call any such $\theta$ a \textit{multiplicative character} of $\bbF_Q$,
and denote the set of them as \smash{$\widehat{\bbF}_Q^\times$}.
We extend any \smash{$\theta\in\widehat{\bbF}_Q^\times$} to a vector $\bftheta\in\bbC^{\bbF_Q}$ by letting $\bftheta(0):=0$ and $\bftheta(x)=\theta(x)$ for any \smash{$x\in\bbF_Q^\times$}.
(When $Q$ is odd, $\bbF_Q$ has a unique real-valued nonidentity multiplicative character $\theta$ and its extension $\bftheta$ is the \textit{quadratic character} known as the \textit{Legendre symbol} when $Q$ is prime.)
Any such $\bftheta$ satisfies $\bftheta(xy)=\bftheta(x)\bftheta(y)$ for any $x,y\in\bbF_Q$.
A \textit{Gauss sum} over $\bbF_Q$ is an inner product
\begin{equation}
\label{eq.Gauss sum}
\ip{\bfgamma}{\bftheta}
=\sum_{y\in\bbF_Q}\overline{\bfgamma(y)}\bftheta(y)
=\sum_{y\in\bbF_Q^\times}\overline{\gamma(y)}\theta(y)
\end{equation}
of an additive character $\bfgamma$ and extended multiplicative character $\bftheta$ of $\bbF_Q$.
In this section, we use multiplicative characters and properties of Gauss sums to construct SUNTFs.
We rely on the following result, implicit in the proof of Theorem~5.2 of~\cite{FickusIJM23}:

\begin{lemma}[Theorem~5.2 of~\cite{FickusIJM23}]
\label{lem.autocor of mult char}
Let $Q$ be a prime power and let \smash{$\bftheta\in\bbC^{\bbF_Q}$} be the extension of a nonidentity multiplicative character of $\bbF_Q$ with $\bftheta(0):=0$.
Then the $\bbF_Q$-translates \smash{$(\brT^x\bftheta)_{x\in\bbF_Q}$} of $\bftheta$ form a regular simplex for their span,
with $\norm{\bftheta}^2=Q-1$ and $\ip{\brT^x\bftheta}{\bftheta}=-1$ when $x\neq0$.
\end{lemma}

For the sake of completeness,
an elementary proof of this result is provided in the appendix.
We now use it to generalize the $\SUNTF_\bbR(2,5,5)$ of Example~\ref{ex.SUNTF(2,5,5)},
taking inspiration from the ETF construction of~\cite{Renes07,Strohmer08}:

\begin{theorem}
\label{thm.Gauss sums}
A $\SUNTF_\bbC(D,Q,Q)$ exists for any prime power $Q$ and positive integer $D\leq Q$.
Such a SUNTF can be chosen to be real if either $D$ or $Q$ is odd.
\end{theorem}

\begin{proof}[Proof.]
The result is trivial if $D\in\set{1,Q}$ and so we assume that $2\leq D\leq Q-1$ for the remainder of this proof.
Let \smash{$(\theta_d)_{d=2}^{D}$} be distinct nontrivial multiplicative characters of \smash{$\bbF_Q$}.
Define vectors \smash{$(\bfphi_x)_{x\in\bbF_Q}$} in \smash{$\bbC^D$} by letting $\bfphi_0:=\bfdelta_1$ and as follows for each $x\in\bbF_Q^\times$:
\begin{equation}
\label{eq.pf.Gauss.1}
\bfphi_x=u\bfdelta_1+v\sum_{d=2}^{D}\theta_d(x)\bfdelta_d
\quad\text{where}\quad
u:=\sqrt{\tfrac{Q-D}{D(Q-1)}},\
v:=\sqrt{\tfrac{Q}{D(Q-1)}}.
\end{equation}
We claim that \smash{$\bigparen{(\bfphi_{x+y})_{y\in\bbF_Q}}_{x\in\bbF_Q}$} is a $\SUNTF_\bbC(D,Q,Q)$.
Note that $\norm{\bfphi_0}=\norm{\bfdelta_1}=1$ and that
\begin{equation*}
\norm{\bfphi_x}^2
=\biggnorm{u\bfdelta_1+v\sum_{d=2}^{D}\theta_d(x)\bfdelta_d}^2
=u^2+v^2\sum_{d=2}^{D}\abs{\theta_d(x)}^2
=u^2+(D-1)v^2
=1
\end{equation*}
for any $x\in\bbF_Q^\times$.
By Theorem~\ref{thm.SUNTF}(a), it thus suffices to show that
\begin{equation*}
u^2
=\tfrac{Q-D}{D(Q-1)}
=\tfrac1{Q}\sum_{x\in\bbF_Q}\abs{\ip{\bfphi_{x+y_1}}{\bfphi_{x+y_2}}}^2
=\tfrac1{Q}\sum_{x\in\bbF_Q}\abs{\ip{\bfphi_{x}}{\bfphi_{x-(y_1-y_2)}}}^2
\end{equation*}
for any distinct $y_1,y_2\in\bbF_Q$, namely that
\begin{equation*}
Qu^2
=\sum_{x\in\bbF_Q}\abs{\ip{\bfphi_{x}}{\bfphi_{x-y}}}^2
=\abs{\ip{\bfphi_{0}}{\bfphi_{-y}}}^2
+\abs{\ip{\bfphi_{y}}{\bfphi_{0}}}^2
+\sum_{x\in\bbF_Q\backslash\set{0,y}}\abs{\ip{\bfphi_{x}}{\bfphi_{x-y}}}^2
\end{equation*}
for any $y\in\bbF_Q^\times$.
In light of~\eqref{eq.pf.Gauss.1} and the fact that $\bfphi_0=\bfdelta_1$, this reduces to having
\begin{equation*}
Qu^2
=2u^2+\sum_{x\in\bbF_Q\backslash\set{0,y}}\biggabs{u^2+v^2\sum_{d=2}^{D}\overline{\bftheta_d(x)}\bftheta_d(x-y)}^2
\end{equation*}
for any $y\in\bbF_Q^\times$.
Continuing to simplify,
note that since $\bftheta_d(0)=0$ whenever $d\geq 2$,
it thus suffices to have
\begin{equation*}
Qu^2
=2u^2-2u^4+\sum_{x\in\bbF_Q}\biggabs{u^2+v^2\sum_{d=2}^{D}\overline{\bftheta_d(x)}\bftheta_d(x-y)}^2
\end{equation*}
for any $y\in\bbF_Q^\times$.
Expanding the squared modulus and rearranging, this is equivalent to having
\begin{align*}
(Q-2)u^2(1-u^2)
&=2u^2v^2\real\biggparen{\,\sum_{d=2}^{D}\sum_{x\in\bbF_Q}\overline{\bftheta_d(x)}\bftheta_d(x-y)}\\
&\quad+v^4\sum_{d_1=2}^{D}\sum_{d_2=2}^{D}\sum_{x\in\bbF_Q}
\bftheta_{d_1}(x)\overline{\bftheta_{d_1}(x-y)}
\overline{\bftheta_{d_2}(x)}\bftheta_{d_2}(x-y)
\end{align*}
for any $y\in\bbF_Q^\times$.
For any $d_1,d_2\in\set{2,\dotsc,D}$,
the extension of \smash{$\eta_{d_1,d_2}:=\theta_{d_1}^{-1}\theta_{d_2}^{}$} satisfies \smash{$\bfeta_{d_1,d_2}(x)=\overline{\bftheta_{d_1}(x)}\bftheta_{d_2}(x)$} for any $x\in\bbF_Q$,
and so this further simplifies to having
\begin{equation*}
(Q-2)u^2(1-u^2)
=2u^2v^2\sum_{d=2}^{D}\real(\ip{\bftheta_d}{\brT^y\bftheta_d})
+v^4\sum_{d_1=2}^{D}\sum_{d_2=2}^{D}\ip{\bfeta_{d_1,d_2}}{\brT^y\bfeta_{d_1,d_2}}
\end{equation*}
for any $y\in\bbF_Q^\times$.
Since $\theta_d$ and $\eta_{d_1,d_2}$ are nonidentity multiplicative characters whenever $d_1\neq d_2$,
Lemma~\ref{lem.autocor of mult char} further simplifies this to having
\begin{equation*}
(Q-2)u^2(1-u^2)
=-2u^2v^2(D-1)
-v^4(D-1)(D-2)
+v^4\sum_{d=2}^{D}\ip{\bfeta_{d,d}}{\brT^y\bfeta_{d,d}}
\end{equation*}
for any $y\in\bbF_Q^\times$.
Since each $\bfeta_{d,d}$ is the extension of the identity multiplicative character with $\bfeta_{d,d}(0)=0$,
we have $\ip{\bfeta_{d,d}}{\brT^y\bfeta_{d,d}}=Q-2$ for any $d\in\set{2,\dotsc,D}$ and $y\in\bbF_Q^\times$,
meaning this sufficient condition simplifies to the following condition that is independent of $y$:
\begin{equation*}
(Q-2)u^2(1-u^2)
=-2u^2v^2(D-1)
-v^4(D-1)(D-2)
+v^4(D-1)(Q-2).
\end{equation*}
To see that it indeed holds,
recall that $1-u^2=(D-1)v^2$ and so it further reduces to having
\begin{equation*}
(Q-2)u^2
=-2u^2
-v^2(D-2)
+v^2(Q-2),
\end{equation*}
namely to having $Qu^2=v^2(Q-D)$,
which immediately follows from the definitions of $u$ and $v$.

Having shown that \smash{$\bigparen{(\bfphi_{x+y})_{y\in\bbF_Q}}_{x\in\bbF_Q}$} is indeed a $\SUNTF_\bbC(D,Q,Q)$,
now note that a real such SUNTF exists whenever the $[D]\times\bbF_Q$ complex synthesis matrix $\bfPhi$ of $(\bfphi_x)_{x\in\bbF_Q}$ is equivalent---via multiplication on the left by a unitary---to a matrix with real-valued entries.
In light of~\eqref{eq.conjugate pairs} and~\eqref{eq.pf.Gauss.1},
it suffices to choose $(\theta_d)_{d=2}^{D}$ as $\frac{D-1}{2}$ conjugate pairs when $D$ is odd,
and as $\frac{D}{2}-1$ conjugate pairs along with the quadratic character when instead $D$ is even and $Q$ is odd.
\end{proof}

Combining Theorems~\ref{thm.SUNTF}(d), \ref{thm.DF construction} and~\ref{thm.Gauss sums} yields the following positive answer to Problem~\ref{prob.SUNTF} whenever $N$ is a prime power:

\begin{corollary}
For any prime power $Q$ and positive integers $D\leq Q$ and $R\geq(Q-2)(Q-1)$,
a $\SUNTF_\bbC(D,Q,R)$ exists.
Such a SUNTF can be chosen to be real if either $D$ or $Q$ is odd.
\end{corollary}

\begin{proof}[Proof.]
For any such $Q$ and $D$,
Theorems~\ref{thm.DF construction} and~\ref{thm.Gauss sums} yield a $\SUNTF_\bbC(D,Q,R)$ when $R$ is either $Q-1$ or $Q$, respectively.
Applying Theorem~\ref{thm.SUNTF}(d) with copies of these two SUNTFs yields a $\SUNTF_\bbC(D,Q,R)$ for any positive integer $R$ that is a nonnegative integer linear combination of $Q-1$ and $Q$.
Since $Q-1$ and $Q$ are relatively prime,
Sylvester's classical solution to the Frobenius coin problem implies that this includes all
$R\geq
Q(Q-1)-(Q-1)-Q+1=(Q-2)(Q-1)$.
Applying the same argument to the real instances of such SUNTFs provided by Theorems~\ref{thm.DF construction} and~\ref{thm.Gauss sums} gives the analogous real result.
\end{proof}
We leave additional research into SUNTFs,
and Problem~\ref{prob.SUNTF} in particular,
for future work.

\backmatter

\bmhead{Acknowledgements}

JJ was supported by NSF DMS 2220320.
DGM was supported by NSF DMS 2220304.
All novel mathematics in this paper was developed by the authors without the aid of artificial intelligence tools.
The authors moreover wrote the manuscript in the traditional way,
with the exception that GPT-6 Astra Max served as a post facto proof checker and editor.
It found no major errors, some minor omissions, and a number of typographical errors, both mathematical and grammatical,
which the authors corrected manually.
The views expressed in this manuscript are those of the authors and do not reflect the official policy or position of the United States Air Force, the Department of Defense or the United States Government.

%
%
%

\begin{appendices}

\section{Supplementary Proofs}

\begin{proof}[Proof of Lemma~\ref{lem.equal Gram}.]
($\Leftarrow$) Immediate.
($\Rightarrow$) Note that $\bfL\bfL^*:\calE\rightarrow\calE$ has $\im(\bfL\bfL^*)=\im(\bfL)=\calE$ and so is invertible,
implying \smash{$\bfU:=\widetilde{\bfL}\bfL^*(\bfL\bfL^*)^{-1}$} is a well-defined linear map from $\calE$ into \smash{$\widetilde{\calE}$} where
\begin{align*}
\dim(\calE)
&=\rank(\bfL)
=\rank(\bfL^*\bfL)
=\rank(\widetilde{\bfL}^*\widetilde{\bfL})
=\rank(\widetilde{\bfL})
=\dim(\widetilde{\calE}),\\
\bfU^*\bfU
&=(\bfL\bfL^*)^{-1}\bfL\widetilde{\bfL}^*\widetilde{\bfL}\bfL^*(\bfL\bfL^*)^{-1}
=(\bfL\bfL^*)^{-1}\bfL\bfL^*\bfL\bfL^*(\bfL\bfL^*)^{-1}
=\brI,\\
\bfU^*\widetilde{\bfL}
&=(\bfL\bfL^*)^{-1}\bfL\widetilde{\bfL}^*\widetilde{\bfL}
=(\bfL\bfL^*)^{-1}\bfL\bfL^*\bfL
=\bfL.
\end{align*}
Thus, $\bfU$ is invertible, unitary and satisfies $\widetilde{\bfL}=\bfU\bfL$.
\end{proof}

\vspace{-\bigskipamount}

\begin{proof}[Proof of the Gerzon bound \eqref{eq.Gerzon}.]
The finite sequence $(\bfPi_n)_{n\in\calN}$ lies in the real Euclidean space of self-adjoint operators on $\calE$ with
\smash{$\ip{\bfPi_{n_1}}{\bfPi_{n_2}}_\Fro
=\Tr(\bfPi_{n_1}\bfPi_{n_2})
=\norm{\bfPhi_{n_1}^*\bfPhi_{n_2}^{}}_{\Fro}^2$} for any $n_1,n_2\in\calN$.
As such, $(\calF_n)_{n\in\calN}$ is EC with chordal constant $\varLambda$ if and only if the Gram matrix $\bfG$ of $(\bfPi_n)_{n\in\calN}$ is $R(1-\varLambda)\brI+R\varLambda\brJ$ where $\brJ$ is an $\calN\times\calN$ all-ones matrix.
When this occurs and at least two members of $(\calF_n)_{n\in\calN}$ are distinct then $\varLambda<1$.
This implies $\bfG$ is positive definite and so $(\bfPi_n)_{n\in\calN}$ is linearly independent, giving~\eqref{eq.Gerzon}.
\end{proof}

\vspace{-\bigskipamount}

\begin{proof}[Proof of the Welch bounds~\eqref{eq.Welch bounds}.]
Recall that $(\calF_n)_{n\in\calN}$ is a TFF for $\calE$ if and only if
\smash{$\bfPhi\bfPhi^*=\frac{NR}{D}\brI$}.
In general,
\begin{equation*}
0
\leq\norm{\bfPhi\bfPhi^*-\tfrac{NR}{D}\brI}_{\Fro}^2
=\biggnorm{\sum_{n\in\calN}\bfPhi_{n}^{}\bfPhi_{n}^{*}-\tfrac{NR}{D}\brI}_{\Fro}^2
=\Tr\biggparen{\Bigparen{\,\sum_{n\in\calN}\bfPhi_{n}^{}\bfPhi_{n}^{*}-\tfrac{NR}{D}\brI}^2}.
\end{equation*}
Expanding the square, distributing and cycling traces, and simplifying gives
\begin{equation}
\label{eq.pf.Welch bounds.1}
0
\leq\norm{\bfPhi\bfPhi^*-\tfrac{NR}{D}\brI}_{\Fro}^2
=\sum_{n_1\neq n_2}\norm{\bfPhi_{n_1}^{*}\bfPhi_{n_2}^{}}_\Fro^2
-\tfrac{NR(NR-D)}{D}.
\end{equation}
Assuming $N>1$,
rearranging and dividing by $N(N-1)R>0$ gives
\begin{equation}
\label{eq.pf.Welch bounds.2}
\tfrac{NR-D}{D(N-1)}
\leq\tfrac{1}{N(N-1)}\sum_{n_1\neq n_2}\tfrac{1}{R}\norm{\bfPhi_{n_1}^{*}\bfPhi_{n_2}^{}}_\Fro^2.
\end{equation}
Equality holds in~\eqref{eq.pf.Welch bounds.2} if and only if it holds in~\eqref{eq.pf.Welch bounds.1},
namely if and only if $(\calF_n)_{n\in\calN}$ is a TFF for $\calE$.
Bounding this average by its maximum gives
\begin{equation}
\label{eq.pf.Welch bounds.3}
\tfrac{NR-D}{D(N-1)}
\leq\tfrac{1}{N(N-1)}\sum_{n_1\neq n_2}\tfrac{1}{R}\norm{\bfPhi_{n_1}^{*}\bfPhi_{n_2}^{}}_\Fro^2
\leq\max_{n_1\neq n_2}\tfrac{1}{R}\norm{\bfPhi_{n_1}^{*}\bfPhi_{n_2}^{}}_\Fro^2
\leq\max_{n_1\neq n_2}\norm{\bfPhi_{n_1}^{*}\bfPhi_{n_2}^{}}_2^2,
\end{equation}
namely~\eqref{eq.Welch bounds}.
Moreover, \smash{$\mu_{\rmc}^2=\tfrac{NR-D}{D(N-1)}$} if and only if equality holds in the first and second inequalities of~\eqref{eq.pf.Welch bounds.3},
namely if and only if $(\calF_n)_{n\in\calN}$ is a TFF for $\calE$ and is EC, respectively.
Meanwhile, $\mu_{\rms}^2=\tfrac{NR-D}{D(N-1)}$ if and only if equality holds throughout
\begin{equation*}
\tfrac{NR-D}{D(N-1)}
\leq\tfrac{1}{N(N-1)}\sum_{n_1\neq n_2}\tfrac{1}{R}\norm{\bfPhi_{n_1}^{*}\bfPhi_{n_2}^{}}_\Fro^2
\leq\tfrac{1}{N(N-1)}\sum_{n_1\neq n_2}\norm{\bfPhi_{n_1}^{*}\bfPhi_{n_2}^{}}_2^2
\leq\max_{n_1\neq n_2}\norm{\bfPhi_{n_1}^{*}\bfPhi_{n_2}^{}}_2^2,
\end{equation*}
namely if and only if $(\calF_n)_{n\in\calN}$ is a TFF for $\calE$ where each pair of its subspaces is isoclinic with $\norm{\bfPhi_{n_1}^{*}\bfPhi_{n_2}^{}}_2^2$ being constant over all distinct $n_1,n_2\in\calN$,
namely if and only if $(\calF_n)_{n\in\calN}$ is an EITFF for $\calE$.
Notably, to show that $(\calF_n)_{n\in\calN}$ is an EITFF or ECTFF for $\calE$,
it suffices to show that $(\calF_n)_{n\in\calN}$ is EI or EC with chordal constant \smash{$\varLambda=\frac{NR-D}{D(N-1)}$}, respectively.
\end{proof}

\vspace{-\bigskipamount}

\begin{proof}[Proof of Lemma~\ref{lem.equivalence}.]
(i $\Leftrightarrow$ ii)
The projection onto the image of an isometry $\bfPhi$ is $\bfPhi\bfPhi^*$.
Thus, for any unitary \smash{$\bfU:\calE\rightarrow\widetilde{\calE}$},
each $\bfU\bfPhi_n^{}\bfPhi_n^{*}\bfU^*$ is the projection onto
\smash{$\im(\bfU\bfPhi_n)
=\bfU\bigparen{\im(\bfPhi_n)}
=\bfU(\calF_n)$}
while
\smash{$\widetilde{\bfPhi}_n^{}\widetilde{\bfPhi}_n^{*}$} is the projection onto \smash{$\widetilde{\calF}_n$}.
(ii $\Rightarrow$ iii)
For any such $\bfU$ and $n\in\calN$,
apply Lemma~\ref{lem.equal Gram} where
\smash{$\bfL=\bfPhi_n^{*}\bfU^*$}
and
\smash{$\widetilde{\bfL}=\widetilde{\bfPhi}_n^{*}$}.
(iii $\Rightarrow$ ii) Immediate.
(iii $\Rightarrow$ iv) Immediate.
(iv $\Rightarrow$ iii)
The fusion synthesis maps $\bfPhi$ and $\widetilde{\bfPhi}$ of $(\bfPhi_n)_{n\in\calN}$ and $(\widetilde{\bfPhi}_n)_{n\in\calN}$, respectively, satisfy
\smash{$\widetilde{\bfPhi}^{*}\widetilde{\bfPhi}=\widehat{\bfV}\bfPhi^*\bfPhi\widehat{\bfV}^*$}
where $\widehat{\bfV}$ is the block-diagonal matrix that has \smash{$(\bfV_n)_{n\in\calN}$} as its diagonal blocks.
Regarding the codomains of $\bfPhi$ and $\widetilde{\bfPhi}$ to be their images and applying Lemma~\ref{lem.equal Gram} with \smash{$\bfL=\bfPhi\widehat{\bfV}^*$} and \smash{$\widetilde{\bfL}=\widetilde{\bfPhi}$} thus gives a unitary \smash{$\bfU:\Span(\calF_n)_{n\in\calN}\rightarrow\Span(\widetilde{\calF}_n)_{n\in\calN}$} such that
\smash{$\widetilde{\bfPhi}=\bfU\bfPhi\widehat{\bfV}^*$},
namely such that \smash{$\widetilde{\bfPhi}_n^{}=\bfU\bfPhi_n^{}\bfV_n^*$} for any $n\in\calN$.
Since \smash{$\dim(\widetilde{\calE})=\dim(\calE)$} by assumption,
this $\bfU$ extends to a unitary from $\calE$ into $\widetilde{\calE}$.
\end{proof}

\vspace{-\bigskipamount}

\begin{proof}[Proof of Lemma~\ref{lem.autocor of mult char}.]
We consider a particular isomorphism from the additive group of $\bbF_Q$ onto its Pontryagin dual \smash{$\widehat{\bbF}_Q$}.
In particular, let $P$ be the (prime) characteristic of $\bbF_Q$,
let $J$ be the positive integer such that $Q=P^J$,
and consider the \textit{field trace} $\tr:\bbF_Q\rightarrow\bbF_P$,
\smash{$\tr(x)=\sum_{j=0}^{J-1}x^{P^j}$}.
It is well-known that $\tr$ is a nontrivial $\bbF_P$-linear functional.
For any $x\in\bbF_Q$, the linearity of $\tr$ implies that $\gamma_x:\bbF_Q\rightarrow\bbT$, \smash{$\gamma_x(y):=\exp\bigparen{\frac{2\pi\rmi}{P}\tr(xy)}$} is a well-defined additive character of $\bbF_Q$.
Moreover, the mapping $x\mapsto\gamma_x$ is clearly a homomorphism from $\bbF_Q$ into \smash{$\widehat{\bbF}_Q$}.
Its kernel consists of those $x\in\bbF_Q$ for which $\exp\bigparen{\frac{2\pi\rmi}{P}\tr(xy)}=1$ for any $y\in\bbF_Q$,
namely those $x\in\bbF_Q$ for which $\tr(xy)=0$ for any $y\in\bbF_Q$.
Since the trace is nontrivial, this implies that $x=0$ (since otherwise, take $y=x^{-1}x_0$ where $\tr(x_0)\neq0$).
Thus, $x\mapsto\gamma_x$ is an injective homomorphism from $\bbF_Q$ into the group $\widehat{\bbF}_Q$.
Since $\#(\widehat{\bbF}_Q)\leq\#(\bbF_Q)<\infty$,
this map is in fact an isomorphism onto \smash{$\widehat{\bbF}_Q$}.

Now let \smash{$\theta\in\widehat{\bbF}_Q^\times$} be a nonidentity multiplicative character and consider the Gauss sums of the form~\eqref{eq.Gauss sum}.
If $\gamma=1$, then~\eqref{eq.Gauss sum} vanishes since distinct (multiplicative) characters are orthogonal.
Meanwhile, if $\gamma\neq1$ then $\gamma=\gamma_x$ for some $x\neq0$,
implying $\gamma_x(y)=\gamma_1(xy)$ and so
\begin{equation*}
\ip{\bfgamma_x}{\bftheta}
=\sum_{y\in\bbF_Q^\times}\overline{\gamma_1(xy)}\theta(y)
=\sum_{z\in\bbF_Q^\times}\overline{\gamma_1(z)}\theta(x^{-1}z)
=\overline{\theta(x)}\ip{\bfgamma_1}{\bftheta}.
\end{equation*}
In particular, $\abs{\ip{\bfgamma_x}{\bftheta}}^2=\abs{\ip{\bfgamma_1}{\bftheta}}^2$ for any \smash{$x\in\bbF_Q^\times$}.
Moreover, Parseval--Plancherel gives
\begin{equation*}
Q(Q-1)
=Q\norm{\bftheta}^2
=\norm{\bfGamma^*\bftheta}^2
=\sum_{\gamma\in\widehat{\bbF}_Q}\abs{\ip{\bfgamma}{\bftheta}}^2
=\sum_{x\in\bbF_Q}\abs{\ip{\bfgamma_x}{\bftheta}}^2
=(Q-1)\abs{\ip{\bfgamma_1}{\bftheta}}^2.
\end{equation*}
Thus, $\abs{\ip{\bfgamma_x}{\bftheta}}^2=\abs{\ip{\bfgamma_1}{\bftheta}}^2=Q$ for any \smash{$x\in\bbF_Q^\times$}.
The discrete Fourier transform (with respect to the additive characters of $\bbF_Q$) of the autocorrelation of the extended multiplicative character $\bftheta$ thus satisfies
\begin{equation*}
\bigparen{\bfGamma^*(\bftheta*\bftheta^{\star})}(\gamma)
=\abs{(\bfGamma^*\bftheta)(\gamma)}^2
=\left\{\begin{array}{cl}
0,&\ \gamma=1\\
Q,&\ \gamma\neq1
\end{array}\right\}
=Q(\bfchi_{\widehat{\bbF}_Q}-\bfdelta_1)(\gamma)
=\bigparen{\bfGamma^*(Q\bfdelta_0-\bfchi_{\bbF_Q})}(\gamma)
\end{equation*}
for any $\gamma\in\widehat{\bbF}_Q$.
Thus, $\bftheta*\bftheta^{\star}=Q\bfdelta_0-\bfchi_{\bbF_Q}$ and so for any $x\in\bbF_Q$,
\begin{equation*}
\ip{\brT^x\bftheta}{\bftheta}
=\sum_{y\in\bbF_Q}\bftheta(y)\overline{\bftheta(y-x)}
=(\bftheta*\bftheta^{\star})(x)
=(Q\bfdelta_0-\bfchi_{\bbF_Q})(x)
=\left\{\begin{array}{rl}
Q-1,&\ x=0,\\
 -1,&\ x\neq 0.
\end{array}\right.\quad\qedhere
\end{equation*}
\end{proof}

\end{appendices}
\end{document}